\documentclass[10pt,reqno]{amsart}
\usepackage{amssymb,mathrsfs,graphicx}
\usepackage{ifthen}

\usepackage{hyperref}

\usepackage[margin=1in]{geometry}
\usepackage{caption}
\usepackage{sidecap}
\usepackage{rotating,dsfont}
\usepackage{enumitem}

\usepackage{soul}
\usepackage{cancel}
\usepackage[normalem]{ulem}

\usepackage{colortbl}
\definecolor{black}{rgb}{0.0, 0.0, 0.0}
\definecolor{red}{rgb}{1.0, 0.5, 0.5}
 \usepackage{xcolor}
\provideboolean{shownotes} 
\setboolean{shownotes}{true} 
\newcommand{\margnote}[1]{
\ifthenelse{\boolean{shownotes}}%
{\marginpar{\raggedright\tiny\texttt{#1}}}%
{}%
}
\newcommand{\hole}[1]{
\ifthenelse{\boolean{shownotes}}%
{\begin{center} \fbox{ \rule {.25cm}{0cm} \rule[-.1cm]{0cm}{.4cm}
\parbox{.85\textwidth}{\begin{center} \texttt{#1}\end{center}} \rule
{.25cm}{0cm}}\end{center}} {} }

\title[Relative entropy estimates for Vlasov--Fokker--Planck equations]{A unified relative entropy framework for macroscopic limits of Vlasov--Fokker--Planck equations}

\author[Choi]{Young-Pil Choi}
\address[Young-Pil Choi]{\newline Department of Mathematics\newline
Yonsei University, 50 Yonsei-Ro, Seodaemun-Gu, Seoul 03722, Republic of Korea}
\email{ypchoi@yonsei.ac.kr}

\author[Jung]{Jinwook Jung}
\address[Jinwook Jung]{\newline Department of Mathematics and  Research Institute for Natural Sciences \newline
Hanyang University, 222 Wangsimni-ro, Seongdong-gu, Seoul 04763, Republic of Korea}
\email{jinwookjung@hanyang.ac.kr}

\numberwithin{equation}{section}

\newtheorem{theorem}{Theorem}[section]
\newtheorem{lemma}{Lemma}[section]
\newtheorem{corollary}{Corollary}[section]
\newtheorem{proposition}{Proposition}[section]
\newtheorem{remark}{Remark}[section]
\newtheorem{definition}{Definition}[section]

\newcommand{\R}{\mathbb R}

\newcommand{\ls}{\lesssim}

\newcommand{\T}{\mathbb T}

\newcommand{\bq}{\begin{equation}}
\newcommand{\eq}{\end{equation}}
\newcommand{\e}{\varepsilon}
\newcommand{\lt}{\left}
\newcommand{\rt}{\right}

\newcommand{\pa}{\partial}

\newcommand{\mh}{\mathcal{H}}

\newcommand{\intr}{\int_{\R^d}}

\newcommand{\intrr}{\iint_{\R^d \times \R^d}}

\newcommand{\into}{\int_{\R^d}}

\newcommand{\calM}{\mathcal M}

\newcommand{\calP}{\mathcal P}

\makeatletter
\def\moverlay{\mathpalette\mov@rlay}
\def\mov@rlay#1#2{\leavevmode\vtop{%
   \baselineskip\z@skip \lineskiplimit-\maxdimen
   \ialign{\hfil$\m@th#1##$\hfil\cr#2\crcr}}}
\newcommand{\charfusion}[3][\mathord]{
    #1{\ifx#1\mathop\vphantom{#2}\fi
        \mathpalette\mov@rlay{#2\cr#3}
      }
    \ifx#1\mathop\expandafter\displaylimits\fi}
\makeatother

\begin{document}
\allowdisplaybreaks

\date{\today}


\begin{abstract}  
We develop a unified relative entropy framework for macroscopic limits of kinetic equations with Riesz-type interactions and Fokker--Planck relaxation. Our analysis covers three prototypical singular regimes: the diffusive limit leading to a drift-diffusion equation, the high-field limit yielding the aggregation equation in the repulsive regime, and the strong magnetic field limit producing a generalized surface quasi-geostrophic equation. The method combines entropy dissipation, Fisher-information control, and modulated interaction energies into a robust stability theory yielding both strong and weak convergence results. For the strong convergence, we establish quantitative relative entropy estimates toward macroscopic limits under well-prepared initial data, extending the scope of the method to settings where nonlocal forces and singular scalings play a decisive role.  For the weak convergence, our approach captures three complementary phenomena: in the diffusive regime, it yields sharper quantitative estimates in weak topologies consistent with the formally optimal scaling; in the high-field regime, it propagates bounded Lipschitz stability for a class of mildly prepared initial data, even when the relative entropy diverges with respect to the singular scaling parameter; and in the strong magnetic field regime, it provides quantitative weak estimates, including bounded Lipschitz control of the rescaled momentum and negative Sobolev control of the density. This broader perspective shows that relative entropy provides not only a tool for strong convergence, but also a mechanism for treating low-regularity and mildly prepared regimes. The analysis highlights the unifying role of relative entropy in connecting microscopic dissipation with both strong and weak macroscopic convergence.
\end{abstract}

\maketitle \centerline{\date}

\tableofcontents

%
%
%
%
\section{Introduction}

One of the central themes in kinetic theory is the derivation of macroscopic models from microscopic or mesoscopic descriptions. Such asymptotic analyses not only provide a rigorous justification of effective equations but also reveal the precise mechanisms by which microscopic dissipation, transport, and interaction forces translate into collective macroscopic dynamics. 
The hydrodynamic limit program, initiated in classical works such as \cite{BGL93, BGL91,Caf80,CIP94}, has played a fundamental role in connecting kinetic equations (Boltzmann, Vlasov, Fokker--Planck) with fluid models (Euler, Navier--Stokes, drift-diffusion). Beyond its mathematical significance, this line of research underpins multiscale modeling in physics, biology, and engineering, where effective macroscopic equations are indispensable for large-scale simulations and qualitative understanding. 

A key challenge in this program lies in handling singular scalings, nonlocal interactions, or strong external fields. In these regimes, the usual compactness or moment methods do not provide sufficient control, and more refined tools such as entropy methods, Wasserstein stability estimates, or modulated energies have been developed, e.g. \cite{BIR25, BV05,  CC20, CC21, CCJ21, CCT19, C21, CJe21-2, CJ24, CT22, FK19, GSR99, Han11, SR04, Ser20, Vas08}. Establishing quantitative convergence results in such singular limits is particularly important: they not only provide error bounds for the approximation of kinetic dynamics by macroscopic models but also clarify the regimes of validity and the robustness of effective descriptions in low-regularity or partially prepared settings.

In this paper, we are interested in asymptotic analyses for Vlasov--Fokker--Planck equations, which describe the time evolution of a particle density $f=f(t,x,\xi)$ in phase space:
\bq\label{main1}
{\rm A}\pa_t f + {\rm B} \xi\cdot \nabla_x f  - \nabla_x (-\Delta_x)^{-\alpha} \rho_f \cdot \nabla_\xi f = \frac1\tau \nabla_\xi \cdot (\theta \nabla_\xi f + \xi   f), \quad t>0, \ (x,\xi) \in \R^d \times \R^d.
\eq
Here $d \ge 1$ denotes the spatial dimension, and $f$ represents the probability density (or distribution function) of particles at position $x \in \R^d$ with velocity $\xi \in \R^d$ at time $t$. The macroscopic density and momentum are given by
\[
\rho_f := \intr f\,d\xi, \quad \rho_f u_f := \intr \xi f\,d\xi.
\]
The parameter $\tau>0$ denotes the relaxation time, i.e. the characteristic scale at which the velocity distribution relaxes toward a global Maxwellian under the effect of the Fokker--Planck operator. Small values of $\tau$ correspond to fast equilibration in velocity space, while larger values emphasize the role of transport and long-range interactions.

The equation \eqref{main1} involves several parameters whose roles determine the asymptotic regimes and limiting dynamics. The coefficients ${\rm A}$ and ${\rm B}$ act as scaling factors for the time derivative and the transport term $\xi\cdot\nabla_x f$, respectively. 
Different choices of these parameters, together with the relaxation time $\tau$, select distinct asymptotic regimes, which we summarize below. The constant $\theta>0$ measures the strength of velocity diffusion and can be interpreted as the (scaled) temperature of the system. 
It balances deterministic drift and stochastic fluctuations in velocity space, and together with $\tau$ governs the relaxation toward Maxwellian equilibria.   Finally, the exponent $\alpha\in(0,1]$ specifies the type of interaction kernel: $\alpha=1$ corresponds to the Coulomb potential, while $\alpha\in(0,1)$ corresponds to Riesz-type interactions, which are longer ranged than the Coulomb case and lead to stronger nonlocal effects.

Our main results can be summarized in three representative asymptotic regimes:
\begin{itemize}
\item {\bf Diffusive limit.} Under the scaling ${\rm A}=\e$, ${\rm B}=1$, $\tau=\e$, weak entropy solutions converge to the aggregation--diffusion (McKean--Vlasov) equation with repulsive Riesz interaction. Quantitative stability estimates are obtained, showing convergence in both density and momentum variables.

\item {\bf High-field limit.} For the scaling ${\rm A}={\rm B}=\e$, $\tau=1$, solutions relax toward local Maxwellians whose dynamics are governed by a drift-driven aggregation equation. The kinetic transport is suppressed, and the limiting equation is purely nonlocal.

\item {\bf gSQG limit.} In the strong magnetic field regime ${\rm A}=\e$, ${\rm B}=\tau=1$, the kinetic equation contains an additional fast rotational transport term of the form 
\[
\frac1\e \xi^\perp\cdot\nabla_\xi f,
\]
which represents the cyclotronic motion generated by a constant external magnetic field.  This gyroscopic force induces rapid angular averaging in velocity space and effectively constrains the dynamics to two dimensions.  In combination with the Riesz interaction, the resulting macroscopic evolution is a generalized surface quasi-geostrophic (gSQG) equation with active scalar velocity $\nabla_x^\perp(-\Delta_x)^{-\alpha}\rho$.
\end{itemize}

Beyond these specific regimes, our approach relies on a unified relative entropy framework that applies to all three limits. This method provides not only qualitative convergence but also quantitative bounds. In particular, we establish stability estimates in strong topologies and, in appropriate regimes, obtain additional quantitative convergence results in weak topologies, including bounded Lipschitz distances and negative Sobolev estimates.
 
%
%
%
%
%

\subsection{Functional frameworks}

Here, we gather the main notions and tools that will be used throughout the asymptotic analysis. 

Two complementary classes of functionals will play a central role in our arguments:
\begin{itemize}
\item {\bf Energy and entropy-based functionals}, including the Boltzmann entropy, the free energy, the interaction energy, and the modulated energy, which provide quantitative control of the dynamics and encode the dissipative structure of the kinetic system;
\item {\bf Distance functionals and weak topologies}, including the bounded Lipschitz (BL) distance and negative Sobolev norms, which allow us to quantify convergence in weaker topologies and to obtain stability estimates adapted to different singular regimes.
\end{itemize}
Together, these ingredients form the functional framework underlying the statements and proofs of our main theorems.

We begin by introducing the local Maxwellian distribution
\[
M_{\rho, u,\theta}(\xi) := \frac{\rho}{(2\pi \theta)^{\frac d2}}\exp\lt( -\frac{|u-\xi|^2}{2\theta} \rt),
\]
where $\rho > 0$ and $u \in \R^d$. Since the temperature $\theta$ is a fixed constant, we set $\theta=1$ for notational simplicity and write $M_{\rho,u} := M_{\rho,u,1}$. Without loss of generality, we also assume that all distributions $f$ under consideration are normalized to have unit mass:
\[
\iint_{\R^d \times \R^d} f(t,x,\xi)\,dxd\xi = 1, \quad t \geq 0.
\]
This convention is consistent with the conservation of total mass under the kinetic dynamics and will simplify several entropy-related expressions in what follows.

%
%
%
%
%

\subsubsection{Relative entropy}
For a given macroscopic state $(\rho,u)$, the {\it relative entropy} of a distribution $f$ with respect to $M_{\rho,u}$ is defined by
\begin{equation}\label{def:rel-entropy}
\mathscr{H}[f | M_{\rho,u}] := \iint_{\R^d \times \R^d} f \log \frac{f}{M_{\rho,u}}  \, dx d\xi - \iint_{\R^d \times \R^d}  \lt(f - M_{\rho,u} \rt) dx d\xi.
\end{equation}
It quantifies the deviation of $f$ from the local equilibrium $M_{\rho,u}$. The functional \eqref{def:rel-entropy} is nonnegative, convex in $f$, and vanishes if and only if $f=M_{\rho,u}$. Moreover, it is closely related to the variational characterization of Maxwellians under prescribed mass and momentum constraints; see Section \ref{sec_pre} below.

In the present work, relative entropy serves as the basic microscopic stability functional. Combined with the interaction part of the modulated energy, it allows us to measure simultaneously the deviation from local Maxwellians and the mismatch of the associated density variable. This two-level structure is one of the key ingredients of our unified approach to the three singular limits considered in this paper. More specifically, relative entropy admits a natural decomposition into a microscopic part, measuring the deviation from the Maxwellian with the same local mass and momentum as $f$, and a macroscopic part, measuring the mismatch between these moments and the target state $(\rho,u)$. This structure will be made precise in Section \ref{sec_pre}; it is one of the key mechanisms that allows the relative entropy method to connect microscopic dissipation with macroscopic stability.
 
%
%
%
%
%
 
\subsubsection{Energy and modulated functionals}
A natural collection of energy functionals governs the kinetic dynamics and provides quantitative control of solutions. 
For a distribution $f=f(x,\xi)$, we begin with the {\it free energy} (also referred to as the {\it kinetic entropy}), defined as
\[
\mathscr{F}[f]  := \iint_{\R^d \times \R^d} \lt(\frac{|\xi|^2}2 + \log f\rt) f \,dxd\xi,
\]
which naturally decomposes into the {\it kinetic energy} 
\[
\mathscr{K}[f] := \iint_{\R^d\times\R^d} \frac{|\xi|^2}{2} f\,dxd\xi,
\]
and the {\it Boltzmann entropy}
\[
\mathscr{H}[f] := \iint_{\R^d\times\R^d} f\log f\,dxd\xi.
\]
Note that the free energy can be rewritten as 
\[
\mathscr{F}[f] =   \iint_{\R^d \times \R^d} f\log \frac{f}{M_{1,0}}\,dxd\xi - \frac d2 \log (2\pi).
\]

In addition, for the associated density $\rho_f$, we define the {\it interaction energy} (or {\it potential energy})
\[
\mathscr{P}[\rho_f] := \frac{1}{2} \into \rho_f (-\Delta_x)^{-\alpha}\rho_f\,dx,
\]
which accounts for the nonlocal forces induced by the Riesz kernel.  The operator $(-\Delta_x)^{-\alpha}$ can be expressed in terms of convolution with the Riesz potential 
\[
K(x) := \frac{c_{\alpha,d}}{|x|^{(d-2\alpha)}}, \quad c_{\alpha,d}>0, \ \alpha \in \lt(0, \min\lt\{1,\frac d2\rt\}\rt)
\]
so that
\[
(-\Delta_x)^{-\alpha}\rho = K \star \rho = c_{\alpha,d}\intr \frac{1}{|x-y|^{d-2\alpha}}\,\rho(y)\,dy.
\]
Accordingly, the interaction energy admits the representation
\[
\mathscr{P}[\rho] = \frac{c_{\alpha,d}}2\iint_{\R^d \times \R^d} \frac1{|x-y|^{d-2\alpha}} \rho(x)\rho(y)\,dxdy.
\]
Moreover, for $\alpha \in (0,1)$, one may also use the extension characterization of the fractional Laplacian \cite{CS07}. 
Introducing the extended kernel
\[
K(x,\tilde x) := \frac{c_{\alpha, d}}{|(x,\tilde x)|^{d-2\alpha}},
\]
and regarding $\rho$ as a distribution on $\R^d\times\R_+$ via $\rho\otimes\delta_0$, we have
\[
\intr K((x,\tilde x) - (y,0))\rho(y)\,dy =: (K \star (\rho \otimes \delta_0))(x,\tilde x).
\]
Then $K$ satisfies, in the distributional sense,
\[
- \nabla_{(x,\tilde x)} \cdot \lt(\tilde{x}^{1-2\alpha} \nabla_{(x,\tilde x)} K \rt) = \delta_0 \quad \mbox{on } \R^d \times \R_+.
\]
so that
\[
- \nabla_{(x,\tilde x)} \cdot \lt(\tilde{x}^{1-2\alpha} \nabla_{(x,\tilde x)} K \star (\rho \otimes \delta_0) \rt) = \rho(x) \otimes \delta_0(\tilde x).
\]
This yields yet another representation of the interaction energy:
\[
\mathscr{P}[\rho] = \iint_{\R^d \times \R_+} \tilde x^{1-2\alpha} |\nabla_{(x,\tilde x)} K \star (\rho \otimes \delta_0)(x,\tilde x) |^2\,dx d\tilde x.
\]
These alternative formulations are useful both conceptually and technically since they connect the nonlocal interaction to classical potential theory and weighted local elliptic problems. 
We refer to \cite{CJe21-2, CJ24, PS17, Ser20} and references therein for further developments on modulated energy estimates involving Riesz-type interactions.

The free energy $\mathscr{F}[f]$ and the interaction energy $\mathscr{P}[\rho_f]$ together yield the {\it total energy}
\[
\mathscr{F}[f] + \mathscr{P}[\rho_f],
\]
which dissipates along kinetic trajectories and provides a fundamental Lyapunov functional for the system. While $\mathscr{F}[f]$ and $\mathscr{P}[\rho]$ provide absolute energy bounds, they do not directly measure the proximity of a kinetic state to a prescribed macroscopic pair $(\rho,u)$. 
To capture both the microscopic relaxation towards local equilibria and the macroscopic convergence of densities, we introduce the {\it modulated energy functional}:
\begin{equation}\label{def:mod-energy}
\mathscr{E}[f | \rho, u] := \mathscr{H}[f | M_{\rho,u}] + \mathscr{P}[\rho_{f} | \rho],
\end{equation}
where the second term denotes the {\it relative potential energy}
\[
\mathscr{P}[\rho_{f} | \rho] = \frac12 \intr (\rho_{f}-\rho)(-\Delta_x)^{-\alpha}(\rho_{f}-\rho)\,dx = \frac12 \|\rho_{f}-\rho\|_{\dot H^{-\alpha}(\R^d)}^2.
\]

Here the first term in \eqref{def:mod-energy} quantifies the deviation of $f$ from the local Maxwellian $M_{\rho,u}$, while the second measures the distance of the corresponding densities. 

\begin{remark}Although we consider the spatial domain $\R^d$ for simplicity, all definitions and constructions extend naturally to the periodic setting $\T^d$. In that case, the interaction potential is defined up to an additive constant, and the normalization
\[
F[f] = - \nabla_x (-\Delta_x)^{-\alpha}(\rho_f - 1), \quad \int_{\T^d} \rho_f\,dx = 1,
\]
ensures that the potential energy remains well defined and the total interaction force vanishes. Moreover, since the relative potential energy only depends on the difference $\rho_f - \rho$, the same expression
\[
\mathscr{P}[\rho_{f} | \rho] = \frac12 \int_{\T^d} (\rho_{f}-\rho) (-\Delta_x)^{-\alpha}(\rho_{f}-\rho)\,dx = \frac12 \|\rho_{f}-\rho\|_{\dot H^{-\alpha}(\T^d)}^2
\]
holds on $\T^d$ without any modification, provided that $\int_{\T^d} \rho\,dx = 1$. 

Consequently, all subsequent modulated energy and relative entropy estimates remain valid in both domains with only mild adjustments to the regularity assumptions for the macroscopic limit equations.
\end{remark}

The functional $\mathscr{E}[f | \rho, u]$ thus provides a robust stability measure, simultaneously at the kinetic and macroscopic levels. 
It plays a central role in the relative entropy method: although $\mathscr{E}$ is not, in general, dissipative, it admits a Gr\"onwall-type control (with remainders depending on the scaling parameter), and when combined with the monotone decay of the absolute free energy $\mathscr{F}[f]+\mathscr{P}[\rho]$, it yields strong, quantitative convergence estimates toward the limiting macroscopic dynamics. 
In particular, the modulated energy is the key ingredient in defining the notion of {\it weak entropy solutions}, introduced in the next subsection. 
The structure of $\mathscr{F}[f]$ and $\mathscr{P}[\rho_f]$ outlined here will serve as the foundation for that definition, ensuring that the weak formulation inherits the essential energy-dissipative features of the kinetic model.

%
%
%
%
%
\subsubsection{Bounded Lipschitz distance}
Besides entropy-based functionals, we also employ weaker topologies to quantify convergence of distributions. Depending on the scaling regime, these are measured either by bounded Lipschitz distances or by negative Sobolev norms. In particular, while negative Sobolev spaces are useful for capturing quantitative convergence in the diffusive regime, the bounded Lipschitz distance is especially well-suited for weak stability estimates at the level of measures.

For a metric space $(X,{\rm d}_X)$ and finite Borel measures $\mu,\nu$ on $X$, the {\it BL distance} is defined by
\[ 
{\rm d}_{\rm BL}^X(\mu,\nu)
:= \sup_{\|\phi\|_{L^\infty(X)} + {\rm Lip}_X(\phi)\le 1}
\lt| \int_X \phi \, d\mu - \int_X \phi \, d\nu \rt|.
\] 
This metric induces the topology of weak (narrow) convergence of measures. 
In our setting, $X$ will typically be the spatial domain $\R^d$, or the phase space $\R^d\times\R^d$. 
To simplify notation, we shall often write ${\rm d}_{\rm BL}$ without a superscript when the underlying domain is clear from the context.

\medskip
\noindent{\it Time--space variant.}
For measure-valued curves $\{\mu_t\},\{\nu_t\}$ on $\R^d$, we also use the BL distance on the cylinder $(0,T)\times \R^d$, defined by
\[
{\rm d}_{\rm BL}^T(\mu,\nu)
:= \sup_{\|\phi\|_{W^{1,\infty}((0,T)\times \R^d)} \le 1}
\lt| \int_0^T \intr \phi(t,x)\, d\mu_t(x)dt - \int_0^T \intr \phi(t,x)\, d\nu_t(x)dt \rt|.
\]
This provides a convenient way to measure weak convergence in time-space without requiring uniform integrability of higher moments. On bounded domains, ${\rm d}_{\rm BL}$ is equivalent to the first-order Wasserstein distance, up to constants depending only on the diameter of the domain. Hence, on a bounded domain either metric can be used interchangeably to capture weak convergence. 
On unbounded domains such as $\R^d$, however, ${\rm d}_{\rm BL}$ retains finiteness without imposing finite first-moment assumptions, which makes it particularly suitable for our analysis.
 
Beyond these structural advantages, the BL distance plays a crucial role in our quantitative estimates. 
Through Lemma \ref{lem_gd2} below, the dissipation terms appearing in the relative entropy inequality can be converted into control of time-integrated BL distances between $f^\e$ and the corresponding local Maxwellians.  This mechanism allows us to complement the relative entropy framework with quantitative weak stability bounds.  It is particularly useful for the high-field limit, where weak convergence estimates are established under a mildly well-prepared assumption on the initial data.
 
 Before closing this part, we introduce several simplified notations that will be used throughout this paper. For functions $f(x,v)$ and $g(x)$, when there is no risk of confusion, we will use  $\|f\|_{L^p}$ and $\|g\|_{L^p}$ to denote the usual $L^p(\R^d \times \R^d)$-norm and $L^p(\R^d)$-norm, respectively.   We denote by $\mathcal{M}(\R^n)$ the space of finite signed Radon measures on $\R^n$, 
and by $\calP_p(\R^n)$, $p \in [1,\infty]$, the space of probability measures on $\R^n$ with finite $p$-th moment.  We also use a velocity-weighted $L^1$-norm defined by 
 \[
 \|f\|_{L^1_2} := \iint_{\R^d \times \R^d} (1+|\xi|^2)f\,dxd\xi. 
 \]
 Throughout the paper, we use $C>0$ to denote a generic constant which may change from line to line. We also write $f \ls g$ if there exists a constant $C>0$ such that $f \leq Cg$. Finally, we omit the $x$-dependence of differential operators for simplicity, writing, for instance, $\nabla f = \nabla_x f$ and $\Delta f = \Delta_x f$.

%
%
%
%
%
\subsection{Main results}

Our analysis is based on the unified framework of {\it weak entropy solutions} to the kinetic equation \eqref{main1}. This notion provides a natural relative entropy framework, capturing precisely the structural properties of kinetic solutions that allow one to exploit the entropy dissipation and to carry out the compactness arguments required for the singular limit analysis (see, e.g. \cite{GM10, GSR99, Gou05, NPS01, PS00}). In particular, it gives a consistent setting in which one can rigorously connect the kinetic dynamics to the macroscopic models derived under different scaling regimes.

\begin{definition}\label{def_weak}Let $T>0$. We say that $f$ is a weak entropy solution to \eqref{main1} with initial data $f_0$ if the following conditions hold:
\begin{enumerate}[label=(\roman*)]
\item $f \in L^\infty([0,T]; L^1(\R^d \times \R^d, (1+ |\xi|^2)\,dxd\xi) \cap L \log L(\R^d \times \R^d))$,
\item $f$ satisfies \eqref{main1} in the sense of distributions,
\item the total energy inequality holds for almost every $t \in [0,T]$:
\begin{align*}
& \mathscr{F}[f(t)] + \frac1{\rm B}\mathscr{P}[\rho_{f}(t)] +  \frac1{\tau {\rm A}} \int_0^t \iint_{\R^d \times \R^d} \frac1{f}|\nabla_\xi f + \xi  f|^2 \,dxd\xi ds  \leq \mathscr{F}[f_0] +\frac1{\rm B} \mathscr{P}[\rho_{f_0}].
\end{align*}
\end{enumerate}
\end{definition}

With this framework in place, we now turn to the analysis of singular limits under different scaling regimes. 
In the following subsections, we present our main results concerning three distinct limits: 
the {\it diffusive limit}, the {\it high-field limit}, and the {\it generalized surface quasi-geostrophic (gSQG) limit}. 
Each of these regimes highlights a different balance between kinetic transport, nonlocal interaction, and Fokker--Planck relaxation, and the convergence analysis relies on the weak entropy solution framework introduced above.

%
%
%
%
%
\subsubsection{Diffusive limit}\label{sssec_diff} We begin by considering the diffusive scaling 
\[
{\rm A}=\e, \quad {\rm B}=1, \quad \tau=\e.
\]
Such a choice of parameters naturally arises from dimensionless scaling arguments in kinetic theory (see, e.g., \cite{ACGS01, BG08, PS00, WLL15}), and this regime is also referred to in the literature as the {\it low-field limit}. Under this regime, the kinetic equation \eqref{main1} reduces to 
\bq\label{eq_diff_sca}
\e \pa_t f^\e +  \xi\cdot \nabla f^\e  - \nabla (-\Delta)^{-\alpha} \rho_{f^\e} \cdot \nabla_\xi f^\e = \frac1\e\nabla_\xi \cdot (\nabla_\xi f^\e + \xi   f^\e).
\eq
This scaling reflects the regime where the relaxation mechanism dominates, and the system exhibits a drift-diffusion behavior at the macroscopic level. 
  We now outline the formal derivation of the limiting equation as $\e \to 0$. By integrating against the velocity moments $1$ and $\xi$, we readily obtain
\begin{align*}
&\e \pa_t \rho_{f^\e} + \nabla \cdot m_{f^\e} = 0,\cr
&\e \pa_t m_{f^\e} + \nabla \cdot \intr \xi \otimes \xi f^\e\,d\xi = - \rho_{f^\e} \nabla (-\Delta)^{-\alpha} \rho_{f^\e} - \frac1\e m_{f^\e},
\end{align*}
where the local momentum is defined as
\[
m_f := \rho_f u_f = \intr \xi f\,d\xi.
\]
The strong friction term $-\frac1\e m_{f^\e}$ in the momentum equation indicates rapid relaxation of the momentum variable, so that in the small-$\e$ regime we expect $m_{f^\e}$ to be determined by the density $\rho_{f^\e}$. Replacing the momentum variable $m_{f^\e}$ by its limiting expression in terms of the density $\rho_{f^\e}$ leads formally to
\bq\label{eq_fd}
\pa_t \rho_{f^\e} - \nabla \cdot (\rho_{f^\e} \nabla (-\Delta)^{-\alpha}\rho_{f^\e}) =  \nabla\otimes \nabla : \intr \xi \otimes \xi f^\e\,d\xi + \e^2 \nabla \cdot \pa_t m_{f^\e}.
\eq
To identify the limit of the second-order moment, we observe that the right-hand side of \eqref{eq_diff_sca} can be rewritten as
\[
\nabla_\xi \cdot (\nabla_\xi f^\e + \xi   f^\e) = \nabla_\xi \cdot \lt(f^\e \nabla_\xi \log \lt(\frac{f^\e}{M_{\rho_{f^\e},0}} \rt)\rt),
\]
which highlights the entropy-dissipating structure of the linear Fokker--Planck operator. This form reveals that, for $\e \ll 1$, the solution $f^\e$ should be close to the local Maxwellian $M_{\rho_{f^\e},0}$. Consequently, 
\[
\nabla\otimes \nabla : \intr \xi \otimes \xi f^\e\,d\xi \sim \nabla\otimes \nabla : \intr \xi \otimes \xi M_{\rho_{f^\e}, 0}(\xi)\,d\xi = \Delta \rho_{f^\e}.
\]
By substituting this asymptotic relation into \eqref{eq_fd} and neglecting higher-order terms, we obtain the limiting macroscopic equation
\bq\label{eq_agdi}
\partial_t \rho - \nabla \cdot (\rho \nabla (-\Delta)^{-\alpha} \rho) = \Delta \rho.
\eq
This equation is commonly referred to as the {\it aggregation--diffusion equation}, or the {\it McKean--Vlasov equation}, here considered in the repulsive interaction regime. It can be interpreted as the Wasserstein gradient flow of the macroscopic free energy functional
\[
\mathscr{F}_{\rm AD}[\rho]:=\into \rho \log \rho\,dx + \frac{1}{2} \into \rho (-\Delta)^{-\alpha}\rho\,dx.
\]
The entropic term $\into \rho \log \rho\,dx$ gives rise to the linear diffusion, while the interaction term $\mathscr{P}[\rho]$ induces additional repulsive effects through the Riesz kernel. Thus, the diffusive limit reflects a balance between entropy-driven spreading and nonlocal repulsion, and it naturally fits into the general theory of Wasserstein gradient flows \cite{AGS08}.
 
This formal derivation indicates that, in the limit $\e \to 0$, the macroscopic density $\rho$ evolves according to the aggregation--diffusion equation \eqref{eq_agdi}. We now make this rigorous. More precisely, under suitable assumptions on the initial data, weak entropy solutions of the kinetic equation \eqref{eq_diff_sca} converge to the regular solution $\rho$ of \eqref{eq_agdi}.

Our first result provides quantitative convergence estimates in the natural modulated energy topology, yielding strong control of both the kinetic density and the momentum variable. In addition, by exploiting the parabolic structure of the limiting equation, we also derive a sharper estimate for the density in a weaker topology of negative Sobolev type. This low-regularity estimate is particularly useful in the diffusive regime: although the topology is weaker, it is well adapted to the kinetic-to-macroscopic passage and captures the optimal-order convergence rate predicted by the formal expansion.

\begin{theorem}[Diffusive limit]\label{thm_kin1} Let $T>0$, and let $\rho \in L^\infty([0,T];\calP(\R^d))$ be the unique classical solution to \eqref{eq_agdi} satisfying
\[
 \nabla (-\Delta)^{-\alpha} \rho + \nabla \log \rho \in L^\infty(0,T; W^{1,\infty}(\R^d)) \cap W^{1,\infty}(0,T; L^\infty(\R^d)).
\]
Let $\{f^\e\}_{\e > 0}$ be a family of weak entropy solutions to the equation \eqref{eq_diff_sca} on the time interval $[0,T]$, with initial data $\{f^\e_0\}_{\e > 0}$ satisfying
\[
\sup_{0<\e <1}\e^2 \lt(\mathscr{K}[f^\e_0]  + \mathscr{P}[\rho_{f^\e_0}]\rt) < \infty.
\]

If the initial data is  well-prepared  in the sense that
\[
  \mathscr{H}[f^\e | M_{\rho, 0}](0) + \mathscr{P}[\rho_{f^\e} | \rho](0) \to 0 \quad \mbox{as } \e \to 0,
\]
then we have
\begin{align*}
f^\e &\to M_{\rho, 0} \quad \mbox{in } L^\infty(0,T; L^1(\R^d \times \R^d)), \cr
  \intr f^\e\,d\xi &\to \rho \quad \mbox{in } L^\infty(0,T; L^1 \cap \dot{H}^{-\alpha}(\R^d)),\cr
 \frac1\e \intr \xi f^\e\,d\xi &\to -\rho\lt(\nabla (-\Delta)^{-\alpha} \rho + \nabla \log \rho \rt) \quad \mbox{in } L^1((0,T) \times \R^d).
\end{align*}
In addition, the following stability estimate holds:
\begin{align*}
&\|f^\e - M_{\rho, 0}\|_{L^\infty(0,T;L^1)}^2 + \|\rho_{f^\e} - \rho\|_{L^\infty(0,T;L^1\cap \dot{H}^{-\alpha})}^2 + \|\rho_{f^\e} \frac{u_{f^\e}}{\e} + \rho(\nabla (-\Delta)^{-\alpha} \rho + \nabla \log \rho)\|_{L^1((0,T) \times \R^d)}^2\cr
 &\quad \leq C\lt( \mathscr{H}[f^\e | M_{\rho, 0}](0) + \mathscr{P}[\rho_{f^\e} | \rho](0)\rt) + C\e^2.
\end{align*}

If we further assume that $\rho \in L^\infty(0,T;H^{\beta+1}(\R^d))$ for some $\beta>\frac d2+1$, then the density error also satisfies the following quantitative estimate in the negative Sobolev space $H^{-\beta}(\R^d)$:
\[
\|\rho_{f^\e} -\rho\|_{L^2(0,T;H^{-\beta})} \le C\lt(\|\rho_{f_0^\e} -\rho_0\|_{H^{-\beta}}+ \mathscr{H}[f^\e | M_{\rho, 0}](0)  + \mathscr{P}[\rho_{f^\e} | \rho](0) + \e^2\rt).
\]

Here, the constant $C>0$ is independent of $\e>0$, depending only on the initial data, $T$, and the limiting solution $\rho$.
\end{theorem}

\begin{remark} 
Formal computations in Appendix \ref{app_optimal_diff} show that 
\[
f^\e = M_{\rho,0} + \e f_1 + O(\e^2) \quad \text{with } \intr f_1\,d\xi = 0.
\]
Accordingly, one expects the convergence rates
\[
f^\e - M_{\rho,0} = O(\e), \quad \rho_{f^\e} - \rho = O(\e^2), \quad \text{and} \quad  \rho_{f^\e} \frac{u_{f^\e}}{\e} + \rho(\nabla (-\Delta)^{-\alpha} \rho + \nabla \log \rho) = O(\e)
\]
to be essentially optimal.

Moreover, introducing the effective velocity
\[
{\rm u}_\e:=-\e (\nabla(-\Delta)^{-\alpha}\rho+\nabla\log\rho),
\]
the first-order corrector is exactly captured by the linearization of the local Maxwellian $M_{\rho,{\rm u}_\e}$ around $M_{\rho,0}$. Indeed, one formally obtains the improved approximation
\[
f^\e-M_{\rho,{\rm u}_\e}=O(\e^2).
\]
Thus, while $M_{\rho,0}$ approximates the microscopic distribution only at order $O(\e)$, the shifted Maxwellian $M_{\rho,{\rm u}_\e}$ incorporates the leading drift correction and provides a second-order approximation. This refined convergence can also be justified rigorously in a weaker negative Sobolev topology, see Remark \ref{rmk_opt_diff} for more details below.

In particular, the estimate in the weaker topology of negative Sobolev spaces is consistent with the diffusive scaling and provides a natural framework in which this optimal-order behavior can be quantified.
 \end{remark}

 \begin{remark}
In the case of the Coulomb potential, i.e. $\alpha = 1$, the global-in-time existence theory for the limit equation \eqref{eq_diff_sca} is well established. We refer to \cite{Bou93, CS95, Deg86} for the global-in-time existence of weak and strong solutions, and to \cite{GM10} for the construction of global renormalized solutions. When $\alpha = \tfrac12$, corresponding to the so-called Manev potential, the global-in-time existence of weak solutions has been obtained in \cite{CJe23}. 

The proof strategy used for the Manev potential in \cite{CJe23} can in fact be extended to cover the whole range $\alpha \in [\tfrac12,1)$, although rigorous results in this direction are not yet available in the literature. By contrast, for $\alpha \in (0,\tfrac12)$, the global-in-time existence of renormalized or weak solutions, to the best of our knowledge, remains open and appears to be highly challenging. However, for small and smooth initial data, global solutions can be constructed; see, for example, \cite{Cha23, CJe24, CJK24}. For the existence of solutions to the aggregation--diffusion equation \eqref{eq_agdi} with the regularity required in Theorem \ref{thm_kin1}, we provide a detailed discussion in Appendix \ref{app_reg}.
\end{remark}

%
%
%
%
%
\subsubsection{High-field limit} We next turn to the high-field scaling, corresponding to
\[
{\rm A}={\rm B}=\e  \quad \mbox{and} \quad \tau=1.
\]
This scaling also stems from standard dimensionless analysis of the kinetic equation, corresponding to the so-called high-field regime (see, e.g., \cite{ACGS01, GNPS05, NPS01}). Under this scaling, the kinetic equation \eqref{main1} becomes
\bq\label{eq_high_sca}
\e \pa_t f^\e + \e \xi\cdot \nabla f^\e  - \nabla (-\Delta)^{-\alpha} \rho_{f^\e} \cdot \nabla_\xi f^\e = \nabla_\xi \cdot (\nabla_\xi f^\e + \xi   f^\e).
\eq
This regime corresponds to a situation in which the interaction force and the Fokker--Planck relaxation determine the leading-order behavior, while the kinetic transport terms are suppressed by the small parameter $\e$. More precisely, \eqref{eq_high_sca} can be rewritten as
\[
\e \pa_t f^\e + \e \xi\cdot \nabla f^\e  = \nabla_\xi \cdot \lt(f^\e \nabla_\xi \log \lt(\frac{f^\e}{M_{\rho_{f^\e},{\rm u}^\e}} \rt)\rt),
\]
where 
\[ 
{\rm u}^\e := - \nabla (-\Delta)^{-\alpha} \rho_{f^\e}.
\] 
This formulation highlights that, in the limit $\e \to 0$, the right-hand side enforces a rapid relaxation of $f^\e$ towards the local Maxwellian $M_{\rho_{f^\e},{\rm u}^\e}$. In other words, $f^\e$ is expected to remain close to this local equilibrium, while the residual contributions from transport are negligible at leading order.
 
 To outline the formal derivation of the high-field limit, we proceed as in the diffusive regime by investigating the dynamics of the velocity moments. Multiplying \eqref{eq_high_sca} by $1$ and $\xi$, and integrating over the velocity variable, we obtain the continuity and momentum balance equations:
\begin{align*}
&\pa_t \rho_{f^\e} + \nabla \cdot m_{f^\e} = 0,\cr
&\e \pa_t m_{f^\e} + \e \nabla \cdot \intr \xi \otimes \xi f^\e\,d\xi = - \rho_{f^\e} \nabla (-\Delta)^{-\alpha} \rho_{f^\e} -  m_{f^\e}.
\end{align*}
Formally eliminating the momentum variable, we arrive at
\[
\pa_t \rho_{f^\e} - \nabla \cdot (\rho_{f^\e} \nabla (-\Delta)^{-\alpha}\rho_{f^\e}) =  \e \nabla\otimes \nabla : \intr \xi \otimes \xi f^\e\,d\xi + \e \nabla \cdot \pa_t m_{f^\e}.
\]
Passing to the limit $\e \to 0$, at the formal level, the remainder terms vanish, and we obtain the limiting equation
\bq\label{eq_ag}
\partial_t \rho -  \nabla \cdot (\rho \nabla (-\Delta)^{-\alpha} \rho) =0.
\eq
This equation corresponds to the well-known {\it aggregation equation}, here considered in the repulsive interaction regime. It preserves mass and exhibits a gradient-flow structure with respect to the interaction energy $\mathscr{P}[\rho]$ in the 2-Wasserstein metric. In this setting, the entropic contribution is absent, and the dynamics are driven solely by the nonlocal repulsive effects induced by the Riesz kernel. This interpretation parallels that of the diffusive limit, where the governing free energy functional $\mathscr{F}_{\rm AD}[\rho]$ combines entropy and interaction terms, whereas the high-field regime corresponds to the purely interaction-driven flow associated with $\mathscr{P}[\rho]$.

We next provide a rigorous justification of this limit. The following theorem establishes the convergence of weak entropy solutions of \eqref{eq_high_sca} towards solutions of the aggregation equation \eqref{eq_ag} under suitable assumptions on the initial data.

\begin{theorem}[High-field limit]\label{thm_kin2}
 Let $T>0$, and let $\rho \in L^\infty([0,T];\calP(\R^d))$ be the unique classical solution to \eqref{eq_ag} satisfying
\[
\nabla (-\Delta)^{-\alpha} \rho \in L^\infty(0,T; W^{1,\infty}(\R^d)) \cap W^{1,\infty}(0,T; L^\infty(\R^d)) \quad \mbox{and} \quad \nabla \log \rho \in L^\infty((0,T) \times \R^d).
\]
Let $\{f^\e\}_{\e > 0}$ be a family of weak entropy solutions to the equation \eqref{eq_high_sca} on the time interval $[0,T]$, with initial data   $\{f^\e_0\}_{\e > 0}$ satisfying
\[
 \sup_{0 < \e < 1}\e\lt(\mathscr{K}[f^\e_0] + \frac1\e \mathscr{P}[\rho_{f^\e_0}]\rt) < \infty.
\]

If the initial data is well-prepared in the sense that
\[
\lt( \mathscr{H}[f^\e | M_{\rho, -\nabla (-\Delta)^{-\alpha} \rho}](0) + \frac1\e\mathscr{P}[\rho_{f^\e} | \rho](0)\rt) \to 0 \quad \mbox{as } \e \to 0,
\]
then we have
\begin{align*}
f^\e &\to M_{\rho, -\nabla (-\Delta)^{-\alpha} \rho} \quad \mbox{in } L^\infty(0,T; L^1(\R^d \times \R^d)), \cr
\intr f^\e\,d\xi &\to \rho \quad \mbox{in } L^\infty(0,T; L^1 \cap \dot{H}^{-\alpha}(\R^d)),\cr
 \intr \xi f^\e\,d\xi &\to  -\rho\nabla (-\Delta)^{-\alpha}\rho \quad \mbox{in } L^\infty(0,T; L^1(\R^d)).
\end{align*}
Moreover, we obtain the quantitative bound estimates:
\[
 \|\rho^\e - \rho\|_{L^\infty(0,T;\dot{H}^{-\alpha})}^2 \leq C\lt(\e \mathscr{H}[f^\e | M_{\rho, -\nabla (-\Delta)^{-\alpha} \rho}](0) +   \mathscr{P}[\rho_{f^\e} | \rho](0) + \e^2\rt)
 \]
 and
\begin{align*}
&\|f^\e - M_{\rho, -\nabla (-\Delta)^{-\alpha} \rho}\|_{L^\infty(0,T;L^1)}^2 + \|\rho_{f^\e} - \rho\|_{L^\infty(0,T;L^1)}^2 + \|\rho_{f^\e} u_{f^\e}  + \rho\nabla (-\Delta)^{-\alpha} \rho\|_{L^\infty(0,T;L^1)}^2\cr
 &\quad \leq C\lt( \mathscr{H}[f^\e | M_{\rho, -\nabla (-\Delta)^{-\alpha} \rho}](0) + \frac1\e \mathscr{P}[\rho_{f^\e} | \rho](0) + \e\rt).
\end{align*}

If the initial data is  mildly well-prepared in the sense that
\[
{\rm d}_{\rm BL}(\rho_{f^\e_0}, \rho_0) \to 0  \quad \mbox{and} \quad \e \mathscr{H}[f^\e | M_{\rho, -\nabla (-\Delta)^{-\alpha} \rho}](0) +   \mathscr{P}[\rho_{f^\e} | \rho](0) \to 0 \quad \mbox{as } \e \to 0,
\]
then we have
\begin{align*}
f^\e &\overset{*}{\rightharpoonup} M_{\rho, -\nabla (-\Delta)^{-\alpha} \rho} \quad \mbox{in } L^2(0,T; \calM(\R^d \times \R^d)), \cr
\intr f^\e\,d\xi &\overset{*}{\rightharpoonup} \rho \quad \mbox{in } L^\infty(0,T; \mathcal{M}(\R^d)),\cr
  \intr \xi f^\e\,d\xi &\overset{*}{\rightharpoonup} -\rho\nabla (-\Delta)^{-\alpha}\rho \quad \mbox{in }  L^2(0,T; \mathcal{M}(\R^d)) .
  \end{align*}
In addition, the following stability estimate holds:
\begin{align*}
&\sup_{0 \leq t \leq T} {\rm d}^2_{\rm BL}(\rho_{f^\e}(t), \rho(t))  + \int_0^t {\rm d}^2_{\rm BL}(\rho_{f^\e} u_{f^\e}, -\rho\nabla (-\Delta)^{-\alpha} \rho)\,ds + \int_0^t {\rm d}^2_{\rm BL}(f^\e, M_{\rho, -\nabla (-\Delta)^{-\alpha} \rho})\,ds \cr
&\quad \leq C  {\rm d}^2_{\rm BL}(\rho_{f^\e_0}, \rho_0) + C\lt(\e \mathscr{H}[f^\e | M_{\rho, -\nabla (-\Delta)^{-\alpha} \rho}](0) +   \mathscr{P}[\rho_{f^\e} | \rho](0) + \e^2\rt).
\end{align*}
Here, the constant $C>0$ is independent of $\e>0$, depending only on the initial data, $T$, and the limiting solution $\rho$.
\end{theorem}

The notion of mildly well-prepared initial data can be interpreted as an intermediate regime between the classical well-prepared and the fully ill-prepared settings. In contrast to the standard well-prepared case, we do not require the relative entropy $\mathscr{H}[f^\e | M_{\rho, -\nabla (-\Delta)^{-\alpha} \rho}](0)$ to vanish as $\e \to 0$. Our framework permits substantially larger deviations of the kinetic initial state from equilibrium, and in fact accommodates situations where this quantity diverges, with growth up to order
\bq\label{grow_hf}
\mathscr{H}[f^\e | M_{\rho, -\nabla (-\Delta)^{-\alpha} \rho}](0) \sim \frac1{\e^{1-}} \quad \mbox{as } \e \to 0.
\eq
To illustrate this, consider
\[
f^\e_0(x,\xi) = \frac{\rho_0(x)}{(2\pi)^\frac d2}\exp\lt( - \frac{|u_\e(x) - \xi|^2}{2} \rt), \quad u_\e(x) = \frac{v(x)}{\e^{\frac{1-\delta}2}}, 
\]
where $v \in L^2(\rho_0\,dx)$ and $\delta \in (0,1)$.  Moreover, assume that  
\[
\rho_0 \in \calP \cap \dot H^{-\alpha}(\R^d) \quad \mbox{and} \quad \nabla (-\Delta)^{-\alpha} \rho_0 \in L^2(\rho_0\,dx).
\]
Then, we find
\[
 \e \mathscr{K}[f^\e_0] + \mathscr{P}[\rho_{f^\e_0}]  =    \mathscr{P}[\rho_0]  + \frac{\e^\delta}2\intr \rho_0 |v|^2\,dx   +\frac d2 \intr \rho_0\,dx  <\infty
\]
uniformly in $\e \in (0,1)$.  

Let $v_0(x) := -\nabla (-\Delta)^{-\alpha} \rho_0$. Using the identity for Maxwellians, we have
\[
\mathscr{H}[f^\e | M_{\rho, -\nabla (-\Delta)^{-\alpha} \rho}](0) = \mathscr{H}[M_{\rho_0, u_\e} | M_{\rho_0, v_0}] = \frac12 \intr \rho_0 | u_\e - v_0|^2\,dx.
\]
Since $|u_\e|^2=\e^{-(1-\delta)}|v|^2$ and $v_0 \in L^2(\rho_0\,dx)$, it follows that
\[
 \intr \rho_0 | u_\e - v_0|^2\,dx \sim \frac1{\e^{1-\delta}}
\]
for sufficiently small $\e \in (0,1)$ and any $\delta \in (0,1)$, which yields \eqref{grow_hf}. This shows that our convergence result remains valid well beyond the classical well-prepared regime, and
encompasses a range of ill-prepared configurations with large initial modulated energies. 

Nevertheless, the terminology ``mildly well-prepared'' is justified by the additional requirement
\[
{\rm d}_{\rm BL}(\rho_{f^\e_0}, \rho_0) \to 0,
\]
which ensures that the macroscopic density is correctly aligned with the limiting initial data. This condition is essential for quantitative convergence estimates: without such consistency at the level of densities, one cannot close the stability inequalities. If this condition is removed, only qualitative convergence can be expected, typically via compactness arguments (see \cite{NPS01} for example). 

This viewpoint is also consistent with the formal asymptotic expansion presented in Appendix \ref{app_optimal_high}. In the high-field regime, the expansion suggests that both the kinetic solution and the macroscopic density converge at first order, namely
\[
f^\e - M_{\rho,-\nabla(-\Delta)^{-\alpha}\rho} = O(\e), \quad \rho_{f^\e} - \rho = O(\e),
\]
with no additional cancellation  in the density variable.  In this sense, the BL-based stability estimate in Theorem \ref{thm_kin2} is consistent with the formally optimal scaling in a weak topology.   More precisely, whenever the right-hand side of the BL-based stability estimate is of order $O(\e^2)$, the estimate yields the natural first-order convergence rate at the level of weak distances.

\begin{remark}In the absence of the diffusion term in \eqref{eq_high_sca}, the kinetic equation simplifies to
\[
\e \pa_t f^\e + \e \xi\cdot \nabla f^\e  = \nabla_\xi \cdot \lt( (\xi + \nabla (-\Delta)^{-\alpha} \rho_{f^\e})   f^\e\rt).
\]
At the formal level, as $\e \to 0$, this structure suggests a concentration of $f^\e$ in velocity space, corresponding to the monokinetic ansatz:
\[
f^\e(x,\xi) \sim \rho_{f^\e}(x) \otimes \delta_{{\rm u}^\e}(\xi), \quad {\rm u}^\e = - \nabla (-\Delta)^{-\alpha} \rho_{f^\e}.
\]
Carrying out the same moment computations as above yields the aggregation equation \eqref{eq_ag} as the limiting macroscopic dynamics.

It is worth emphasizing that within the high-field scaling, the limiting macroscopic equation \eqref{eq_ag} remains the same whether or not diffusion is present. 
What differs is the kinetic relaxation mechanism: with diffusion we obtain convergence towards the local Maxwellian
\[
f^\e \to M_{\rho, -\nabla (-\Delta)^{-\alpha}\rho},
\] 
whereas without diffusion, the solution converges instead to the monokinetic ansatz
\[
f^\e \to \rho(x)\otimes \delta_{-\nabla(-\Delta)^{-\alpha}\rho}(\xi).
\] 
Thus, although the macroscopic dynamics are identical, the microscopic pathway to the limit is fundamentally different. Rigorous justifications of this limiting behavior can be found in \cite{FS15, Jab00} via compactness arguments, and in \cite{CC20, CFI25, CJ24} for quantitative stability estimates.
\end{remark}
 
%
%
%
%
%
\subsubsection{gSQG limit}

We next study the gSQG limit. Specifically, by setting ${\rm A} = \e$, ${\rm B} = \tau =1$, and assuming a strong, constant magnetic field, the kinetic equation reduces to
\bq\label{kin_sqg_sca}
\e \pa_t f^\e + \xi\cdot \nabla f^\e  - \nabla (-\Delta)^{-\alpha} \rho_{f^\e}  \cdot \nabla_\xi  f^\e + \frac1\e \xi^\perp \cdot \nabla_\xi f^\e = \nabla_\xi \cdot (\nabla_\xi f^\e + \xi f^\e).
\eq
Here the additional term $\frac1\e \xi^\perp \cdot \nabla_\xi f^\e$ represents the rapid rotation induced by the magnetic field, enforcing a gyroscopic constraint that suppresses the longitudinal velocity component and generates an effective drift in the perpendicular direction. In the three-dimensional setting with a magnetic field aligned along $e_3$, this operator corresponds to $\xi^\perp = (-\xi_2,\xi_1,0)$. Since $\nabla^\perp$ acts only on the plane orthogonal to the magnetic field, our analysis naturally reduces to the two-dimensional case. Such a scaling is in line with dimensionless analyses of kinetic equations under strong magnetic fields (see, e.g., \cite{Bos07,BV25,GSR99}).

Unlike the diffusive and high-field regimes, here transport and forcing terms act at the same order as the linear Fokker--Planck operator, thus the limit cannot be obtained directly from the moment equations. The key mechanism is the strong perpendicular rotation operator $\xi^\perp \cdot \nabla_\xi$, which enforces rapid angular averaging in velocity space. At leading order, this yields $f^\e \sim M_{\rho_{f^\e},0}$, while the first-order correction produces a transverse flux. In this process, a term involving $\nabla^\perp \rho$ does appear in the {\it kinetic corrector}, but it vanishes upon taking the divergence in the continuity equation due to the identity $\nabla \cdot \nabla^\perp \equiv 0$ in two dimensions. This cancellation mechanism mirrors what occurs in the rigorous relative entropy analysis, where cross terms involving $\nabla^\perp \rho$ disappear after integration by parts.

As a consequence, passing to the limit $\e \to 0$, we arrive at the generalized SQG equation
\bq\label{eq_sqg}
\partial_t \rho -  \nabla \cdot (\rho \nabla^\perp (-\Delta)^{-\alpha} \rho) = 0.
\eq
A detailed operator-based derivation, including the construction of the corrector, is given in Appendix \ref{app:gsqg}.

The gSQG equation \eqref{eq_sqg} can be interpreted as a conservative transport model in which the density is advected by the incompressible velocity field $\nabla^\perp (-\Delta)^{-\alpha}\rho$. 
It preserves mass and inherits a Hamiltonian structure with respect to the interaction energy $\mathscr{P}[\rho]$, in contrast with the dissipative gradient-flow character of the diffusive and high-field limits. 
In particular, when $\alpha=1$, equation \eqref{eq_sqg} coincides with the two-dimensional incompressible Euler equation written in vorticity form, thereby situating the gSQG model within the broader family of active scalar dynamics.

We now turn to the rigorous justification of this limit. The following theorem establishes the convergence of weak entropy solutions of the kinetic equation \eqref{kin_sqg_sca} towards solutions of the gSQG equation \eqref{eq_sqg}, under suitable assumptions on the initial data.
 
\begin{theorem}[gSQG limit]\label{thm_kin3}
Let $d=2$, $T>0$, and let $\rho \in L^\infty([0,T];\calP(\R^d))$ be the unique classical solution to \eqref{eq_sqg} satisfying
\[
 \nabla^\perp (-\Delta)^{-\alpha} \rho \in L^\infty(0,T; W^{1,\infty}(\R^2)) \cap W^{1,\infty}(0,T; L^\infty(\R^2)),
\]
\[
\rho \in L^\infty(0,T; H^{1-\alpha}(\R^2)), \quad \mbox{and} \quad \nabla \log \rho \in W^{1,\infty}(0,T; L^\infty(\R^2)).
\]
Let $\{f^\e\}_{\e > 0}$ be a family of weak entropy solutions to the equation \eqref{kin_sqg_sca} on the time interval $[0,T]$, with initial data   $\{f^\e_0\}_{\e > 0}$ satisfying
\[
  \sup_{0 < \e < 1}\e\lt(\mathscr{K}[f^\e_0] + \mathscr{P}[\rho_{f^\e_0}]\rt)< \infty.
\]

If the initial data is well-prepared  in the sense that
\[
  \mathscr{H}[f^\e | M_{\rho, 0}](0) + \mathscr{P}[\rho_{f^\e} | \rho](0) \to 0 \quad \mbox{as } \e \to 0,
  \]
then we have
\begin{align*}
f^\e &\to M_{\rho, 0} \quad \mbox{in } L^\infty(0,T; L^1(\R^2 \times \R^2)), \cr
\int_{\R^2} f^\e\,d\xi &\to \rho \quad \mbox{in } L^\infty(0,T; L^1 \cap \dot{H}^{-\alpha}(\R^2)),\cr
\frac1\e \int_{\R^2} \xi f^\e\,d\xi &\overset{*}{\rightharpoonup} - \rho \nabla^\perp (-\Delta)^{-\alpha} \rho - \rho \nabla^\perp \log\rho  \quad \mbox{in } \calM((0,T)\times \R^2).
\end{align*}
In addition, the following stability estimate holds:
 \[
 \|f^\e - M_{\rho, 0}\|_{L^\infty(0,T;L^1)}^2 + \|\rho_{f^\e} - \rho\|_{L^\infty(0,T;L^1\cap \dot{H}^{-\alpha})}^2 \leq C\lt(\mathscr{H}[f^\e | M_{\rho, 0}](0) +   \mathscr{P}[\rho_{f^\e} | \rho](0) + \e\rt)
 \]
 and
\begin{align}\label{gSQG_mom_q}
\begin{aligned}
&\frac1\e {\rm d}_{\rm BL}^T(\rho_{f^\e} u_{f^\e}, - \rho \nabla^\perp (-\Delta)^{-\alpha} \rho -  \rho \nabla^\perp \log\rho ) \cr
&\quad \leq C \lt( \sqrt{\mathscr{H}[f^\e | M_{\rho, 0}](0)} + \sqrt{\mathscr{P}[\rho_{f^\e} | \rho](0)}  + \sqrt \e\rt) + C \lt(\mathscr{H}[f^\e | M_{\rho, 0}](0) +  \mathscr{P}[\rho_{f^\e} | \rho](0)  +  \e\rt).
\end{aligned}
\end{align}

If we further assume that $\alpha \in [\frac12, 1)$ and $\rho \in L^\infty(0,T;H^{\gamma+1}(\R^2))$ for some $\gamma \ge 4$, then the density error also satisfies the following quantitative estimate in the negative Sobolev space $H^{-\gamma}(\R^2)$:
\bq\label{gSQG_rho}
\|(\rho_{f^\e} -\rho)(t)\|_{H^{-\gamma}} \le C\lt(\|\rho_{f_0^\e} - \rho_0\|_{H^{-\gamma}} + \mathscr{H}[f^\e | M_{\rho, 0}](0) + \mathscr{P}[\rho_{f^\e} | \rho](0)+ \e\rt).
\eq

 Here, the constant $C>0$ is independent of $\e>0$, depending only on the initial data, $T$, and the limiting solution $\rho$.
\end{theorem}

\begin{remark}
Formal computations in Appendices \ref{app:gsqg} and \ref{app_optimal} show that
\[
f^\e = M_{\rho,0} + \e f_1 + O(\e^2),
\]
with
\[
f_1= - \lt((\nabla \rho + \rho \nabla (-\Delta)^{-\alpha}\rho)^\perp \cdot \xi\rt)M_{1,0}.
\]

Since $f_1$ is odd in $\xi$, we have
\[
\int_{\R^2} f_1\,d\xi=0.
\]
Thus, at the formal level, this suggests
\[
f^\e-M_{\rho,0}=O(\e), \quad \rho_{f^\e}-\rho=O(\e^2).
\]
In particular, as in the diffusive regime, the density variable gains one additional order due to the oddness of the first-order corrector in the velocity variable. We emphasize, however, that Theorem \ref{thm_kin3} only yields an $O(\e)$-type quantitative stability estimate, and does not address this formally sharper $O(\e^2)$ density rate.

On the other hand, the first-order correction can still be absorbed into a shifted Maxwellian. Indeed, introducing the effective velocity
\[
{\rm u}_\e:= - \e\lt(\nabla^\perp\log\rho + \nabla^\perp(-\Delta)^{-\alpha}\rho\rt),
\]
one may rewrite the first-order term as
\[
\e f_1=\xi\cdot {\rm u}_\e\,M_{\rho,0}.
\]
Hence, using the same first-order Taylor expansion of $M_{\rho,{\rm u}_\e}$ around ${\rm u}=0$ as in the diffusive case, one formally obtains the improved approximation
\[
f^\e-M_{\rho,{\rm u}_\e} =O\bigl(\e^2(1+|\xi|^2)M_{\rho,0}\bigr).
\]
Thus, while the reference Maxwellian $M_{\rho,0}$ approximates the kinetic distribution only at order $O(\e)$, the shifted Maxwellian $M_{\rho,{\rm u}_\e}$ captures the leading drift correction and yields a formally sharper second-order approximation at the kinetic level. A corresponding rigorous estimate for $f^\e-M_{\rho,{\rm u}_\e}$ in a weaker negative Sobolev topology is given in Remark \ref{rmk_opt_gsqg}.
\end{remark}

A key structural feature of the gSQG scaling is that the first-order momentum flux contains the divergence-free component $\nabla^\perp \rho$. As a consequence, when the momentum balance is paired with gradients of test functions, this part disappears thanks to the identity $\nabla\cdot\nabla^\perp \equiv 0$. This cancellation underlies the weak estimate \eqref{gSQG_mom_q} for the rescaled momentum.

This mechanism should be contrasted with the other two regimes. In the diffusive limit, the parabolic structure of the limiting equation makes it possible to obtain sharper quantitative estimates for the density in weaker topologies, in a way consistent with the formally optimal $O(\e^2)$ scaling. In the high-field regime, the BL framework is naturally adapted to the transport structure of the limit equation and leads to first-order weak convergence estimates consistent with the formal asymptotics.

By contrast, although the formal expansion in the gSQG regime also suggests the improved density scaling $\rho_{f^\e}-\rho=O(\e^2)$ due to the oddness of the first-order corrector, this sharper rate is not directly captured by our rigorous argument. The reason is that the limiting equation is now a conservative active scalar equation with Hamiltonian structure, rather than a parabolic equation with regularizing effects. Accordingly, the natural quantitative object in this regime is different from those in the diffusive and high-field cases: rather than a refined density estimate of the same type, the relevant object is the rescaled momentum, for which the rotational cancellation encoded in $\nabla^\perp \rho$ yields the appropriate weak stability structure. In this sense, estimate \eqref{gSQG_rho} reflects the weak quantitative structure intrinsic to the gSQG scaling, although the formally suggested $O(\e^2)$ density improvement remains beyond the scope of the present argument.

%
%
%
%
%

 \subsection{Novelty and contributions}

The present work provides the first unified analysis of three representative asymptotic regimes for the Vlasov--Fokker--Planck equation with Riesz interactions: the diffusive limit, the high-field limit, and the gSQG limit. Our approach is based on a common relative entropy framework, supplemented with modulated interaction energies and bounded Lipschitz stability estimates. This combination yields not only qualitative convergence but also explicit quantitative error bounds. It also allows us to treat well-prepared initial data in all three regimes and, in the high-field regime, a class of mildly well-prepared initial data even without a uniform bound on the total energy.

For clarity, we discuss the novelty and contributions of each regime separately.

%
%
%
%
%
\subsubsection{Diffusive limit}

The study of the diffusive (or low-field) limit of kinetic Fokker--Planck type equations has a long history, beginning with Kramers' formal analysis of the Smoluchowski--Kramers approximation \cite{Kra40} and Nelson's rigorous stochastic interpretation \cite{Nel67}. Subsequent works employed probabilistic and asymptotic techniques, as well as variational methods, to justify the limit in the absence of interaction forces \cite{Fre04}. In particular, the variational approach in \cite{DLPSS18} provided quantitative estimates in terms of entropy and Wasserstein distances by exploiting functional inequalities. When nonlocal interactions are present, rigorous derivations were obtained under smoothness and integrability assumptions on the potential, often relying on duality arguments and compactness methods \cite{DLPS17}. For the Vlasov--Poisson--Fokker--Planck (VPFP) system with Coulomb interactions, the low-field limit toward drift-diffusion equations has been established through entropy bounds combined with compactness arguments \cite{GM10, Gou05, PS00,  WLL15}, though without quantitative error control. More recently, quantified diffusive limits have been developed in \cite{CT22} for equations with smooth interaction forces, and further advances have extended these techniques to include singular interaction kernels by combining coarse-graining, Wasserstein stability estimates, and uniform weighted Sobolev bounds. In addition, optimal convergence rates for the VPFP system have been established in \cite{Zho22} under perturbation framework, while asymptotic-preserving schemes capable of capturing the diffusion limit have been proposed in \cite{LJH21}.

In contrast to the above approaches, our work provides what appears to be the \emph{first quantitative} analysis of the diffusive limit based directly on the relative entropy method. Previous results on quantitative convergence in weak topologies typically relied on coarse-graining the kinetic dynamics into a drift-diffusion type structure and employing Wasserstein distance estimates \cite{CT22}. 

Our strategy departs from this paradigm: we reformulate the limiting aggregation-diffusion equation in conservative form and apply the relative entropy framework directly, thereby exploiting the full dissipative structure of the kinetic equation. This allows us to strengthen convergence estimates in weak topologies under well-prepared initial data. For instance, while \cite{CT22} established
\[
\sup_{0 \leq t \leq T}  {\rm d}_2(\rho_{f^\e}(t), \rho(t)) \leq C{\rm d}_2(\rho_{f^\e_0}, \rho_0) + C \e,
\]
where ${\rm d}_2$ denotes the second-order Wasserstein distance, so that the convergence rate is saturated by the $\e$-term regardless of the accuracy of the initial data, our analysis shows that, for suitably prepared initial data, the error can be reduced to $O(\e^2)$ in low-regularity topologies. We also note that the perturbative analysis in \cite{Zho22} yields an $O(\e)$ convergence rate for the kinetic distribution $f^\e$, but does not address the improved convergence of the macroscopic density $\rho^\e$. In this sense, the $O(\e^2)$ density estimate obtained here is consistent with the formally optimal diffusive scaling.

%
%
%
%
%

\subsubsection{High-field limit}

The high-field scaling was first introduced in semiconductor kinetic theory to describe charge transport under strong external fields \cite{Pou92}. A rigorous analysis for the Vlasov--Poisson--Fokker--Planck (VPFP) system in this regime was developed in \cite{NPS01}, where convergence toward drift-type aggregation equations was established.

The high-field regime has been extensively studied for VPFP systems \cite{ACGS01,CCP22,GNPS05,NPS01}. Classical works such as \cite{GNPS05,NPS01} established convergence toward drift-type aggregation equations, typically relying on compactness arguments, entropy solution frameworks, or modulated energy methods. These analyses provided a rigorous qualitative justification of the limiting dynamics, but remained essentially non-quantitative, yielding at best weak convergence of measures or stability in entropy-type topologies. In particular, quantitative error estimates in the presence of diffusion have not been available. In parallel, in the absence of diffusion, the high-field limit for the Vlasov equation with nonlocal interactions has been analyzed in \cite{CFI25,CJ24,FS15}, further highlighting the interplay between interaction and relaxation effects.

Our contribution is to introduce a genuinely quantitative framework for the high-field limit with diffusion. By employing relative entropy estimates directly at the kinetic level, we exploit the dissipative structure of the Fokker--Planck operator and derive explicit convergence rates under well-prepared initial data. This represents, to our knowledge, \emph{the first quantitative} justification of the high-field regime for Vlasov--Fokker--Planck-type models. Moreover, in the weak-topology setting, our BL-based estimates are consistent with the first-order scaling predicted by the formal asymptotic expansion.

%
%
%
%
%

\subsubsection{gSQG limit}

In the periodic setting without a point charge, convergence to the incompressible Euler equation has been established under suitable energy and $L^\infty$ bounds \cite{GSR99,SR00}, while related extensions in the presence of a point charge were obtained in \cite{Mio19}. An alternative approach was proposed in \cite{Bre00}, which introduced a modulated energy framework to prove the convergence of the flux (current) toward dissipative solutions of the incompressible Euler equations. More recently, \cite{BV25} incorporated Fokker--Planck collisions and applied the relative entropy method to derive quantitative convergence rates for the density, but without addressing momentum variables.

Our analysis uncovers a different perspective: under the strong magnetic field scaling, the Vlasov--Fokker--Planck model with Riesz interaction potentials gives rise not to the incompressible Euler equation, but to a generalized surface quasi-geostrophic (gSQG) equation. To the best of our knowledge, this connection between fast rotation and gSQG-type dynamics in the presence of Riesz interactions has not been previously established. It highlights how dimension reduction and nonlocal forces can interact in unexpected ways. By combining entropy dissipation with modulated interaction energy estimates, we are able to rigorously capture this singular limit and, crucially, provide explicit convergence rates not only for the density but also for the flux variables. In this way, our analysis complements \cite{BV25} by extending quantitative control beyond the density to momentum observables in a setting where nonlocal Riesz interactions drive a gSQG-type macroscopic limit.

%
%
%
%
%
 
 \subsection{Organization of paper} The rest of this paper is organized as follows. 
In Section \ref{sec_pre}, we develop the analytic framework underlying our asymptotic analysis. This includes structural decompositions of the relative entropy and uniform energy-dissipation identities.
 Section \ref{sec_qs} establishes quantitative stability tools, combining relative entropy, weak-topology estimates, and modulated interaction energies. In Section \ref{sec_vfp}, these methods are applied to derive the hydrodynamic limits in the diffusive, high-field, and gSQG regimes.  Finally, the appendices provide the formal derivation of the gSQG limit, a discussion of the formal optimal convergence rates in the three singular regimes, and the verification of the regularity assumptions imposed on the limiting equations.

%
%
%
%

\section{Preliminaries}\label{sec_pre}

In this section, we present the analytic framework that underpins our asymptotic analysis. The discussion is organized around two complementary ingredients. We first recall the basic structural properties of relative entropy, including its variational characterization, decomposition formulas, and functional inequalities, which allow entropy methods to connect microscopic information with macroscopic quantities. We then establish energy-dissipation balances that provide uniform control of kinetic and interaction energies throughout the evolution. Taken together, these tools form the basis for the quantitative convergence estimates developed in the next section.
%
%
%
%
%
\subsection{Relative entropy and fundamental properties}
We begin with the concept of relative entropy with respect to local Maxwellians. This quantity plays a central role in our analysis: on the one hand, it admits a natural decomposition into microscopic deviations and macroscopic mismatches; on the other hand, it enjoys several useful variational and functional-analytic properties. In particular, classical inequalities relate relative entropy to Fisher information, $L^1$-type distances, and coercivity estimates.

A first important aspect is its variational interpretation. For prescribed $(\rho,u)$, the local Maxwellian $M_{\rho,u}$ is the natural equilibrium state associated with the mass and momentum constraints, and in fact arises as the unique minimizer of a constrained free energy functional. We recall this characterization here, since it clarifies the distinguished role of Maxwellians in the relative entropy method.

More specifically, $M_{\rho,u}$ is the unique minimizer of the variational problem:
\bq\label{eq:var-prob}
\min \lt\{ \mathcal{F}[g] := \intr  \lt(g(\xi)\log g(\xi) + \frac12 |\xi|^2 g(\xi)\rt)d\xi : \intr g\,d\xi = \rho, \ \intr \xi g\,d\xi = \rho u \rt\}.
\eq
This characterization can be justified in several ways:

\medskip
\noindent {\bf Via convexity:} The map $z\mapsto z\log z$ is strictly convex on $[0,\infty)$; hence $\mathcal{F}$ in \eqref{eq:var-prob} is strictly convex. For any $g\ge0$ with finite first/second moments, a direct calculation using 
\[
\log M_{\rho,u}(\xi)=\log\rho-\frac d2\log(2\pi)-\frac12|\xi-u|^2
\]
gives the identity
\[
\mathcal{F}[g] = \mathcal{F}[M_{\rho,u}] +  \intr g\log\frac{g}{M_{\rho,u}}\,d\xi
 + \lt(\log\rho-\frac d2\log(2\pi)\rt)\lt(\int g\,d\xi - \rho\rt)
 + u\cdot \lt(\int \xi g\,d\xi - \rho u\rt).
\]
Under the constraints in \eqref{eq:var-prob}, the last two bracketed terms vanish. Thus, we have
\[
\mathcal{F}[g]-\mathcal{F}[M_{\rho,u}] =  \intr g(\xi)\log\frac{g(\xi)}{M_{\rho,u}(\xi)}\,d\xi \ \ge 0.
\]
This decomposition (see also Lemma \ref{lem:rel_decom} below) shows explicitly that $\mathcal{F}[g]$ exceeds its minimal value $\mathcal{F}[M_{\rho,u}]$ by the relative entropy $\mathscr{H}[g|M_{\rho,u}]$. Since the latter is nonnegative and vanishes only when $g=M_{\rho,u}$ almost everywhere, the Maxwellian must be the unique minimizer. In particular, the strict convexity ensures that no other distribution with the same constraints can yield the same value. This argument emphasizes the canonical role of Maxwellians: they are singled out not only by physical heuristics of local equilibrium, but also by a purely convex-analytic property of the free energy functional. 

\medskip
\noindent {\bf Via Lagrange multipliers \cite{Bou93, CIP94}:} 
Alternatively, one may minimize the functional $\mathcal{F}$ in \eqref{eq:var-prob} directly by introducing multipliers $\lambda_0 \in \R$ and $\lambda \in \R^d$ for the mass and momentum constraints. We thus consider
\[
\mathcal{L}[g] := \mathcal{F}[g] 
+ \lambda_0\lt(\intr g\,d\xi - \rho\rt) 
+ \lambda\cdot\lt(\intr \xi g\,d\xi - \rho u\rt).
\]
The Euler--Lagrange condition $\frac{\delta\mathcal{L}}{\delta g}=0$ gives
\[
\log g(\xi) + 1 + \frac12|\xi|^2 + \lambda_0 + \lambda\cdot\xi = 0,
\]
which implies
\[
g(\xi) = C_0 e^{-\frac12|\xi|^2 - \lambda\cdot\xi}, \quad C_0 := e^{-1-\lambda_0}.
\]
Completing the square in the exponent,
\[
-\frac12|\xi|^2 - \lambda\cdot\xi 
= -\frac12|\xi+\lambda|^2 + \frac12|\lambda|^2,
\]
so that
\[
g(\xi) = C_0 e^{\frac12|\lambda|^2}e^{-\frac12|\xi+\lambda|^2}.
\]
The normalization constraint $\intr g\,d\xi = \rho$ determines
\[
C_0 e^{\frac12|\lambda|^2} =\frac{\rho}{(2\pi)^{\frac d2}},
\]
and the moment constraint $\intr \xi g\,d\xi = \rho u$ yields $-\lambda=u$. Hence, we have
\[
g(\xi) = M_{\rho,u}(\xi) ,
\]
which shows that the Maxwellian is the unique minimizer of \eqref{eq:var-prob}. 

This derivation highlights that the Gaussian structure of Maxwellians arises unavoidably from the Euler--Lagrange equations under the moment constraints. The parameters $(\rho,u)$ appear naturally through the Lagrange multipliers, confirming that no other distribution can achieve the minimum. Thus, Maxwellians are singled out as canonical equilibria not only by physical arguments but also by variational principles.
 
 \medskip
\noindent {\bf Via Donsker--Varadhan duality \cite{DZ98, DV75, DE97}:} A third and more abstract route to the characterization of Maxwellians makes use of a variational representation of the relative entropy, often referred to as the Donsker--Varadhan formula. 
It states that for any pair of probability measures $\mu$ and $\nu$ on $\R^d$,
\[
\intr \log \lt(\frac{d\mu}{d\nu}\rt) d\mu 
= \sup_{\varphi}\lt\{ \intr \varphi\,d\mu - \log \intr e^{\varphi}\,d\nu \rt\},
\]
where the supremum may be taken over all bounded continuous test functions $\varphi$, or more generally over measurable $\varphi$ for which $\intr e^{\varphi}\,d\nu<\infty$. 

Specializing this representation to $\mu=g(\xi)\,d\xi$ and $\nu=M_{\rho,u}(\xi)\,d\xi$ yields
\[
\intr g(\xi)\log\frac{g(\xi)}{M_{\rho,u}(\xi)}\,d\xi
= \sup_{\varphi}\lt\{ \intr \varphi(\xi)\, g(\xi)\,d\xi - \log \intr e^{\varphi(\xi)} M_{\rho,u}(\xi)\,d\xi \rt\}.
\]
The right-hand side is clearly nonnegative, since the choice $\varphi\equiv 0$ already yields $0$. 
Moreover, equality can hold only if the supremum is attained at $\varphi=\log(g/M_{\rho,u})$, which forces $g=M_{\rho,u}$ almost everywhere. 
Thus the Maxwellian is once again identified as the unique minimizer.

This dual formulation complements the previous arguments in two important ways. 
First, it highlights a direct link between entropy and exponential moments: the inequality compares a linear functional of $g$ against a logarithmic Laplace transform with respect to $M_{\rho,u}$. 
Second, it connects the variational characterization of Maxwellians to the framework of large deviations, where relative entropy plays the role of a rate functional. 
In this perspective, Maxwellians arise as canonical equilibrium measures not only because they minimize the free energy in a convex-analytic sense, or as stationary points of Euler--Lagrange equations, but also because they saturate the fundamental variational duality between entropy and exponential integrability. 
This observation further reinforces the fundamental role of Maxwellians as equilibrium states in kinetic theory.

Beyond this variational characterization, relative entropy also admits a useful structural decomposition into a microscopic part and a macroscopic part. This decomposition will be one of the main tools in our later stability estimates, since it allows us to separate the deviation from the local Maxwellian with the same moments as $f$ from the mismatch between those moments and a prescribed target state. We now make this structure precise.

\begin{lemma}\label{lem:rel_decom}
Let $f:\R^d\to[0,\infty)$ be a density with $f \in L\log L(\R^d) \cap L^1(\R^d, (1+|\xi|^2)\,d\xi)$. Define
\[
\rho_f := \intr f\,d\xi,\quad 
u_f := \frac1{\rho_f}\intr \xi f\,d\xi,
\]
with the convention $u_f=0$ if $\rho_f=0$. For $(\rho,u)\in\R_+\times\R^d$, let
\[
M_{\rho,u}(\xi) = \frac{\rho}{(2\pi)^{\frac d2}} \exp \Big(-\frac{|\xi-u|^2}{2}\Big).
\]
Then the following decomposition holds:
\begin{equation}\label{eq:rel_decom}
 \intr f \log \frac{f}{M_{\rho,u}}\,d\xi
=   \intr f \log \frac{f}{M_{\rho_f,u_f}}\,d\xi
+   \rho_f \log \frac{\rho_f}{\rho}
+ \frac{\rho_f}{2}|u_f - u|^2.
\end{equation}
\end{lemma}
\begin{proof} 
We first split the logarithm using intermediate Maxwellians:
\[
\log \frac{f}{M_{\rho,u}} = \log \frac{f}{M_{\rho_f,u_f}} + \log \frac{M_{\rho_f,u_f}}{M_{\rho,u_f}} + \log \frac{M_{\rho,u_f}}{M_{\rho,u}}.
\]
Integrating against $f$ gives
\[
\intr f \log \frac{f}{M_{\rho,u}}\,d\xi
= \intr f \log \frac{f}{M_{\rho_f,u_f}}\,d\xi
+ \intr f \log \frac{M_{\rho_f,u_f}}{M_{\rho,u_f}}\,d\xi
+ \intr f \log \frac{M_{\rho,u_f}}{M_{\rho,u}}\,d\xi.
\]
For the second term, since $\log \frac{M_{\rho_f,u_f}}{M_{\rho,u_f}} = \log \frac{\rho_f}{\rho}$, we have 
\[
\intr f \log \frac{M_{\rho_f,u_f}}{M_{\rho,u_f}}\,d\xi = \rho_f \log \frac{\rho_f}{\rho}.
\]
For the third term, expanding
\[
\log \frac{M_{\rho,u_f}}{M_{\rho,u}}
= -\frac{|\xi - u_f|^2 - |\xi - u|^2}{2}
= -\frac{|u_f|^2 - |u|^2 - 2(u_f - u)\cdot\xi}{2}
\]
and integrating yields
\[
\intr f \log \frac{M_{\rho,u_f}}{M_{\rho,u}}\,d\xi
= \frac{\rho_f}{2}|u_f - u|^2.
\]
Combining the three identities gives \eqref{eq:rel_decom}.
\end{proof}

In addition to the decomposition above, several classical inequalities for the relative entropy will be frequently used in our analysis. 
They connect entropy with Fisher information, provide coercivity in the density variable, and relate entropy to $L^1$-type distances between distributions. 
For completeness, we record them here.

\begin{lemma} \label{lem:prop}
Under the same assumptions as in Lemma \ref{lem:rel_decom}, the following properties hold:
\begin{enumerate}[label=(\roman*)]
\item {\rm (Logarithmic Sobolev/Fisher information bound)} For any $u \in \R^d$,
\bq\label{log_sob}
 \intr f \log \frac{f}{M_{\rho_f, u}}\,d\xi 
\leq \frac12 \intr \frac1{f}|\nabla_\xi f + (\xi - u)f|^2\,d\xi.
\eq
\item {\rm (Coercivity: piecewise quadratic-linear)}  
For any $\rho>0$,  
\[
  \rho_f \log \frac{\rho_f}{\rho}
\ge \frac{1}{4\rho} |\rho_f - \rho|^2 \chi_{\{\rho_f \geq \rho\}} + \frac{1}{4} |\rho_f - \rho| \chi_{\{\rho_f < \rho\}}.
\]
\item {\rm (Csisz\'ar--Kullback--Pinsker inequality)}  
For any two probability densities $p,q$,
\[
\|p - q\|_{L^1}^2 \le 2 \int  p \log \frac{p}{q}\,dz.
\]
\end{enumerate}
\end{lemma}
\begin{proof}
For (i), we apply the logarithmic Sobolev inequality for the Gaussian measure (e.g. \cite{G75}) to obtain
\[
 \intr f \log \frac{f}{M_{\rho_f,u}}\,d\xi \leq \frac{1}2 \intr f \lt|\nabla_\xi \log \frac{f}{M_{\rho_f,u}}\rt|^2 d\xi.
\]
Expanding the derivative gives
\[
\nabla_\xi \log \frac{f}{M_{\rho_f,u}} = \frac{\nabla_\xi f}{f} + \xi - u,
\]
which leads to \eqref{log_sob}. The inequalities (ii) and (iii) follow from 
standard convexity arguments and \cite{AMTU00}, respectively.
\end{proof}

\begin{remark} For $a,b > 0$, we get
\[
(\sqrt{a} - \sqrt{b})^2 \leq a\log\lt(\frac ab \rt) - (a-b).
\]
Thus, for two probability densities $p,q$, 
\[
\frac12 \|p - q\|_{L^1}^2 + \|\sqrt{p} - \sqrt{q}\|_{L^2}^2 \leq 2\int p\log\lt(\frac pq \rt)dz.
\]
\end{remark}

%
%
%
%
%
\subsection{Energy inequalities} 
In this subsection, we derive a priori bounds that will form the backbone of our convergence analysis. 
The central quantities are the kinetic free energy $\mathscr{F}[f]$ and the interaction energy $\mathscr{P}[\rho]$, whose balance laws reflect the interplay between transport, diffusion, and nonlocal forces. 
These energy identities yield uniform-in-time controls on kinetic and potential contributions, while the associated dissipation functional $\mathscr{D}[f]$ quantifies relaxation toward local Maxwellians. 
In combination, they provide the fundamental stability estimates needed to handle each asymptotic regime.

Recall our main kinetic equation:
\bq\label{eq_aux}
{\rm A} \pa_t f + {\rm B} \xi\cdot \nabla f  - \nabla (-\Delta)^{-\alpha} \rho_f \cdot \nabla_\xi f = \frac1\tau\nabla_\xi \cdot (\nabla_\xi f + \xi   f).
\eq
For this dynamics, we use the kinetic free energy $\mathscr{F}[f]$ and the interaction energy $\mathscr{P}[\rho_f]$:  
\[
\mathscr{F}[f]= \mathscr{K}[f] + \mathscr{H}[f]= \iint_{\R^d \times \R^d} \lt(\frac{|\xi|^2}2 + \log f\rt) f \,dxd\xi \quad \mbox{and} \quad  \mathscr{P}[\rho_f] = \frac12 \intr \rho_f (-\Delta)^{-\alpha}\rho_f\,dx.
\]
\begin{lemma}\label{lem_en} Let $f$ be a regular solution to the equation \eqref{eq_aux}. Then  we have
\[
\mathscr{F}[f(t)] + \frac1{\rm B} \mathscr{P}[\rho_f(t)] + \frac1{\tau {\rm A}} \int_0^t \mathscr{D}[f(s)]\,ds = \mathscr{F}[f_0] + \frac1{\rm B} \mathscr{P}[\rho_{f_0}],
\]
where the dissipation rate $\mathscr{D}[f]$ is given as
\[
\mathscr{D}[f] := \iint_{\R^d \times \R^d}  \frac1{f}|\nabla_\xi f + \xi  f|^2 \,dxd\xi.
\]
Moreover, we obtain
\[
\int_0^T\mathscr{K}[f(t)]\,dt \leq dT + \tau {\rm A}\mathscr{K}[f_0] +\frac{\tau {\rm A}}{\rm B} \mathscr{P}[\rho_{f_0}].
\]
\end{lemma}
\begin{proof}
Straightforward computations give
\begin{align}\label{en_est1}
\begin{aligned}
\frac{d}{dt}\iint_{\R^d \times \R^d}  \frac{|\xi|^2}2 f \,dxd\xi &= - \frac1{{\rm A}}\iint_{\R^d \times \R^d}  \xi \cdot \nabla (-\Delta)^{-\alpha} \rho_f f \,dxd\xi - \frac1{\tau {\rm A}} \iint_{\R^d \times \R^d}  \xi \cdot  (\nabla_\xi f + \xi  f)\,dxd\xi\cr
&=-\frac{1}{2 {\rm B}} \frac{d}{dt}\intr \rho_f (-\Delta)^{-\alpha}\rho_f\,dx  - \frac1{\tau {\rm A}} \iint_{\R^d \times \R^d}  \xi \cdot  (\nabla_\xi f + \xi  f)\,dxd\xi
\end{aligned}
\end{align}
and
\[
\frac{d}{dt}\iint_{\R^d \times \R^d}  f \log f \,dxd\xi =  - \frac1{\tau {\rm A}}\iint_{\R^d \times \R^d}  \frac{\nabla_\xi f}{f} \cdot (\nabla_\xi f + \xi  f)\,dxd\xi.
\]
Here we used
\begin{align*}
 \iint_{\R^d \times \R^d}  \xi \cdot \nabla (-\Delta)^{-\alpha} \rho_f f \,dxd\xi &= \intr \rho_f u_f \cdot \nabla (-\Delta)^{-\alpha} \rho_f \,dx \cr
 &= - \intr \nabla \cdot (\rho_f u_f) (-\Delta)^{-\alpha} \rho_f  \,dx \cr
 &=  \frac{\rm A}{2 {\rm B}}\frac{d}{dt}\intr \rho_f (-\Delta)^{-\alpha}\rho_f\,dx. 
\end{align*}
Combining the above estimates provides the total energy inequality. 

For the kinetic energy bound estimate, we use \eqref{en_est1} to obtain 
\begin{align*}
\frac{d}{dt}\lt(\iint_{\R^d \times \R^d}  \frac{|\xi|^2}2 f \,dxd\xi +\frac{1}{2 {\rm B}}\intr \rho_f (-\Delta)^{-\alpha}\rho_f\,dx \rt) &= - \frac1{\tau {\rm A}} \iint_{\R^d \times \R^d}  \xi \cdot  (\nabla_\xi f + \xi  f)\,dxd\xi \cr
&= \frac{d}{\tau {\rm A}} - \frac1{\tau {\rm A}}\iint_{\R^d \times \R^d}  |\xi|^2 f \,dxd\xi,
\end{align*}
and this implies
\[
\int_0^T\iint_{\R^d \times \R^d}   |\xi|^2  f \,dxd\xi dt \leq d  T + \tau {\rm A}\iint_{\R^d \times \R^d}  \frac{|\xi|^2}2 f_0 \,dxd\xi +\frac{\tau {\rm A}}{2 {\rm B}}\intr \rho_{f_0} (-\Delta)^{-\alpha}\rho_{f_0}\,dx.
\]
This completes the proof.
\end{proof}

\begin{remark} We emphasize that our argument only requires a bound on the time-integrated kinetic energy $\int_0^T \mathscr{K}[f(t)]\,dt$, which follows directly from the dissipation of the Fokker--Planck operator, without any further control of the entropy functional. In contrast, obtaining a uniform-in-time bound on the kinetic energy, it is necessary to ensure the integrability of the logarithmic moment $f|\log f|$ together with the spatial moment $|x|f$, see \cite{CJe23, CT22} for instance. The logarithmic term provides control near vacuum regions where the entropy density $f\log f$ becomes negative, while the spatial moment assumption $|x|f_0 \in L^1(\R^d\times\R^d)$ guarantees sufficient decay at infinity. These two ingredients are required only for estimating the instantaneous kinetic energy $\sup_{0\leq t\leq T}\mathscr{K}[f(t)]$.
\end{remark}

%
%
%
%
%
%

\section{Quantitative stability via relative entropy, weak-topology estimates, and modulated interaction energy}\label{sec_qs}

In this section, we develop quantitative stability tools that link the dissipation structure of the kinetic equation with weak topologies suited for convergence analysis. The energy estimates in Section \ref{sec_pre} already provide robust control of kinetic fluctuations around their associated local Maxwellians, and the crucial point is that the dissipation terms generated in the relative entropy identity carry scaling prefactors of order $\frac1\e$ or $\frac1{\e^2}$. The key coercive term here is the \emph{Fisher information dissipation}, which directly controls microscopic deviations from the local equilibrium. Combined with weak-topology stability tools, this mechanism allows us to convert microscopic dissipation into quantitative convergence estimates for the macroscopic system, both in strong topologies under well-prepared data and, in suitable regimes, in weaker topologies adapted to low-regularity or mildly prepared settings.

%
%
%
%
%
%

\subsection{Weak-topology estimates: BL distance and negative Sobolev norms}

In this subsection, we derive weak-topology estimates that connect the entropy-dissipation structure of the kinetic equation with quantitative convergence tools. While the relative entropy provides a natural measure of distance to local Maxwellians, several of our convergence statements are formulated in weak topologies on probability measures. As a first step, we therefore analyze the bounded Lipschitz distance, which metrizes weak convergence on $\calP(\R^d)$ and $\calP(\R^d\times\R^d)$ without additional moment assumptions; on bounded domains, it is equivalent to the first-order Wasserstein distance up to constants. We then complement the BL framework with negative Sobolev estimates, which will be useful when converting microscopic dissipation into macroscopic stability estimates in weak spatial topologies.

We first recall a stability estimate for continuity equations in BL distance whose proof can be found in \cite{C21} (see also \cite{CC20, CC21, FK19}). 
\begin{lemma}\label{lem_gd} Let $T>0$ and $\bar \rho : [0,T] \to \mathcal{P}(\R^d)$ be a narrowly continuous solution of 
\[
\pa_t \bar\rho + \nabla \cdot (\bar\rho \bar u) = 0,
\] 
that is, for every $\varphi \in C_b(\R^d)$ the map
\[
t \mapsto \intr \varphi(x) \bar\rho(t,x)\,dx
\]
is continuous. Assume that the Borel vector field $\bar u$ satisfies
\[
\int_0^T\intr |\bar u(t,x)|\bar\rho(t,x)\,dx dt < \infty.
\]
Let $\rho \in C([0,T];\mathcal{P}_1(\R^d))$ be a solution of the following continuity equation:
\[
\pa_t \rho + \nabla \cdot (\rho u) = 0
\]
with the velocity fields $u \in L^\infty(0,T; \dot{W}^{1,\infty}(\R^d))$. Then there exists a positive constant $C$ depending only on $T$ and $\|\nabla u\|_{L^\infty((0,T) \times \R^d)}$ such that for all $t \in [0,T]$:
\begin{enumerate}[label=(\roman*)]
\item Estimate for the density:
\[
{\rm d}^2_{\rm BL}(\bar \rho(t), \rho(t))  \leq C {\rm d}^2_{\rm BL}(\bar\rho(0), \rho(0)) + C\int_0^t \intr |(\bar u- u)(s,x)|^2\bar \rho(s,x)\,dxds.
\]
\item Estimate for the momentum:
\begin{align*}
{\rm d}^2_{\rm BL}((\bar \rho \bar u)(t) , (\rho u)(t)) &\leq C {\rm d}^2_{\rm BL}(\bar\rho(0), \rho(0)) + C\intr |(\bar u- u)(t,x)|^2\bar \rho(t,x)\,dx \cr
&\quad + C\int_0^t \intr |(\bar u- u)(s,x)|^2\bar \rho(s,x)\,dxds.
\end{align*}
In particular, we have
\[
\int_0^t {\rm d}^2_{\rm BL}((\bar \rho \bar u)(s) , (\rho u)(s))\,ds \leq C {\rm d}^2_{\rm BL}(\bar\rho(0), \rho(0)) + C\int_0^t \intr |(\bar u- u)(s,x)|^2\bar \rho(s,x)\,dxds.
\]
\end{enumerate}
\end{lemma}

We next state a complementary estimate for kinetic distributions, linking the BL distance to local Maxwellians and providing a quantitative connection between the Fisher information dissipation and weak convergence in the BL distance.

\begin{lemma}\label{lem_gd2} Let $T > 0$ and $f \in C([0,T]; \mathcal{P}(\R^d \times \R^d))$ satisfy
\[
|\xi|^2 f\in L^1((0,T) \times \R^d \times \R^d) \quad \mbox{and} \quad \nabla_\xi \sqrt{f} \in L^2((0,T) \times \R^d \times \R^d).
\]
Let $(\rho_f, u_f)$ and $(\rho, u)$ satisfy the assumptions of Lemma \ref{lem_gd}. Then, there exists $C>0$ depending only on $d$, $T$, and $\|\nabla u\|_{L^\infty}$, such that for all $t \in [0,T]$
\begin{align*}
\int_0^t {\rm d}^2_{\rm BL}(f, M_{\rho, u})\,ds &\leq C  {\rm d}^2_{\rm BL}(\rho_{f_0}, \rho_0) + C\int_0^t\iint_{\R^d \times \R^d} \frac1{f}|\nabla_\xi f + (\xi - u_f) f|^2\,dxd\xi ds \cr
&\quad + C\int_0^t \intr \rho_f |u_f - u|^2\,dxds.
\end{align*}
In particular, we have
\begin{align*}
\int_0^t {\rm d}^2_{\rm BL}(f, M_{\rho, 0})\,ds &\leq C  {\rm d}^2_{\rm BL}(\rho_{f_0}, \rho_0) + C\int_0^t\iint_{\R^d \times \R^d} \frac1{f}|\nabla_\xi f + \xi f|^2\,dxd\xi ds   + C\int_0^t \intr \rho_f |u_f - u|^2\,dxds.
\end{align*}
\end{lemma}
\begin{proof}
Let $\varphi \in W^{1,\infty}(\R^d \times \R^d)$. We decompose
\begin{align*}
&\iint_{\R^d \times \R^d} \varphi(x,\xi) (f - M_{\rho, u})\,dxd\xi \cr
&\quad = \iint_{\R^d \times \R^d} \varphi(x,\xi) (f - M_{\rho_f, u_f})\,dxd\xi + \iint_{\R^d \times \R^d} \varphi(x,\xi) ( M_{\rho_f, u_f} - M_{\rho_f, u})\,dxd\xi \cr
&\qquad + \iint_{\R^d \times \R^d} \varphi(x,\xi) ( M_{\rho_f, u} - M_{\rho, u})\,dxd\xi\cr
&\quad =: I + II + III.
\end{align*}
For $I$, we use Lemma \ref{lem:prop} (i) and (iii) to obtain
\[
I \leq \|\varphi\|_{L^\infty}\|f - M_{\rho_f,u_f}\|_{L^1} \leq  C\lt(\iint_{\R^d \times \R^d} \frac1{f}|\nabla_\xi f + (\xi - u_f) f|^2\,dxd\xi\rt)^\frac12.
\]
For $II$, by the mean-value theorem, we get
\begin{align*}
\lt| M_{1, u_f} - M_{1, u}\rt| &= \lt| \int_0^1 \frac{d}{dt}M_{1, \tau u_f + (1-\tau) u}\,d\tau \rt| \cr
&=   \lt| \int_0^1 M_{1, \tau u_f + (1-\tau) u} (\tau u_f + (1-\tau) u-\xi)\cdot (u_f - u) \,d\tau \rt|,
\end{align*}
and thus
\begin{align*}
\intr \lt| M_{1, u_f}(\xi) - M_{1, u}(\xi)\rt| d\xi &\leq  |u_f - u|\int_0^1 \intr M_{1, \tau u_f + (1-\tau) u}(\xi) |\tau u_f + (1-\tau) u-\xi|\,d\xi d\tau\cr
&=  |u_f - u| \intr M_{1, 0}(\xi) |\xi|\,d\xi \cr
&\leq C|u_f - u|
\end{align*}
for some $C>0$ depending only on $d$. This gives
\[
II  = \iint_{\R^d \times \R^d} \varphi(x,\xi) \rho_f (M_{1, u_f} - M_{1, u})\,dxd\xi \leq C \intr \rho_f |u_f - u|\,dx \leq C\|\rho_f\|_{L^1} \lt( \intr \rho_f |u_f - u|^2\,dx \rt)^{\frac12}.
\]
For  $III$, we observe that $(x,\xi) \mapsto \varphi(x,\xi) M_{1,u(x)}(\xi)$ is bounded and Lipschitz if $u$ is Lipschitz continuous. Thus, we find
\[
III =  \iint_{\R^d \times \R^d} \varphi(x,\xi) (\rho_f - \rho) M_{1, u}\,dxd\xi \leq C {\rm d}_{\rm BL} (\rho_f,\rho).
\]
Taking the supremum over all $\varphi$ with $\|\varphi\|_{W^{1,\infty}}\le 1$, we deduce
\begin{align*}
&\lt|\iint_{\R^d \times \R^d} \varphi(x,\xi) (f - M_{\rho, u})\,dxd\xi\rt| \cr
&\quad \leq C\lt(\iint_{\R^d \times \R^d} \frac1{f}|\nabla_\xi f + (\xi - u_f) f|^2\,dxd\xi\rt)^\frac12 + C \lt( \intr \rho_f |u_f - u|^2\,dx \rt)^{\frac12}   + C {\rm d}_{\rm BL} (\rho_f,\rho),
\end{align*}
and subsequently,
\[
{\rm d}^2_{\rm BL}(f, M_{\rho, u}) \leq C\iint_{\R^d \times \R^d} \frac1{f}|\nabla_\xi f + (\xi - u_f) f|^2\,dxd\xi + C \intr \rho_f |u_f - u|^2\,dx + C {\rm d}^2_{\rm BL} (\rho_f,\rho),
\]
where $C>0$ depends only on $\|\varphi\|_{W^{1,\infty}}$, $\|\nabla u\|_{L^\infty}$,  and  $d$.   We finally combine this with Lemma \ref{lem_gd} to conclude 
\begin{align*}
\int_0^t {\rm d}^2_{\rm BL}(f, M_{\rho, u})\,ds &\leq C  {\rm d}^2_{\rm BL}(\rho_{f_0}, \rho_0) + C\int_0^t\iint_{\R^d \times \R^d} \frac1{f}|\nabla_\xi f + (\xi - u_f) f|^2\,dxd\xi ds \cr
&\quad + C\int_0^t \intr \rho_f |u_f - u|^2\,dxds.
\end{align*}
This completes the proof.
\end{proof}

The BL estimate above can be complemented by a negative Sobolev estimate, obtained by testing the same decomposition against $H^\sigma$ functions rather than bounded Lipschitz test functions.
 
 \begin{corollary}\label{cor_gd2}
Under the assumptions of Lemma \ref{lem_gd2}, for each $\sigma>d$, there exists $C>0$ depending only on $d$ and $T$ such that for all $t\in[0,T]$,
\begin{align*}
\int_0^t \|f-M_{\rho, u}\|_{H_{x,\xi}^{-\sigma}}^2\,ds&\le   C\int_0^t\iint_{\R^d \times \R^d} \frac1{f}|\nabla_\xi f + (\xi-u_f) f|^2\,dxd\xi ds +C\int_0^t \intr \rho_f |u_f - u|^2\,dxds \cr
&\quad  + C\|u\|_{L^\infty((0,T)\times\R^d)}^2 \int_0^t \|\rho_f - \rho\|_{L^1}^2\,ds  +C\int_0^t \|\rho_f - \rho\|_{H^{-\sigma}}^2\,ds.
\end{align*}
In particular, if $u=0$, then
\[
\int_0^t \|f-M_{\rho, 0}\|_{H_{x,\xi}^{-\sigma}}^2\,ds \le   C\int_0^t\iint_{\R^d \times \R^d} \frac1{f}|\nabla_\xi f + \xi f|^2\,dxd\xi ds   + C\int_0^t \|\rho_f - \rho\|_{H^{-\sigma}}^2\,ds.
\]
\end{corollary}

\begin{proof}
We follow the decomposition used in the proof of Lemma \ref{lem_gd2}, namely,
\[
\iint_{\R^d\times\R^d} \varphi(x,\xi)\bigl(f-M_{\rho,u}\bigr)\,dxd\xi = I+II+III,
\]
where
\[
I:=\iint_{\R^d\times\R^d} \varphi (f-M_{\rho_f,u_f})\,dxd\xi,\quad II:=\iint_{\R^d\times\R^d} \varphi (M_{\rho_f,u_f}-M_{\rho_f,u})\,dxd\xi,
\]
and
\[
III:=\iint_{\R^d\times\R^d}\varphi (M_{\rho_f,u}-M_{\rho,u})\,dxd\xi.
\]

Fix $\varphi\in H^\sigma(\R^d\times\R^d)$ with $\sigma>d$. Since $H^\sigma(\R^d\times\R^d)\hookrightarrow L^\infty(\R^d\times\R^d)$, the estimates for $I$ and $II$ obtained in Lemma \ref{lem_gd2} remain valid with $\|\varphi\|_{L^\infty}$ replaced by $C\|\varphi\|_{H^\sigma}$. Thus,
\[
|I| \le C\|\varphi\|_{H^\sigma} \lt(\iint_{\R^d\times\R^d}\frac1f |\nabla_\xi f+(\xi-u_f)f|^2\,dxd\xi\rt)^{1/2} 
\]
and
\[
|II|
\le C\|\varphi\|_{H^\sigma} \lt(\int_{\R^d}\rho_f |u_f-u|^2\,dx\rt)^{1/2}.
\]

For $III$, we write
\[\begin{aligned}
III
&=\iint_{\R^d\times\R^d} \varphi(x,\xi)(\rho_f-\rho)(x) ( M_{1,u(x)}(\xi)-M_{1,0}(\xi) )\,dxd\xi  + \iint_{\R^d\times\R^d}\varphi(x,\xi)(\rho_f-\rho)(x)M_{1,0}(\xi)\,dxd\xi.
\end{aligned}\]
Using the mean-value theorem in $u$, we have
\[
|M_{1,u}-M_{1,0}|\le C|u|  (1+|\xi|)M_{1,\theta u}
\]
for some $\theta\in(0,1)$, and hence
\[
\lt|\iint_{\R^d\times\R^d} \varphi(\rho_f-\rho)(M_{1,u}-M_{1,0})\,dxd\xi\rt| \le C\|\varphi\|_{H^\sigma}\|u\|_{L^\infty}\|\rho_f-\rho\|_{L^1}.
\]
For the second term, by duality between $H_x^{-\sigma}$ and $H_x^\sigma$,
\[
\lt|\iint_{\R^d\times\R^d} \varphi(x,\xi)(\rho_f-\rho)(x)M_{1,0}(\xi)\,dxd\xi\rt| \le \|\rho_f-\rho\|_{H_x^{-\sigma}} \lt\|\int_{\R^d}\varphi(\cdot,\xi)M_{1,0}(\xi)\,d\xi\rt\|_{H_x^\sigma}.
\]
Moreover, by Minkowski's inequality and Cauchy--Schwarz inequality,
\[\begin{aligned}
\lt\|\int_{\R^d}\varphi(\cdot,\xi)M_{1,0}(\xi)\,d\xi\rt\|_{H_x^\sigma}
&= \lt\|\int_{\R^d}(1-\Delta_x)^{\sigma/2}\varphi(\cdot,\xi)M_{1,0}(\xi)\,d\xi\rt\|_{L_x^2}\\
&\le \int_{\R^d}\|(1-\Delta_x)^{\sigma/2}\varphi(\cdot,\xi)\|_{L_x^2}M_{1,0}(\xi)\,d\xi\\
&\le C\|\varphi\|_{H^\sigma(\R^d\times\R^d)}.
\end{aligned}\]
Hence, we have
\[
|III| \le C\|\varphi\|_{H^\sigma}\|u\|_{L^\infty}\|\rho_f-\rho\|_{L^1} + C\|\varphi\|_{H^\sigma}\|\rho_f-\rho\|_{H_x^{-\sigma}}.
\]

Combining the bounds for $I$, $II$, and $III$, and taking the supremum over all $\varphi$ with $\|\varphi\|_{H^\sigma}\le 1$, we obtain
\[\begin{aligned}
\|f-M_{\rho,u}\|_{H_{x,\xi}^{-\sigma}}^2 &\le C\iint_{\R^d \times \R^d}\frac1f |\nabla_\xi f + (\xi-u_f)f|^2\,dxd\xi + C\int_{\R^d}\rho_f |u_f-u|^2\,dx \\
&\quad + C\|u\|_{L^\infty((0,T)\times\R^d)}^2\|\rho_f-\rho\|_{L^1}^2 + C\|\rho_f-\rho\|_{H_x^{-\sigma}}^2.
\end{aligned}\]
Integrating over $(0,t)$ yields the desired estimate. The case $u=0$ follows immediately.
\end{proof}

%
%
%
%
%
\subsection{Modulated energy for the Riesz interaction} A central ingredient in the relative entropy framework is the control of the nonlocal interaction term. To this end, we employ modulated interaction energies, which provide a quantitative way of comparing two densities evolving under continuity equations. Such estimates were developed in \cite[Lemma 2.1]{CJe21-2} (see also \cite{CJ24, NRS22}), and we recall the version adapted to our setting below.

\begin{proposition}\label{prop_mod}
Let $T>0$, $\alpha\in(0,1]$, and let $(\rho,u)$ and $(n,v)$ be sufficiently regular solutions of the continuity equations
\[
{\rm A} \partial_t \rho + {\rm B} \nabla \cdot (\rho  u) =0 \quad \mbox{and} \quad  {\rm A} \partial_t n + {\rm B} \nabla \cdot (nv) =0,
\]
respectively. Then, for $t\in[0,T)$, the following inequality holds:
\begin{align*}
- {\rm B} \intr \rho(u-v) \cdot \nabla (-\Delta)^{-\alpha}  (\rho - n)\,dx  &\leq - \frac{\rm A}2\frac{d}{dt}\intr (\rho -n) (-\Delta)^{-\alpha} (\rho - n)\,dx \cr
&\quad +  C{\rm B}\intr (\rho -n) (-\Delta)^{-\alpha} (\rho - n)\,dx,
\end{align*}
where the constant $C>0$ depends on $\|\nabla v\|_{L^\infty((0,T) \times \R^d)}$.  
\end{proposition}

\begin{remark}
In the macroscopic limits we apply Proposition \ref{prop_mod} with $(\rho,u)=(\rho_f,u_f)$ and $(n,v)=(\rho,u)$, where $u$ is the velocity field associated with the limiting density $\rho$. The regularity requirement falls only on $v=u$, which is assumed classical in our main theorems; no gradient bound on $u_f$ is needed.
\end{remark}

\begin{remark}
The key observation behind Proposition \ref{prop_mod} is the identity
\begin{align*}
&- \frac{\rm A}2\frac{d}{dt}\intr (\rho -n) (-\Delta)^{-\alpha} (\rho - n)\,dx \cr
&\quad = - {\rm B}\intr \rho (u - v) \cdot \nabla (-\Delta)^{-\alpha} (\rho - n)\,dx  - {\rm B}\intr (\rho  -n) v \cdot \nabla (-\Delta)^{-\alpha} (\rho - n)\,dx.
\end{align*}
so that the desired estimate reduces to bounding the last term on the right. This requires a regularity assumption on $v$ ensuring that 
\[
\intr (\rho  -n) v \cdot \nabla (-\Delta)^{-\alpha} (\rho - n)\,dx \leq C\intr (\rho -n) (-\Delta)^{-\alpha} (\rho - n)\,dx 
\]
with a constant $C$ depending only on $v$. 
\end{remark}

%
%
%
%
%
%
\subsection{Relative entropy estimate}
The last stability tool we develop is a relative entropy estimate for general kinetic equations. 
While the previous subsections established weak stability bounds in terms of BL distances and modulated energies, the relative entropy method provides a stronger framework that directly quantifies the measure of discrepancy between $f$ and its associated local Maxwellian $M_{\rho,u}$. 
Crucially, the dissipation structure inherent in the entropy identity produces coercive terms with scaling prefactors of order $\frac1\e$ or $\frac1{\e^2}$, 
which play a central role in handling mildly well-prepared initial data. 
This makes the relative entropy approach essential for converting microscopic dissipation into strong convergence toward the macroscopic system.

To keep the presentation general, we consider the following kinetic equation:
\bq\label{eq:gkin}
{\rm A} \pa_t f + {\rm B} \xi\cdot \nabla f  +\nabla_\xi \cdot (F[f] f) = \frac1\tau\nabla_\xi \cdot (\nabla_\xi f + \xi  f),
\eq
where ${\rm A}$ and ${\rm B}$ are positive constants, and $F[f]$ is a general force field.

The corresponding macroscopic system reads
\begin{align}\label{eq:ghydro}
\begin{aligned}
&{\rm A}\pa_t \rho + {\rm B} \nabla \cdot (\rho u) = 0,\cr
&{\rm A} \pa_t (\rho u) + {\rm B} \nabla \cdot (\rho u \otimes u) + {\rm B} \nabla \rho =  \rho R[\rho, u],
\end{aligned}
\end{align}
with a source term $R[\rho,u]$ reflecting the macroscopic effect of $F[f]$.

The next proposition establishes a relative entropy identity and its dissipative structure. 
It shows how the time derivative of $\mathscr{H}[f|M_{\rho,u}]$ can be decomposed into coercive terms, error terms, and force contributions. 
This identity provides the analytic backbone for all quantitative convergence results proved in later sections.

\begin{proposition}\label{prop_key} Let $f$ and $(\rho, u)$ be sufficiently regular solutions to the equations \eqref{eq:gkin} and \eqref{eq:ghydro}, respectively. Then we have
\begin{align*}
& \frac{d}{dt}\iint_{\R^d \times \R^d}  f \log \lt( \frac{f}{ M_{\rho, u}} \rt)  dxd\xi + \frac1{{\rm A} \tau} \iint_{\R^d \times \R^d} \frac1{f} |\nabla_\xi f + (\xi - u_f)f|^2 dxd\xi \cr
&\quad =   - \frac{\rm B}{{\rm A}} \iint_{\R^d \times \R^d}  f \lt((u-\xi) \otimes (u-\xi)  -  \mathbb{I}_d \rt): \nabla u\,dxd\xi  - \frac1{{\rm A} \tau} \intr  \rho_f u_f \cdot (u_f - u)\,dx\cr
&\qquad  - \frac1{{\rm A}}  \iint_{\R^d \times \R^d} (\nabla_\xi \cdot F[f]) f \,dxd\xi  - \frac1{{\rm A}} \iint_{\R^d \times \R^d}  f (u-\xi) \cdot \lt(F[f] - R[\rho, u]\rt)  dxd\xi.
\end{align*}
Moreover, if $f \in L^1_2(\R^d \times \R^d)$ and $\nabla u \in L^\infty(\R^d)$, then
 \begin{align*}
&\frac{d}{dt}\iint_{\R^d \times \R^d} f \log \lt( \frac{f}{ M_{\rho, u}} \rt)  dxd\xi + \frac1{2\tau {\rm A}} \iint_{\R^d \times \R^d} \frac1{f} |\nabla_\xi f + (\xi - u_f)f|^2 dxd\xi \cr
&\quad \leq \frac{{\rm B}}{{\rm A}}  \|\nabla u\|_{L^\infty} \intr \rho_f  |u_f - u|^2\,dx  + \frac{{\rm B}^2 \tau}{2{\rm A}} \|\nabla u\|_{L^\infty}^2 \|f\|_{L^1_2}  - \frac1{{\rm A}\tau} \intr  \rho_f u_f \cdot (u_f - u)\,dx    \cr
&\qquad - \frac{1}{{\rm A}}  \iint_{\R^d \times \R^d} (\nabla_\xi \cdot F[f]) f \,dxd\xi   -   \frac1{{\rm A}}\iint_{\R^d \times \R^d}  f (u-\xi) \cdot \lt(F[f] - R[\rho, u]\rt) dxd\xi.
\end{align*}
\end{proposition}

\begin{remark}
The dissipation term
\[
\iint_{\R^d \times \R^d} \frac1{f}\,|\nabla_\xi f + (\xi - u_f)f|^2\,dxd\xi
\]
is the Fisher information relative to the local Maxwellian $M_{\rho_f,u_f}$ (see Lemma \ref{lem:prop}). 
For the nonlinear Fokker--Planck operator $\nabla_\xi \cdot (\nabla_\xi f + (\xi-u_f)f)$ 
it appears naturally in the entropy inequality. 
In our setting, it arises instead 
from measuring the relative entropy with respect to $M_{\rho,u}$, 
and this additional dissipation is a key ingredient in our stability and convergence analysis 
(see Lemma \ref{lem_gd2}).
\end{remark}
 
\begin{proof}[Proof of Proposition \ref{prop_key}] We begin by recalling the identities
\begin{align*}
\nabla_\xi \log M_{\rho, u} &= \frac1{M_{\rho, u}}\nabla_\xi M_{\rho, u} = u-\xi,\cr
\nabla_x \log M_{\rho, u} &= \frac1{M_{\rho, u}}\nabla_x M_{\rho, u} = \nabla_x \log \rho - (\nabla_x u) (u-\xi) .
\end{align*}
The time derivative of the relative entropy can be split into two main contributions:
\begin{align*}
  \frac{d}{dt}\iint_{\R^d \times \R^d} f \log \lt( \frac{f}{ M_{\rho, u}} \rt) dxd\xi &=    \iint_{\R^d \times \R^d} \pa_t f \lt(  \log f - \log M_{\rho, u}\rt)  dxd\xi + \iint_{\R^d \times \R^d} f \lt( \frac{\pa_t f}{f}  - \frac{\pa_t M_{\rho, u}}{M_{\rho, u}}\rt) dxd\xi \cr
&=: I + II.
\end{align*}
By substituting the kinetic equation \eqref{eq:gkin}, we obtain
\begin{align*}
I &= \frac1{{\rm A}} \iint_{\R^d \times \R^d} \lt(  \log f - \log M_{\rho, u}\rt)  \cdot \lt( - {\rm B} \xi \cdot \nabla f  - \nabla_\xi \cdot (F[f] f) + \frac1\tau \nabla_\xi \cdot (\nabla_\xi f + \xi    f)\rt) dxd\xi \cr
&=: I_1 + I_2 + I_3.
\end{align*}
Here each term can be estimated as follows:
\begin{align*}
I_1 &= \frac{\rm B}{\rm A} \iint_{\R^d \times \R^d} \xi f \cdot \lt( \frac{\nabla f}{f}  - \frac{\nabla M_{\rho, u}}{M_{\rho, u}}\rt) dxd\xi \cr
&= -\frac{\rm B}{\rm A}  \iint_{\R^d \times \R^d} \xi f \cdot \lt( \nabla \log \rho - (\nabla u)(u-\xi)\rt) dxd\xi \cr
&= - \frac{\rm B}{\rm A} \intr \rho_f u_f \cdot \nabla \log \rho\,dx + \frac{\rm B}{{\rm A}}  \iint_{\R^d \times \R^d}  (u - \xi) \otimes \xi f : \nabla u \,dxd\xi,
\end{align*}
\begin{align*}
I_2 &= \frac1{{\rm A}}  \iint_{\R^d \times \R^d} F[f] f \cdot \lt( \frac{\nabla_\xi f}{f} - (u-\xi)\rt) dxd\xi \cr
&= - \frac1{{\rm A}}  \iint_{\R^d \times \R^d} (\nabla_\xi \cdot F[f]) f \,dxd\xi -  \frac1{{\rm A}}  \iint_{\R^d \times \R^d} (u-\xi) \cdot F[f] f\,dxd\xi,
\end{align*}
and
\begin{align*}
I_3 &= - \frac1{ {\rm A} \tau} \iint_{\R^d \times \R^d} \lt( \nabla_\xi f  + \xi   f\rt) \cdot \lt( \frac{\nabla_\xi f}{f} - \frac{\nabla_\xi M_{\rho, u}}{M_{\rho, u}}\rt) dxd\xi \cr
&= - \frac1{ {\rm A} \tau} \iint_{\R^d \times \R^d} \lt(\nabla_\xi f  + \xi  f\rt) \cdot \lt( \frac{\nabla_\xi f}{f} -  (u-\xi)\rt) dxd\xi \cr
&= - \frac1{{\rm A} \tau} \iint_{\R^d \times \R^d} \frac1{f} |\nabla_\xi f + (\xi - u_f)f|^2 dxd\xi  - \frac1{{\rm A} \tau} \intr   \rho_f u_f \cdot (u_f - u)\,dx,
\end{align*}
where we used
\[
\intr \lt( \nabla_\xi f + (\xi - u_f )f\rt) d\xi =0.
\]
This implies
\begin{align*}
I &= - \frac{\rm B}{\rm A} \intr \rho_f u_f \cdot \nabla \log \rho\,dx + \frac{\rm B}{{\rm A}}  \iint_{\R^d \times \R^d} \xi \otimes (u - \xi) f : \nabla u \,dxd\xi  \cr
&\quad -  \frac1{{\rm A} \tau} \iint_{\R^d \times \R^d} \frac1{f} |\nabla_\xi f + (\xi - u_f)f|^2 dxd\xi    - \frac1{{\rm A} \tau}  \intr   \rho_f u_f \cdot (u_f - u)\,dx\cr
&\quad - \frac1{{\rm A}}  \iint_{\R^d \times \R^d} (\nabla_\xi \cdot F[f] )f \,dxd\xi -  \frac1{{\rm A}}  \iint_{\R^d \times \R^d}  f (u-\xi) \cdot F[f] \,dxd\xi.
\end{align*}
For $II$, we get
\[
II = - \iint_{\R^d \times \R^d} f \lt( \frac{\pa_t \rho}{\rho}  - (u-\xi)\cdot \pa_t u\rt) dxd\xi =: II_1 + II_2,
\]
where
\[
II_1 = \frac{\rm B}{\rm A} \intr \frac{\rho_f}{\rho} \nabla \cdot (\rho u)\,dx  =  \frac{\rm B}{\rm A}  \intr \rho_f u \cdot \nabla \log \rho \,dx + \frac{\rm B}{\rm A}  \intr\rho_f \mathbb{I}_d : \nabla  u \,dx
\]
and
\begin{align*}
II_2 &=\frac1{{\rm A}} \iint_{\R^d \times \R^d} f (u-\xi) \cdot \lt( - {\rm B}(u \cdot \nabla) u - {\rm B} \nabla \log \rho  + R[\rho, u]\rt) dx d\xi \cr
&= - \frac{\rm B}{{\rm A}} \iint_{\R^d \times \R^d}   (u-\xi)\otimes u  f : \nabla u \,dxd\xi - \frac{\rm B}{{\rm A}} \into \rho_f (u - u_f) \cdot \nabla \log \rho\,dx \cr
& \quad + \frac1{{\rm A}} \iint_{\R^d \times \R^d} f(u-\xi) \cdot R[\rho, u]\,dxd\xi.
\end{align*}
Combining $II_1$ and $II_2$ gives
\begin{align*}
II &=  \frac{\rm B}{\rm A}  \intr \rho_f u_f \cdot \nabla \log \rho \,dx + \frac{\rm B}{\rm A}  \intr \rho_f \mathbb{I}_d : \nabla  u \,dx - \frac{\rm B}{{\rm A}}\iint_{\R^d \times \R^d}  (u-\xi)\otimes u   f : \nabla u \,dxd\xi \cr
& \quad + \frac1{{\rm A}} \iint_{\R^d \times \R^d} f(u-\xi) \cdot R[\rho, u]\,dxd\xi.
\end{align*}
Adding the contributions from $I$ and $II$ yields the exact identity in the first part of the proposition.

To obtain the inequality, we further expand
\[
\intr f (u-\xi) \otimes (u-\xi)\,d\xi =  \rho_f (u - u_f) \otimes (u - u_f) - \intr f (u_f - \xi) \otimes \xi\,d\xi,
\]
and this deduces
\begin{align*}
&- \iint_{\R^d \times \R^d}  \lt( f (u-\xi) \otimes (u-\xi)  - \rho_f \mathbb{I}_d \rt): \nabla  u\,dxd\xi \cr
&\quad = - \intr \rho_f (u - u_f) \otimes (u - u_f) : \nabla u\,dx   + \intr\lt(  \intr \lt( f (u_f - \xi) - \nabla_\xi f \rt) \otimes \xi\,d\xi \rt) : \nabla u \,dx.
\end{align*}
Note that the right-hand side of the above can be bounded as
\[
\lt|\intr \rho_f (u - u_f) \otimes (u - u_f) : \nabla u\,dx\rt| \leq \|\nabla u\|_{L^\infty} \intr \rho_f  |u_f - u|^2\,dx 
\]
and
\begin{align*}
&\lt| \intr \lt(  \intr \lt( f (u_f - \xi) - \nabla_\xi f \rt) \otimes \xi\,d\xi \rt) : \nabla u\,dx \rt| \cr
&\quad \leq \|\nabla u\|_{L^\infty}\lt(\iint_{\R^d \times \R^d} |\xi|^2 f\,dxd\xi \rt)^\frac12 \lt( \iint_{\R^d \times \R^d} \frac1{f} |\nabla_\xi f + (\xi - u_f)f|^2\,dxd\xi\rt)^\frac12\cr
&\quad \leq \frac{{\rm B} \tau }{2}\|\nabla u\|_{L^\infty}^2 \|f\|_{L^1_2} + \frac{1}{2{\rm B} \tau} \iint_{\R^d \times \R^d} \frac1{f} |\nabla_\xi f + (\xi - u_f)f|^2\,dxd\xi.
\end{align*}
Hence we have
 \begin{align*}
& \frac{d}{dt}\iint_{\R^d \times \R^d} f \log \lt( \frac{f}{ M_{\rho, u}} \rt)  dxd\xi + \frac1{2\tau {\rm A}} \iint_{\R^d \times \R^d} \frac1{f} |\nabla_\xi f + (\xi - u_f)f|^2 dxd\xi \cr
&\quad \leq \frac{{\rm B}}{{\rm A}} \|\nabla u\|_{L^\infty} \intr \rho_f  |u_f - u|^2\,dx  + \frac{{\rm B}^2 \tau}{2{\rm A}} \|\nabla u\|_{L^\infty}^2 \|f\|_{L^1_2}  - \frac1{{\rm A}\tau} \intr   \rho_f u_f \cdot (u_f - u)\,dx  \cr
&\qquad - \frac{1}{{\rm A}}  \iint_{\R^d \times \R^d} (\nabla_\xi \cdot F[f]) f \,dxd\xi  - \frac1{{\rm A}}\iint_{\R^d \times \R^d}  f (u-\xi) \cdot \lt(F[f] - R[\rho, u]\rt) dxd\xi.
\end{align*}
This completes the proof.
\end{proof}

%
%
%
%
%
%

\section{Quantitative limits: Diffusive, High-Field, and gSQG}\label{sec_vfp}

We now apply the quantitative stability tools developed in Section \ref{sec_qs} to three hydrodynamic regimes of the kinetic model: the diffusive, high-field, and gSQG limits. Each regime corresponds to a different balance between transport, interaction, and relaxation, yet all can be analyzed within the same entropy-dissipation framework. The common strategy is to reformulate the limit equation in conservative form so that it fits into the structure of Proposition \ref{prop_key}, thereby enabling a comparison with the kinetic dynamics via relative entropy. In this way, the relative entropy method is coupled with modulated interaction energies and, when appropriate, bounded Lipschitz estimates, so that microscopic dissipation can be transferred into quantitative control of macroscopic variables. Depending on the regime, this leads either to strong convergence estimates under well-prepared data or to quantitative stability in weaker topologies adapted to the corresponding limiting dynamics.
 
Although the diffusive, high-field, and gSQG limits highlight different balances between transport, interaction, and relaxation, the analytical approach remains uniform: entropy dissipation amplified at the kinetic scale, complemented by weak stability tools, and combined to establish quantitative convergence from kinetic to macroscopic descriptions.

%
%
%
%
%
%
\subsection{Diffusive limit}
We begin with the diffusive regime, corresponding to the parabolic scaling of the kinetic model. 
 
Recall that $f^\e$ solves
\bq\label{prelim_diff}
\e\pa_t f^\e + \xi\cdot \nabla_x f^\e - \nabla (-\Delta)^{-\alpha} \rho^\e  \cdot \nabla_\xi  f^\e = \frac1\e\nabla_\xi \cdot (\nabla_\xi f^\e + \xi f^\e), \quad t>0, \ (x,\xi) \in \R^d \times \R^d.
\eq
Formally, as stated in Section \ref{sssec_diff}, one obtains the nonlinear diffusion equation as $\e \to 0$:
\bq\label{lim_diff}
\pa_t \rho - \nabla \cdot (\rho \nabla (-\Delta)^{-\alpha} \rho) =  \Delta \rho,
\eq
which combines local diffusion with aggregation effects generated by the Riesz potential.

To connect \eqref{prelim_diff} with \eqref{lim_diff}, it is convenient to rewrite the system in moment form:
\begin{align}\label{lim_diff2}
\begin{aligned}
&\e \pa_t \rho +  \nabla \cdot (\rho  {\rm u_\e}) =0, \cr
&\e \pa_t (\rho  {\rm u_\e} ) +  \nabla \cdot ( \rho {\rm u_\e} \otimes {\rm u_\e}) +  \nabla  \rho = - \rho \nabla (-\Delta)^{-\alpha} \rho - \frac1{\e} \rho {\rm u_\e} + \e\rho {\rm e}_\e,
\end{aligned}
\end{align}
where
\bq\label{u_diff}
\rho {\rm e}_\e = \rho \pa_t {\rm u_\e} + \frac1\e \rho {\rm u_\e} \cdot \nabla {\rm u_\e}, \quad {\rm u_\e} = - \e \lt(\nabla (-\Delta)^{-\alpha} \rho + \nabla \log \rho \rt).
\eq
This representation makes explicit the dissipative drag term $-\tfrac1\e \rho {\rm u_\e}$, which enforces the fast relaxation toward the macroscopic closure \eqref{lim_diff}.

It follows from Lemma \ref{lem_en} that we obtain the entropy and kinetic energy bound estimates.
\begin{lemma}\label{en_diffu} Let $T>0$ and $f^\e$ be a weak entropy solution to the equation \eqref{prelim_diff} in the sense of Definition \ref{def_weak} on the time interval $[0,T]$. Then we have
\[
\sup_{0 \leq t \leq T}\lt(\mathscr{F}[f^\e(t)] + \mathscr{P}[\rho^\e(t)]\rt)  + \frac1{\e^2}\int_0^T\iint_{\R^d \times \R^d} \frac1{f^\e}|\nabla_\xi f^\e + \xi f^\e|^2\,dxd\xi dt \leq \mathscr{F}[f_0^\e] + \mathscr{P}[\rho^\e_0]
\]
and
\[
\int_0^T \mathscr{K}[f^\e(t)]\,dt \leq dT + \e^2 \lt(\mathscr{K}[f^\e_0]+ \mathscr{P}[\rho_0^\e]\rt).
\]
\end{lemma}

Finally, combining the relative entropy inequality of Proposition \ref{prop_key} with the modulated interaction energy estimate of Proposition \ref{prop_mod} yields the following quantitative bound.

\begin{proposition}\label{re_diffu} Let $T>0$ and $f^\e$ be a weak entropy solution to the equation \eqref{prelim_diff} in the sense of Definition \ref{def_weak} on the time interval $[0,T]$,  and let $(\rho, {\rm u_\e})$ be a classical solutions to the equation \eqref{lim_diff2} with $\nabla {\rm u_\e} \in L^\infty((0,T) \times \R^d)$. Then we have
 \begin{align*}
& \sup_{0 \leq t \leq T}\lt( \mathscr{H}[f^\e | M_{\rho, {\rm u_\e}}](t) + \mathscr{P}[\rho^\e | \rho](t) \rt)\cr
&\quad + \frac1{2\e^2} \int_0^T \iint_{\R^d \times \R^d} \frac1{f^\e} |\nabla_\xi f^\e + (\xi - u^\e)f^\e|^2 dxd\xi dt + \frac1{2\e} \lt( \frac1\e - 2\|\nabla {\rm u_\e}\|_{L^\infty} \rt)\int_0^T \intr \rho^\e |u^\e - {\rm u_\e}|^2\,dxdt \cr
&\qquad \leq \mathscr{H}[f^\e | M_{\rho, {\rm u_\e}}](0) + \mathscr{P}[\rho^\e | \rho](0) + C\e^2 \int_0^T \mathscr{K}[f^\e(t)]\,dt  + C\e^3 + C\int_0^T \mathscr{P}[\rho^\e | \rho](t)\,dt,
\end{align*}
where $C>0$ is independent of $\e > 0$.
\end{proposition}
\begin{proof} Although Proposition \ref{prop_key} is stated for sufficiently regular solutions,
its time-integrated form can also be justified for weak entropy solutions
by a standard approximation argument. 
Indeed, starting from smooth initial data one applies the relative entropy identity,
then passes to the limit along an approximating sequence, using the entropy inequality
and lower semicontinuity of the dissipation terms. 
This yields the validity of Proposition \ref{prop_key} in the integral inequality form 
for any weak entropy solution $f^\e$ in the sense of Definition \ref{def_weak} (see e.g., \cite{BV25, GM10, SR04, WLL15}).
We may therefore apply this estimate with ${\rm A}=\e$, ${\rm B}=1$, $\tau=\e$, 
\[
F[f^\e] = - \nabla (-\Delta)^{-\alpha} \rho^\e \quad \mbox{and} \quad R[\rho, {\rm u_\e}] =  -    \nabla (-\Delta)^{-\alpha} \rho - \frac1{\e} {\rm u_\e} + \e {\rm e}_\e,
\]
where ${\rm u_\e}$ and ${\rm e}_\e$ are given in \eqref{u_diff}. This gives
  \begin{align*}
& \iint_{\R^d \times \R^d} f^\e \log \lt( \frac{f^\e}{ M_{\rho, {\rm u_\e}}} \rt) dxd\xi + \frac1{2\e^2} \int_0^t \iint_{\R^d \times \R^d} \frac1{f^\e} |\nabla_\xi f^\e + (\xi - u^\e)f^\e|^2 dxd\xi ds \cr
&\quad \leq  \iint_{\R^d \times \R^d} f^\e_0 \log \lt( \frac{f^\e_0}{ M_{\rho_0, {\rm u_{\e0}}}} \rt) dxd\xi +  \frac{\|\nabla {\rm u_\e}\|_{L^\infty}}{\e} \int_0^t\intr \rho^\e |u^\e - {\rm u_\e}|^2\,dx ds \cr
&\qquad  +  \frac12 \|\nabla {\rm u_\e}\|_{L^\infty}^2 \int_0^t \|f^\e(s)\|_{L^1_2}\,ds   - \frac1{\e^2} \int_0^t\intr \rho^\e u^\e \cdot (u^\e - {\rm u_\e})\,dx ds \cr
&\qquad +  \frac1\e \int_0^t\iint_{\R^d \times \R^d}  f^\e ({\rm u_\e}-\xi) \cdot \lt(\nabla (-\Delta)^{-\alpha} (\rho^\e - \rho)   - \frac1{\e} {\rm u_\e} + \e {\rm e}_\e\rt) dxd\xi ds \cr
&\quad =  \iint_{\R^d \times \R^d} f^\e_0 \log \lt( \frac{f^\e_0}{ M_{\rho_0, {\rm u_{\e0}}}} \rt) dxd\xi+  \frac{\|\nabla {\rm u_\e}\|_{L^\infty}}{\e} \int_0^t\intr \rho^\e |u^\e - {\rm u_\e}|^2\,dx ds \cr
&\qquad +  \frac{\|\nabla {\rm u_\e}\|_{L^\infty}^2}{2} \int_0^t\|f^\e(s)\|_{L^1_2}\,ds  - \frac1{\e^2}\int_0^t\intr \rho^\e |u^\e - {\rm u_\e}|^2\,dx ds \cr
&\qquad - \int_0^t\intr \rho^\e(u^\e - {\rm u_\e})\cdot {\rm e}_\e\,dx ds - \frac1\e \int_0^t\intr  \rho^\e (u^\e- {\rm u_\e}) \cdot \nabla (-\Delta)^{-\alpha}(\rho^\e  - \rho) \,dx ds.
\end{align*}
From \eqref{u_diff} we know
\[
 \|\nabla {\rm u_\e}\|_{L^\infty} \leq \e \|\nabla^2 (-\Delta)^{-\alpha} \rho + \nabla^2 \log \rho\|_{L^\infty} \quad \mbox{and} \quad  \e \| {\rm e}_\e\|_{L^\infty} \leq \|  \e\pa_t {\rm u_\e} +   {\rm u_\e} \cdot \nabla {\rm u_\e}\|_{L^\infty} \leq C\e^2
\]
for some $C>0$ independent of $\e > 0$. Thus, we obtain
\[
\lt|\intr \rho^\e ({\rm u_\e} - u^\e) \cdot {\rm e}_\e\,dx\rt| \leq \frac1{2\e^2}\intr \rho^\e |u^\e - {\rm u_\e}|^2\,dx + \frac{\e^2}2 \| {\rm e}_\e\|_{L^\infty}\|\rho^\e\|_{L^1} \leq \frac1{2\e^2}\intr \rho^\e |u^\e - {\rm u_\e}|^2\,dx + C\e^3.
\]
For the nonlocal term, Proposition \ref{prop_mod} with ${\rm A}=\e$, ${\rm B}=1$ gives
\[
-\frac1{\e}\int_0^t \into \rho^\e (u^\e-{\rm u_\e})\cdot \nabla(-\Delta)^{-\alpha}(\rho^\e-\rho)\,dxds
\leq - \mathscr{P}[\rho^\e|\rho](t) + \mathscr{P}[\rho^\e|\rho](0) + C\int_0^t \mathscr{P}[\rho^\e|\rho](s)\,ds.
\]
Combining these bounds concludes the desired result.
\end{proof}

%
%
%
%
%
%
\subsubsection{Proof of Theorem \ref{thm_kin1}: diffusive limit}\label{ssec_diff} From Proposition \ref{re_diffu}, for $\e>0$ sufficiently small so that $4\e \|\nabla {\rm u_\e}\|_{L^\infty}<1$, we obtain
 \begin{align*}
& \sup_{0 \leq t \leq T}\lt( \mathscr{H}[f^\e | M_{\rho, {\rm u_\e}}](t) + \mathscr{P}[\rho^\e | \rho](t) \rt)\cr
&\quad + \frac1{2\e^2} \int_0^T \iint_{\R^d \times \R^d} \frac1{f^\e} |\nabla_\xi f^\e + (\xi - u^\e)f^\e|^2 dxd\xi dt +  \frac1{4\e^2}  \int_0^T \intr \rho^\e |u^\e - {\rm u_\e}|^2\,dxdt \cr
&\qquad \leq \mathscr{H}[f^\e | M_{\rho, {\rm u_\e}}](0) + \mathscr{P}[\rho^\e | \rho](0) + C\e^2 \int_0^T \mathscr{K}[f^\e(t)]\,dt  + C\e^3 + C\int_0^T \mathscr{P}[\rho^\e | \rho](t)\,dt,
\end{align*}
and thus applying Gr\"onwall's lemma yields
\begin{align}\label{con_diff}
\begin{aligned}
&\sup_{0 \leq t \leq T}\lt( \mathscr{H}[f^\e | M_{\rho, {\rm u_\e}}](t) + \mathscr{P}[\rho^\e | \rho](t) \rt) \cr
&\quad + \frac1{2\e^2}\int_0^T \iint_{\R^d \times \R^d} \frac1{f^\e} |\nabla_\xi f^\e + (\xi - u^\e)f^\e|^2 dxd\xi dt + \frac1{4\e^2} \int_0^T \intr \rho^\e |u^\e - {\rm u_\e}|^2\,dxdt \cr
&\qquad \leq C\lt( \mathscr{H}[f^\e | M_{\rho, {\rm u_\e}}](0) + \mathscr{P}[\rho^\e | \rho](0)\rt) + C\e^2 + C\e^2 \int_0^T \mathscr{K}[f^\e(t)]\, dt\cr
&\qquad \leq C\lt( \mathscr{H}[f^\e | M_{\rho, {\rm u_\e}}](0) + \mathscr{P}[\rho^\e | \rho](0)\rt) + C\e^2,
\end{aligned}
\end{align}
due to Lemma \ref{en_diffu}.

Next we compare $\mathscr{H}[f^\e|M_{\rho,0}]$ and $\mathscr{H}[f^\e|M_{\rho,{\rm u_\e}}]$. Since
\begin{align*}
\iint_{\R^d \times \R^d}  f^\e \log \frac{M_{\rho, {\rm u_\e}}}{M_{\rho, 0}}\,dxd\xi &= \frac{1}{2} \iint_{\R^d \times \R^d} f^\e \lt(2 {\rm u_\e} \cdot \xi - |{\rm u_\e}|^2\rt)dxd\xi \cr
&= \frac1{2} \intr \rho^\e (2 {\rm u_\e} \cdot u^\e - |{\rm u_\e}|^2)\,dx \cr
&= -\intr \rho^\e {\rm u_\e} \cdot (u^\e - {\rm u_\e})\,dx + \frac1{2}\intr \rho^\e |{\rm u_\e}|^2\,dx,
\end{align*}
we get
\begin{align*}
\lt|\mathscr{H}[f^\e | M_{\rho, 0}] - \mathscr{H}[f^\e | M_{\rho, {\rm u_\e}}]\rt| &=  \lt|\iint_{\R^d \times \R^d}  f^\e \log \frac{M_{\rho, {\rm u_\e}}}{M_{\rho, 0}}\,dxd\xi\rt| \cr
  & \leq 2\intr \rho^\e|{\rm u_\e}|^2\,dx + \frac1{4}\intr \rho^\e |u^\e - {\rm u_\e}|^2\,dx\cr
  &\leq C\e^2 + \frac12\mathscr{H}[f^\e | M_{\rho, {\rm u_\e}}].
\end{align*}
due to \eqref{u_diff}. Hence the two entropies are equivalent up to $O(\e^2)$, namely
\bq\label{rel_hh}
\mathscr{H}[f^\e | M_{\rho, {\rm u_\e}}] \leq 2 \mathscr{H}[f^\e | M_{\rho, 0}] + C\e^2
\eq
and
\[
\mathscr{H}[f^\e | M_{\rho, 0}] \leq 2 \mathscr{H}[f^\e | M_{\rho, {\rm u_\e}}] + C\e^2.
\]
Combining this with \eqref{con_diff} gives
\begin{align}\label{con_diff2}
\begin{aligned}
&\sup_{0 \leq t \leq T}\lt( \mathscr{H}[f^\e | M_{\rho, 0}](t) + \mathscr{P}[\rho^\e | \rho](t) \rt) \cr
&\quad + \frac1{\e^2}\int_0^T \iint_{\R^d \times \R^d} \frac1{f^\e} |\nabla_\xi f^\e + (\xi - u^\e)f^\e|^2 dxd\xi dt + \frac1{\e^2} \int_0^T \intr \rho^\e |u^\e - {\rm u_\e}|^2\,dxdt \cr
&\qquad \leq C\lt( \mathscr{H}[f^\e | M_{\rho, 0}](0) + \mathscr{P}[\rho^\e | \rho](0)\rt) + C\e^2
\end{aligned}
\end{align}
for some $C>0$ independent of $\e$.
 
 Estimate \eqref{con_diff2}, together with Lemma \ref{lem:prop}, ensures the convergence of $f^\e$ and $\rho^\e$. For the momentum, we note
 \begin{align*}
 \intr | \rho^\e u^\e - \rho {\rm u_\e} | \,dx &\leq   \intr \rho^\e | u^\e - {\rm u_\e}|\, dx +   \intr |\rho^\e - \rho| |{\rm u_\e}|\, dx \cr
 &\leq \lt(  \intr \rho^\e | u^\e - {\rm u_\e}|^2\, dx\rt)^\frac12 + \e\|\nabla (-\Delta)^{-\alpha} \rho + \nabla \log \rho \|_{L^\infty} \|\rho^\e - \rho\|_{L^1},
 \end{align*}
 and thus
 \[
\frac1\e\int_0^T\intr | \rho^\e u^\e - \rho {\rm u_\e} | \,dxdt \leq C \lt(  \frac1{\e^2}\int_0^T\intr \rho^\e | u^\e - {\rm u_\e}|^2\, dxdt\rt)^\frac12 + C\sup_{0 \leq t \leq T} \|(\rho^\e - \rho)(t)\|_{L^1}.
 \]
Using \eqref{con_diff} and the convergence of $\rho^\e$, this implies
\[
\frac1\e\int_0^T\intr | \rho^\e u^\e - \rho {\rm u_\e} | \,dxdt \leq C\e + C\lt( \mathscr{H}[f^\e | M_{\rho, 0}](0) + \mathscr{P}[\rho^\e | \rho](0)\rt)^\frac12.
\]
 Hence, we have
 \[
 \rho^\e \frac{u^\e}\e \to -\rho\lt(\nabla (-\Delta)^{-\alpha} \rho + \nabla \log \rho \rt) \quad \mbox{in } L^1((0,T) \times \R^d).
 \]

We next derive the sharper low-topology estimate for $\rho^\e$, which is consistent with the optimal diffusive scaling suggested by the formal expansion in Appendix \ref{app_optimal_diff}. To this end, we assume that either $T$ is sufficiently small or $\|\rho\|_{L^\infty(0,T;H^{\beta+1})}$ is sufficiently small 
for some $\beta>\frac d2+1$.

Recall that, by Definition~\ref{def_weak}, $f^\e$ satisfies the kinetic equation  in the sense of distributions. In particular, the force term  $\nabla(-\Delta)^{-\alpha}\rho^\e \cdot \nabla_\xi f^\e$ is interpreted through integration by parts in the velocity variable, which implicitly requires
\[
f^\e \, \nabla(-\Delta)^{-\alpha}\rho^\e \in L^1_{\mathrm{loc}}((0,T)\times\R^d\times\R^d).
\]
Since the force field does not depend on $\xi$, this implies
\[
\rho^\e \, \nabla(-\Delta)^{-\alpha}\rho^\e \in L^1_{\mathrm{loc}}((0,T)\times\R^d),
\]
so that the macroscopic drift term
\[
\nabla\cdot(\rho^\e \nabla(-\Delta)^{-\alpha}\rho^\e)
\]
is well defined in the sense of distributions.

Moreover, since
\[
f^\e \in L^\infty\bigl(0,T;L^1(\R^d\times\R^d,(1+|\xi|^2)\,dxd\xi)\bigr),
\]
the velocity moments
\[
\rho^\e:=\intr f^\e\,d\xi,\quad m^\e:=\intr\xi f^\e\,d\xi = \rho^\e u^\e,\quad S^\e:=\intr(\xi\otimes\xi-\mathbb I_d)f^\e\,d\xi
\]
belong to $L^\infty(0,T;L^1(\R^d))$, $L^\infty(0,T;L^1(\R^d;\R^d))$, and $L^\infty(0,T;L^1(\R^d;\R^{d\times d}))$, respectively. Hence, by integrating the kinetic equation with respect to $\xi$, 
we deduce that $\rho^\e$ satisfies
\[
\partial_t \rho^\e-\Delta\rho^\e = \nabla\cdot(\rho^\e\nabla(-\Delta)^{-\alpha}\rho^\e) +\e\nabla\cdot\partial_t m^\e +\nabla\otimes \nabla : S^\e
\]
in the sense of distributions on $(0,T)\times\R^d$.

We work with the inhomogeneous Sobolev space $H^\beta(\R^d)$  equipped with the norm
\[
\|f\|_{H^\beta}^2 := \intr (1+|\eta|^2)^\beta |\hat f (\eta)|^2\,d\eta.
\]
Its dual space is $H^{-\beta}(\R^d)$. Since $L^1(\R^d) \subset H^{-s}(\R^d)$ for any $s>\frac d2$, and $\beta>\frac d2+1$, we have
\[
\rho^\e(t),\,\rho(t) \in L^1(\R^d) \subset H^{-\beta}(\R^d)
\quad \text{for all } t\in[0,T].
\]
Therefore the duality pairing $\langle \rho^\e(t)-\rho(t),\varphi\rangle$ is well defined for any $\varphi\in H^\beta(\R^d)$.

We now record the corresponding Duhamel formula in duality form. For any $\varphi\in H^\beta(\R^d)$, one has
\[
\begin{aligned}
\langle \rho^\e(t),\varphi\rangle
&=\langle \rho_0^\e,\mh_t\star_x\varphi\rangle
-\int_0^t \Big\langle \rho^\e\nabla(-\Delta)^{-\alpha}\rho^\e,\,
\nabla(\mh_{t-s}\star_x\varphi)\Big\rangle\,ds \\
&\quad
-\e\langle \rho^\e u^\e(t),\nabla\varphi\rangle
+\e\langle \rho_0^\e u_0^\e,\nabla(\mh_t\star_x\varphi)\rangle \\
&\quad
-\e\int_0^t \langle \rho^\e u^\e(s),\nabla\Delta(\mh_{t-s}\star_x\varphi)\rangle\,ds
+\int_0^t \langle S^\e(s),\nabla^2(\mh_{t-s}\star_x\varphi)\rangle\,ds,
\end{aligned}
\]
where $\mh_t(x) := (4\pi t)^{-d/2}e^{-\frac{|x|^2}{4t}}$ is the heat kernel and $\star_x$ denotes the convolution with respect to  $x$. Here we used
\[
\int_0^t f(t-s) \frac{dg}{ds}(s)\,ds = \int_0^t \frac{df}{ds}(t-s)g(s)\,ds+ f(0)g(t) - f(t)g(0)
\]
and $\pa_s( \mh_s\star_x f) = \Delta (\mh_s \star_x f)$.
The same identity holds for $\rho$, with the $m^\e$- and $S^\e$-terms removed.

Now, let $\varphi\in H^\beta(\R^d)$ be arbitrary. Applying the above duality formula to 
$\rho^\e$ and $\rho$, and taking the difference, we obtain
\begin{align*}
& \intr (\rho^\e -\rho)\varphi\,dx  \\
&\quad = \intr \mh_t\star_x (\rho_0^\e -\rho_0) \varphi\,dx - \int_0^t \intr(\rho^\e -\rho)(s)\nabla(-\Delta)^{-\alpha}(\rho^\e - \rho)(s) \mh_{t-s}\star_x \nabla\varphi\,dxds  \\
&\qquad - \int_0^t \intr (\rho^\e -\rho)(s) \lt[ (\nabla(-\Delta)^{-\alpha} \rho)(s)\cdot (\mh_{t-s} \star_x\nabla \varphi) - \nabla\cdot(-\Delta)^{-\alpha} (\rho(s) \mh_{t-s}\star_x \nabla\varphi)\rt]\,dxds \\
&\qquad -\e \int_0^t \intrr \Delta\mh_{t-s}(x-y)(\rho^\e u^\e)(s,y)\cdot \nabla\varphi(x)\,dydxds-\e\intr \rho^\e u^\e(t,x) \cdot \nabla\varphi\,dx\\
&\qquad + \e\intr \mh_t\star_x (\rho_0^\e u_0^\e)\cdot \nabla\varphi\,dx + \int_0^t \intr \lt( \intr (\xi\otimes \xi - \mathbb{I}_d)f^\e\,d\xi\rt) :\mh_{t-s}\star_x \nabla\otimes\nabla \varphi\,dxds\\
&\quad =: \sum_{i=1}^7 I_i.
\end{align*}
We estimate the terms $I_i$ one by one. To begin with, one readily obtains
\[
I_1 \le \|\rho_0^\e-\rho_0\|_{H^{-\beta}} \|\mh_t \star\varphi\|_{H^\beta} \le C\|\varphi\|_{H^\beta} \|\rho_0^\e-\rho_0\|_{H^{-\beta}},
\]
and
\begin{align*}
I_2 &\le \int_0^t \|(\rho^\e -\rho)(s)\|_{\dot{H}^{-\alpha}}^2 \|\mh_{t-s}\star_x \nabla\varphi \|_{W_x^{1,\infty}}\,ds\\
&\le C\lt( \mathscr{H}[f^\e | M_{\rho, 0}](0) + \mathscr{P}[\rho^\e | \rho](0) +\e^2 \rt)\int_0^t \|\mh_{t-s}\star_x \nabla\varphi\|_{H^\beta}\,ds\\
&\le C\|\varphi\|_{H^\beta}\lt(\mathscr{H}[f^\e | M_{\rho, 0}](0) + \mathscr{P}[\rho^\e | \rho](0) + \e^2 \rt) \int_0^t (t-s)^{-\frac12}\,ds\\
&\le C\|\varphi\|_{H^\beta}\lt(\mathscr{H}[f^\e | M_{\rho, 0}](0) + \mathscr{P}[\rho^\e | \rho](0) + \e^2 \rt),
\end{align*}
where we used
\[
x^k e^{-ax} \le \lt(\frac{k}{a} \rt)^k e^{-k}\quad \mbox{on} \quad x >0
\]
and
\begin{align*}
\|\mh_{t-s}\star\nabla\varphi\|_{H^\beta}^2 \le \intr (1+|\eta|^2)^\beta |\eta|^2 e^{-2(t-s)|\eta|^2} |\hat\varphi|^2\,d\eta \le C(t-s)^{-1} \|\varphi\|_{H^\beta}^2.
\end{align*}
Next, we decompose $I_3$ as
\begin{align*}
I_3&=- \int_0^t \intr (\rho^\e -\rho)(s) (\nabla(-\Delta)^{-\alpha} \rho)(s) \cdot (\mh_{t-s} \star_x\nabla \varphi) \,dx ds\\
&\quad + \int_0^t \intr (\rho^\e -\rho)(s)  \nabla\cdot(-\Delta)^{-\alpha} (\rho(s) \mh_{t-s}\star_x \nabla\varphi)\,dxds\\
&=: I_{31} + I_{32}.
\end{align*}
For the first part $I_{31}$, we use
\bq\label{ineq_beta}
\|f\|_{H^\beta} \le C(\|f\|_{L^2} + \|f\|_{\dot{H}^\beta}), \quad \|uv\|_{\dot{H}^\beta} \le C(\|u\|_{L^\infty}\|v\|_{\dot{H}^\beta} + \|u\|_{\dot{H}^\beta}\|v\|_{L^\infty})
\eq
to have
\begin{align*}
I_{31} &\le \int_0^t \|\rho^\e - \rho\|_{H^{-\beta}} \|\nabla(-\Delta)^{-\alpha} \rho(s) \cdot (\mh_{t-s} \star_x\nabla \varphi)\|_{H^\beta}\,ds\\
&\le C \int_0^t \|\rho^\e - \rho\|_{H^{-\beta}}\|\nabla(-\Delta)^{-\alpha} \rho(s) \cdot (\mh_{t-s} \star_x\nabla \varphi)\|_{L^2}\,ds\\
&\quad +C \int_0^t \|\rho^\e - \rho\|_{H^{-\beta}}\|\nabla(-\Delta)^{-\alpha} \rho(s) \cdot (\mh_{t-s} \star_x\nabla \varphi)\|_{\dot{H}^\beta}\,ds\\
&\le C\|\nabla(-\Delta)^{-\alpha}\rho\|_{L^\infty(0,T;L^\infty\cap \dot{H}^\beta)} \int_0^t \|\rho^\e - \rho\|_{H^{-\beta}} \|\mh_{t-s} \star_x\nabla \varphi\|_{H^\beta}\,ds\\
&\le C\|\nabla(-\Delta)^{-\alpha}\rho\|_{L^\infty(0,T;L^\infty\cap \dot{H}^\beta)} \|\varphi\|_{H^\beta} \int_0^t (t-s)^{-\frac12} \|\rho^\e - \rho\|_{H^{-\beta}}\,ds.  
\end{align*}
For the second part $I_{32}$, we distinguish the cases $\alpha \in (\frac12,1)$ and $\alpha \in (0,\frac12]$. If $\alpha \in (\frac12,1)$, the Hardy--Littlewood--Sobolev inequality yields
\begin{align*}
 \|\nabla \cdot(-\Delta)^{-\alpha} (\rho(s)\mh_{t-s}\star_x \nabla\varphi)\|_{L^2} &\le C\|\rho(s)\mh_{t-s}\star_x \nabla\varphi\|_{\dot{H}^{1-2\alpha}}\\
 &\le C\|\rho(s)\mh_{t-s}\star_x \nabla\varphi\|_{L^p}\\
 &\le C \|\rho\|_{L^\infty(0,T;L^p)} \|\mh_{t-s}\star_x\nabla\varphi\|_{L^\infty}\\
 &\le C \|\rho\|_{L^\infty(0,T;L^p)} \|\mh_{t-s}\star_x\varphi\|_{H^\beta},
\end{align*}
where $p>1$ satisfies $\frac12 = \frac1p - \frac{2\alpha-1}{d}$. If $\alpha\in(0,\frac12]$, one simply gets
\begin{align*}
 &\|\nabla \cdot(-\Delta)^{-\alpha} (\rho(s)\mh_{t-s}\star_x \nabla\varphi)\|_{L^2} \\
 &\quad \le  C\| \rho(s)\mh_{t-s}\star_x \nabla\varphi\|_{\dot{H}^{1-2\alpha}}\\
 &\quad \le  C \lt(\|\rho\|_{L^\infty(0,T;L^\infty)}\|\mh_{t-s}\star_x \nabla \varphi\|_{\dot{H}^{1-2\alpha}} + \|\rho\|_{L^\infty(0,T;\dot{H}^{1-2\alpha})}\|\mh_{t-s}\star_x \nabla\varphi\|_{L^\infty}\rt) \\
 &\quad \le C\|\rho\|_{L^\infty(0,T;H^\beta)} \|\mh_{t-s}\star_x \varphi\|_{H^\beta}.
\end{align*}

Combining these bounds with \eqref{ineq_beta}, we get
\begin{align*}
I_{32} &\le \int_0^t \|\rho^\e -\rho\|_{H^{-\beta}} \|\nabla \cdot(-\Delta)^{-\alpha} (\rho(s)\mh_{t-s}\star_x \nabla\varphi)\|_{H^\beta}\,ds\\
&\le C\int_0^t \|\rho^\e -\rho\|_{H^{-\beta}} \|\nabla \cdot(-\Delta)^{-\alpha} (\rho(s)\mh_{t-s}\star_x \nabla\varphi)\|_{L^2}\,ds\\
&\quad + C\int_0^t \|\rho^\e -\rho\|_{H^{-\beta}} \| \rho(s)\mh_{t-s}\star_x \nabla\varphi\|_{\dot{H}^{\beta+1-2\alpha}}\,ds\\
&\le C\|\rho\|_{L^\infty(0,T;L^p\cap H^\beta)}\int_0^t \|\rho^\e -\rho\|_{H^{-\beta}} \|\mh_{t-s}\star_x \varphi\|_{H^\beta}\,ds\\
&\quad + C\|\rho\|_{L^\infty(0,T;L^\infty\cap H^{\beta+1-2\alpha})}\int_0^t \|\rho^\e -\rho\|_{H^{-\beta}} \|\mh_{t-s}\star_x \varphi\|_{H^{\beta+2-2\alpha}}\,ds\\
&\le C\|\rho\|_{L^\infty(0,T;L^p\cap H^{\beta+1})}\|\varphi\|_{H^\beta}\int_0^t (t-s)^{-1+\alpha}\|\rho^\e -\rho\|_{H^{-\beta}}\,ds. 
\end{align*}

We now turn to the remaining terms. For $I_4$, we use $\intr \rho^\e|u^\e -{\rm u}_\e|^2\,dx \le \mathscr{H}[f^\e |M_{\rho, {\rm u}_\e}]$ to obtain
\begin{align*}
I_4 &= -\e \int_0^t \intr \rho^\e u^\e \mh_{t-s}\star_x (\Delta \nabla\varphi)\,dxds\\
&\le C\e \|\varphi\|_{H^\beta}\int_0^t \max\lt\{1, (t-s)^{-\frac{\frac d2 +3-\beta}{2}}\rt\} \lt(\|\rho^\e (u^\e - {\rm u}_\e)\|_{L^1} + \|\rho^\e {\rm u}_\e\|_{L^1}\rt)\,ds\\
&\le C\e \|\varphi\|_{H^\beta} \lt(\int_0^t \lt[1+(t-s)^{-\frac{\frac d2 +3-\beta}{2}}\rt]\,ds\rt)^{\frac12} \lt(\int_0^t \intr (t-s)^{-\frac{\frac d2 +3-\beta}{2}}  \rho^\e |u^\e - {\rm u}_\e|^2\,dxds\rt)^{\frac12} +C\e^2 \|\varphi\|_{H^\beta}\\
&\le C\e \|\varphi\|_{H^\beta} \lt( \mathscr{H}[f^\e | M_{\rho, {\rm u}_\e}](0)  + \mathscr{P}[\rho^\e | \rho](0)+ \e^2\rt)^{\frac12},
\end{align*}
where we used $\beta>\frac d2+1$ (i.e.  $\frac{\frac d2 +3-\beta}{2} <1$) and
\[\begin{aligned}
\|\mh_{t-s}\star_x \Delta \nabla \varphi\|_{L^\infty} &\le  C\intr |\eta|^3 e^{-(t-s)|\eta|^2}|\hat\varphi|\,d\eta\\
&\le C\lt(\intr \frac{|\eta|^6}{(1+|\eta|^2)^\beta} e^{-(t-s)|\eta|^2}\,d\eta \rt)^{\frac12}\|\varphi\|_{H^\beta}\\
&\le C\max\lt\{1,(t-s)^{-\frac{\frac d2+3-\beta}{2}}\rt\} \|\varphi\|_{H^\beta}.
\end{aligned}\]
For $I_5$, we find
\begin{align*}
I_5 &\le C\e\|\rho^\e u^\e\|_{L^1} \|\nabla \varphi\|_{L^\infty}\\
&\le C\e \lt( \|\rho^\e (u^\e -{\rm u}_\e)\|_{L^1} + \e\rt) \|\varphi\|_{H^\beta}\\
&\le C\e \|\varphi\|_{H^\beta} \lt( \mathscr{H}[f^\e | M_{\rho, {\rm u}_\e}](0)  + \mathscr{P}[\rho^\e | \rho](0) + \e^2\rt)^{\frac12}.
\end{align*}
Similarly, for the initial contribution $I_6$, we have
\[\begin{aligned}
I_6 &\le \e \|\rho_0^\e u_0^\e\|_{L^1} \|\mh_t \star_x \nabla\varphi\|_{L^\infty} \le C\e \|\varphi\|_{H^\beta} \lt(  \mathscr{H}[f^\e | M_{\rho, {\rm u}_\e}](0) + \e^2\rt)^{\frac12}.
\end{aligned}\]
It remains to estimate $I_7$, which contains the second-order velocity moment. To this end, we first observe that
\begin{align*}
\intr (\xi\otimes \xi -\mathbb{I}_d)f^\e\,d\xi &= \intr (\xi \otimes \xi - u^\e \otimes u^\e - \mathbb{I}_d)f^\e\,d\xi + \rho^\e u^\e \otimes u^\e\\
&= \intr \lt[((\xi - u^\e)f^\e +\nabla_\xi f^\e)\otimes \xi + u^\e \otimes ((\xi -u^\e)f^\e +\nabla_\xi f^\e ) \rt]\,d\xi + \rho^\e u^\e \otimes u^\e.
\end{align*}
This gives
\begin{align*}
I_7 &\le \int_0^t \lt(\intrr \frac{1}{f^\e}|(\xi-u^\e) f^\e + \nabla_\xi f^\e|^2\,dxd\xi\rt)^{\frac12}\lt(\intrr |\xi|^2 f^\e \,dxd\xi \rt)^{\frac12} \|\mh_{t-s} \star_x \nabla \otimes \nabla \varphi\|_{L_x^\infty}\,ds\\
&\quad + \int_0^t \lt(\intrr \frac{1}{f^\e}|(\xi-u^\e) f^\e + \nabla_\xi f^\e|^2\,dxd\xi\rt)^{\frac12}\lt(\intr \rho^\e  |u^\e|^2\,dx \rt)^{\frac12}\|\mh_{t-s} \star_x \nabla \otimes \nabla \varphi\|_{L_x^\infty}\,ds\\
&\quad + \int_0^t \intr \lt(\rho^\e |u^\e - {\rm u}_\e|^2 + \rho^\e |{\rm u}_\e|^2\rt) \|\mh_{t-s} \star_x \nabla \otimes \nabla \varphi\|_{L_x^\infty}\,dxds \\
&\le C\|\varphi\|_{H^\beta} \lt(\int_0^t (t-s)^{-\frac12}\lt(\intrr \frac{1}{f^\e}|(\xi-u^\e) f^\e + \nabla_\xi f^\e|^2(s)\,dxd\xi\rt) ds\rt)^{\frac12}\\
&\quad + C\|\varphi\|_{H^\beta} \int_0^t (t-s)^{-\frac12}\lt(\mathscr{H}[f^\e |M_{\rho, {\rm u}_\e}](0) + \e^2 \rt)\,dxds\\
&\le C\|\varphi\|_{H^\beta} \lt(\int_0^t (t-s)^{-\frac12}\lt(\intrr \frac{1}{f^\e}|(\xi-u^\e) f^\e + \nabla_\xi f^\e|^2(s)\,dxd\xi\rt) ds\rt)^{\frac12}\\
&\quad + C\|\varphi\|_{H^\beta} \lt(\mathscr{H}[f^\e | M_{\rho, {\rm u}_\e}](0)  + \mathscr{P}[\rho^\e | \rho](0) + \e^2\rt).
\end{align*}

Collecting the above estimates,  and using \eqref{rel_hh}, we arrive at
\begin{align}\label{ineq_diff_sha}
\begin{aligned}
\|(\rho^\e -\rho)(t)\|_{H^{-\beta}} &\le C\|\rho_0^\e - \rho_0\|_{H^{-\beta}} + C\lt(\mathscr{H}[f^\e | M_{\rho, 0}](0) + \mathscr{P}[\rho^\e | \rho](0) + \e^2 \rt)\\
&\quad + C\int_0^t \lt((t-s)^{-\frac12} + (t-s)^{-1+\alpha}\rt)\|(\rho^\e-\rho)(s)\|_{H^{-\beta}}\,ds\\
&\quad + C\lt(\int_0^t (t-s)^{-\frac12}\lt(\intrr \frac{1}{f^\e}|(\xi-u^\e) f^\e + \nabla_\xi f^\e|^2(s)\,dxd\xi\rt) ds\rt)^{\frac12}
\end{aligned}
\end{align}
due to \eqref{rel_hh}.

To close the above estimate, we use the following auxiliary lemma.
\begin{lemma}\label{lem_aux_diff}
Let $T>0$, $\alpha>0$, and let $f,g \in L^2(0,T)$ be nonnegative functions such that
\[
f(t) \le g(t) + C\int_0^t \bigl((t-s)^{-\frac12} + (t-s)^{-1+\alpha}\bigr) f(s)\,ds
\]
for a.e. $t\in (0,T)$. Then there exists a constant $C_T>0$, depending only on $C$, $\alpha$, and $T$, such that
\[
\|f\|_{L^2(0,T)} \le C_T \|g\|_{L^2(0,T)}.
\]
\end{lemma}

\begin{proof}
Define
\[
\kappa(t):=C\bigl(t^{-\frac12}+t^{-1+\alpha}\bigr), \quad t\in(0,T),
\]
and, for measurable functions $\phi,\psi$ on $(0,T)$, set
\[
(\phi \star_t \psi)(t):=\int_0^t \phi(t-s)\psi(s)\,ds.
\]
Then the assumption reads
\[
f \le g + \kappa \star_t f \quad \text{a.e. on } (0,T).
\]
For $\lambda>0$, set
\[
F(t):=e^{-\lambda t}f(t), \quad G(t):=e^{-\lambda t}g(t),
\]
and define
\[
\kappa_\lambda(t):=e^{-\lambda t}\kappa(t) = C e^{-\lambda t}\lt(t^{-\frac12}+t^{-1+\alpha}\rt).
\]
Then, we obtain
\[
F(t)\le G(t)+(\kappa_\lambda \star_t F)(t).
\]

Since $\alpha>0$, we find $\kappa_\lambda\in L^1(0,T)$ and
\[
\|\kappa_\lambda\|_{L^1(0,T)} = C\int_0^T e^{-\lambda t}\lt(t^{-\frac12}+t^{-1+\alpha}\rt) dt \to 0
\quad\text{as }\lambda\to\infty.
\]
Thus, choosing $\lambda>0$ sufficiently large, we may assume that
\[
\|\kappa_\lambda\|_{L^1(0,T)}<1.
\]

By Young's inequality for time convolution, we get
\[
\|F\|_{L^2(0,T)}
\le \|G\|_{L^2(0,T)}
 + \|\kappa_\lambda \star_t F\|_{L^2(0,T)}
\le \|G\|_{L^2(0,T)}
 + \|\kappa_\lambda\|_{L^1(0,T)}\|F\|_{L^2(0,T)}.
\]
It follows that
\[
\bigl(1-\|\kappa_\lambda\|_{L^1(0,T)}\bigr)\|F\|_{L^2(0,T)}
\le \|G\|_{L^2(0,T)},
\]
and thus
\[
\|F\|_{L^2(0,T)}
\le \frac{1}{1-\|\kappa_\lambda\|_{L^1(0,T)}}\|G\|_{L^2(0,T)}.
\]

Finally, since $\|G\|_{L^2(0,T)}\le \|g\|_{L^2(0,T)}$ and $\|f\|_{L^2(0,T)}
\le e^{\lambda T}\|F\|_{L^2(0,T)}$, we have
\[
\|f\|_{L^2(0,T)} \le  \frac{e^{\lambda T}}{1-\|\kappa_\lambda\|_{L^1(0,T)}}\|g\|_{L^2(0,T)}.
\]
This concludes the desired result.
\end{proof}

We now define
\begin{align*}
g(t) &:=C\|\rho_0^\e - \rho_0\|_{H^{-\beta}} + C\lt(\mathscr{H}[f^\e | M_{\rho, 0}](0) + \mathscr{P}[\rho^\e | \rho](0) + \e^2 \rt)\\
&\quad + C\lt(\int_0^t (t-s)^{-\frac12}\lt(\intrr \frac{1}{f^\e}|(\xi-u^\e) f^\e + \nabla_\xi f^\e|^2(s)\,dxd\xi\rt) ds\rt)^{\frac12}.
\end{align*}
Then the inequality \eqref{ineq_diff_sha} takes the form
\[
\|(\rho^\e-\rho)(t)\|_{H^{-\beta}}
\le g(t) + C\int_0^t \bigl((t-s)^{-\frac12}+(t-s)^{-1+\alpha}\bigr)\|(\rho^\e-\rho)(s)\|_{H^{-\beta}}\,ds.
\]
Moreover, $g\in L^2(0,T)$, and by Fubini's theorem,
\begin{align*}
\int_0^T |g(t)|^2\,dt
&\le C\lt(\|\rho_0^\e - \rho_0\|_{H^{-\beta}}
 + \mathscr{H}[f^\e | M_{\rho, 0}](0) + \mathscr{P}[\rho^\e | \rho](0)
 + \e^2\rt)^2\\
&\quad + C\int_0^T \int_0^t (t-s)^{-\frac12}
\lt(\iint_{\R^d\times\R^d}\frac{1}{f^\e}|(\xi-u^\e) f^\e + \nabla_\xi f^\e|^2(s)\,dxd\xi\rt) dsdt\\
&= C\lt(\|\rho_0^\e - \rho_0\|_{H^{-\beta}}
 + \mathscr{H}[f^\e | M_{\rho, 0}](0) + \mathscr{P}[\rho^\e | \rho](0)
 + \e^2\rt)^2\\
&\quad + C\int_0^T
\lt(\iint_{\R^d\times\R^d}\frac{1}{f^\e}|(\xi-u^\e) f^\e + \nabla_\xi f^\e|^2(s)\,dxd\xi\rt)
\lt(\int_s^T (t-s)^{-\frac12}\,dt\rt)ds\\
&\le C\lt(\|\rho_0^\e - \rho_0\|_{H^{-\beta}}
 +\mathscr{H}[f^\e | M_{\rho, 0}](0) + \mathscr{P}[\rho^\e | \rho](0)
 + \e^2\rt)^2\\
&\quad + C\int_0^T\iint_{\R^d\times\R^d}\frac{1}{f^\e}|(\xi-u^\e) f^\e + \nabla_\xi f^\e|^2\,dxd\xi dt.
\end{align*}
Applying Lemma \ref{lem_aux_diff}, we deduce
\begin{align*}
\int_0^T \|(\rho^\e -\rho)(t)\|_{H^{-\beta}}^2\,dt 
&\le C\lt(\|\rho_0^\e - \rho_0\|_{H^{-\beta}} + \mathscr{H}[f^\e | M_{\rho, 0}](0) + \mathscr{P}[\rho^\e | \rho](0) + \e^2 \rt)^2 \\
 &\quad + C\int_0^T\intrr \frac{1}{f^\e}|(\xi-u^\e) f^\e + \nabla_\xi f^\e|^2\,dxd\xi dt\\
&\le C\lt(\|\rho_0^\e - \rho_0\|_{H^{-\beta}} + \mathscr{H}[f^\e | M_{\rho, 0}](0) + \mathscr{P}[\rho^\e | \rho](0) + \e^2 \rt)^2
\end{align*}
due to \eqref{con_diff2}. This completes the proof.

\begin{remark}\label{rmk_opt_diff} 
A corresponding rigorous estimate toward the shifted Maxwellian $M_{\rho,{\rm u}_\e}$ can also be obtained in a negative Sobolev topology. Indeed, if we choose $\sigma>\max\{d,\beta\}$, then Corollary \ref{cor_gd2}, combined with Theorem \ref{thm_kin1}, implies
\begin{align*}
\int_0^T \|f^\e-M_{\rho, {\rm u}_\e}\|_{H_{x,\xi}^{-\sigma}}^2\,dt &\le C\int_0^T\iint_{\R^d \times \R^d} \frac1{f^\e}|\nabla_\xi f^\e + (\xi-u^\e) f^\e|^2\,dxd\xi dt +C\int_0^T \intr \rho^\e |u^\e - {\rm u}_\e|^2\,dxdt  \cr
&\quad + C\|{\rm u}_\e\|_{L^\infty((0,T)\times\R^d)}^2 \int_0^T \|\rho^\e - \rho\|_{L^1}^2\,dt  +C\int_0^T \|\rho^\e - \rho\|_{H^{-\sigma}}^2\,dt\\
&\le C\lt( \|\rho_0^\e -\rho_0\|_{H^{-\beta}} + \mathscr{H}[f^\e | M_{\rho, 0}](0)  + \mathscr{P}[\rho^\e | \rho](0) + \e^2\rt)^2,
\end{align*}
where we used $\|\rho^\e -\rho\|_{H^{-\sigma}} \le \|\rho^\e - \rho\|_{H^{-\beta}}$.
\end{remark}

%
%
%
%
%
%
\subsection{High-field limit}

In this subsection, we turn to the high-field limit of \eqref{eq:gkin}. Recall
\bq\label{prelim_high}
\pa_t f^\e + v\cdot \nabla_x f^\e - \frac1\e \nabla (-\Delta)^{-\alpha} \rho^\e  \cdot \nabla_\xi  f^\e = \frac1\e\nabla_\xi \cdot (\nabla_\xi f^\e + \xi f^\e), \quad t>0, \ (x,\xi) \in \R^d \times \R^d.
\eq
Our aim is to derive the limiting aggregation equation
\[
\pa_t \rho -  \nabla \cdot (\rho \nabla (-\Delta)^{-\alpha} \rho) =0.
\]

Compared to the diffusive regime, the high-field scaling introduces a drift term with prefactor $\frac1\e$. At first sight, this looks more singular, but the key cancellation is that the same scaling also appears in the friction term of the Fokker--Planck operator. As a result, the entropy identity features a balance in which the potentially singular contribution is compensated, leaving a clean $\frac1\e$ dissipation structure. Moreover, rewriting the system in conservation form,
\begin{align}\label{lim_high}
\begin{aligned}
&\pa_t \rho +  \nabla \cdot (\rho {\rm u}) =0, \cr
&\e \lt(\pa_t (\rho {\rm u}) +  \nabla \cdot (\rho {\rm u} \otimes {\rm u}) + \nabla \rho\rt)  = - \rho \nabla (-\Delta)^{-\alpha} \rho -   \rho {\rm u} + \e \rho {\rm e},
\end{aligned}
\end{align}
where
\[
\rho {\rm e} = \rho \pa_t {\rm u} +  \rho {\rm u} \cdot \nabla {\rm u} + \nabla \rho, \quad {\rm u} = -\nabla (-\Delta)^{-\alpha} \rho,
\]
shows that the effective velocity field is $\e$-independent and directly controlled by the limit density. This stands in contrast with the diffusive case, where the relevant flux is $\frac1\e \rho^\e u^\e$ and only weak convergence can be extracted. In the high-field setting, however, the unscaled flux $\rho^\e u^\e$ can be handled quantitatively, leading to sharper and structurally simpler estimates.

As in the diffusive regime, we begin by recording uniform estimates from the entropy-dissipation structure and the kinetic energy bound.
\begin{lemma}\label{en_high} Let $T>0$ and $f^\e$ be a weak entropy solution to the equation \eqref{prelim_high} in the sense of Definition \ref{def_weak} on the time interval $[0,T]$. Then we have
\[
\sup_{0 \leq t \leq T}\lt( \mathscr{F}[f^\e(t)] + \frac1{\e} \mathscr{P}[\rho^\e(t)]\rt)  + \frac1{\e}\int_0^T \iint_{\R^d \times \R^d} \frac1{f^\e}|\nabla_\xi f^\e + \xi f^\e|^2 \,dxd\xi dt  \leq \mathscr{F}[f_0^\e] + \frac1{\e} \mathscr{P}[\rho^\e_0]
\]
and
\[
\int_0^T \mathscr{K}[f^\e(t)]\,dt \leq dT + \e\mathscr{K}[f^\e_0] + \mathscr{P}[\rho_0^\e].
\]
\end{lemma}
 
The next step is to combine this with the relative entropy framework. 
In the high-field scaling, the key feature is that both the interaction energy $\mathscr{P}[\rho^\e|\rho]$ and the dissipation functional carry a prefactor $\frac1\e$. 
This enhances microscopic relaxation but simultaneously amplifies the contribution of the nonlocal force, making the error terms harder to control. 
The following proposition quantifies this balance.

\begin{proposition}\label{re_high} Let $T>0$ and $f^\e$ be a weak entropy solution to the equation \eqref{prelim_high} in the sense of Definition \ref{def_weak} on the time interval $[0,T]$,  and let $(\rho, {\rm u})$ be a classical solutions to the equation \eqref{lim_high} with $\nabla {\rm u} \in L^\infty((0,T) \times \R^d)$. Then we have
 \begin{align*}
& \sup_{0 \leq t \leq T}\lt( \mathscr{H}[f^\e | M_{\rho, {\rm u}}](t) + \frac1{\e} \mathscr{P}[\rho^\e | \rho](t) \rt)\cr
&\quad + \frac1{2\e} \int_0^T\iint_{\R^d \times \R^d} \frac1{f^\e} |\nabla_\xi f^\e + (\xi - u^\e)f^\e|^2 \,dxd\xi dt +  \lt( \frac1{2\e} - \|\nabla {\rm u}\|_{L^\infty} \rt)\int_0^T \intr \rho^\e |u^\e - {\rm u}|^2\,dx dt \cr
&\qquad \leq   \mathscr{H}[f^\e | M_{\rho, {\rm u}}](0) + \frac1{\e} \mathscr{P}[\rho^\e | \rho](0)  + C\e \lt(1 + \int_0^T\mathscr{K}[f^\e(t)]\,dt\rt)  + \frac C\e \int_0^T \mathscr{P}[\rho^\e | \rho](t)\,dt,
\end{align*}
where $C>0$ is independent of $\e > 0$.
\end{proposition}
\begin{proof} The argument is analogous to the diffusive case (Proposition \ref{re_diffu}): 
the time-integrated form of Proposition \ref{prop_key} extends to weak entropy solutions 
by approximation, thus we may apply it here with parameters ${\rm A} = {\rm B} = \e$, ${\rm C} = 0$, $\tau = 1$ , and
\[
F[f^\e] = - \nabla (-\Delta)^{-\alpha} \rho^\e, \quad R[\rho, {\rm u}] = -  \nabla (-\Delta)^{-\alpha} \rho -   {\rm u} + \e  {\rm e}.
\]
This deduces
 \begin{align*}
&  \iint_{\R^d \times \R^d} f^\e \log \lt( \frac{f^\e}{ M_{\rho, {\rm u}}} \rt) dxd\xi + \frac1{2\e} \int_0^t\iint_{\R^d \times \R^d} \frac1{f^\e} |\nabla_\xi f^\e + (\xi - u^\e)f^\e|^2 dxd\xi ds\cr
&\quad \leq   \iint_{\R^d \times \R^d} f^\e_0 \log \lt( \frac{f^\e_0}{ M_{\rho_0, {\rm u}_0}} \rt) dxd\xi + \|\nabla {\rm u}\|_{L^\infty} \int_0^t \intr \rho^\e |u^\e - {\rm u}|^2\,dx ds  +  \frac\e2\|\nabla {\rm u}\|_{L^\infty}^2 \int_0^t  \|f^\e(s)\|_{L^1_2}\,ds   \cr
&\qquad - \frac1{\e} \int_0^t \intr \rho^\e u^\e \cdot (u^\e - {\rm u})\,dx ds +  \frac1\e \int_0^t \iint_{\R^d \times \R^d}  f^\e ({\rm u}-\xi) \cdot \lt(\nabla (-\Delta)^{-\alpha} (\rho^\e  - \rho)  -  {\rm u} + \e {\rm e}\rt) dxd\xi ds\cr
&\quad =  \iint_{\R^d \times \R^d} f^\e_0 \log \lt( \frac{f^\e_0}{ M_{\rho_0, {\rm u}_0}} \rt)  + \|\nabla {\rm u}\|_{L^\infty} \int_0^t\intr \rho^\e |u^\e - {\rm u}|^2\,dx ds +  \frac\e2\|\nabla {\rm u}\|_{L^\infty}^2 \int_0^t\|f^\e(s)\|_{L^1_2}\,ds   \cr
&\qquad  - \frac1{\e}\int_0^t\intr \rho^\e |u^\e - {\rm u}|^2\,dx ds - \int_0^t \intr \rho^\e(u^\e - {\rm u})\cdot {\rm e}\,dx ds \cr
&\qquad- \frac1\e \int_0^t \intr  \rho^\e (u^\e- {\rm u}) \cdot \nabla (-\Delta)^{-\alpha}(\rho^\e  - \rho) \,dx ds.
\end{align*}
Finally, the cross term $\into \rho^\e(u^\e-{\rm u})\cdot {\rm e}\,dx$ is bounded by
\[
\lt|\intr \rho^\e (u^\e - {\rm u}) \cdot {\rm e}\,dx\rt| \leq \frac1{4\e}\intr \rho^\e |u^\e - {\rm u}|^2\,dx + \e \| {\rm e}\|_{L^\infty}\|\rho^\e\|_{L^1},
\]
and Proposition \ref{prop_mod} controls the nonlocal contribution. 
This concludes the desired result.
\end{proof}
%
%
%
%
%
%
\subsubsection{Proof of Theorem \ref{thm_kin2}: high-field limit} We now complete the argument for the high-field scaling. In contrast to the diffusive case, the prefactors are only of order $\frac1\e$, and hence the available dissipation is weaker. 
Nevertheless, Proposition \ref{re_high} provides a differential inequality that is still coercive enough to control deviations from equilibrium, provided $\e$ is small and $\|\nabla {\rm u}\|_{L^\infty}$ remains bounded. 

Indeed, Proposition \ref{re_high} and Lemma \ref{en_high} yield, for sufficiently small $\e>0$ such that $4\e \|\nabla {\rm u}\|_{L^\infty}<1$,
\begin{align*}
& \sup_{0 \leq t \leq T}\lt( \mathscr{H}[f^\e | M_{\rho, {\rm u}}](t) + \frac1{\e} \mathscr{P}[\rho^\e | \rho](t) \rt)\cr
&\quad + \frac1{2\e} \int_0^T\iint_{\R^d \times \R^d} \frac1{f^\e} |\nabla_\xi f^\e + (\xi - u^\e)f^\e|^2 \,dxd\xi dt +  \frac1{4\e}\int_0^T \intr \rho^\e |u^\e - {\rm u}|^2\,dx dt \cr
&\qquad \leq \mathscr{H}[f^\e | M_{\rho, {\rm u}}](0) + \frac1\e \mathscr{P}[\rho^\e | \rho](0) + C \e +  \frac{C}{\e} \int_0^T \mathscr{P}[\rho^\e | \rho](t) \,dt.
\end{align*}
We finally apply Gr\"onwall's lemma to obtain the uniform stability bound
\begin{align}\label{con_high}
\begin{aligned}
&\sup_{0 \leq t \leq T}\lt( \mathscr{H}[f^\e | M_{\rho, {\rm u}}](t) + \frac1\e \mathscr{P}[\rho^\e | \rho](t) \rt) \cr
&\quad  + \frac1{2\e}\int_0^T \iint_{\R^d \times \R^d} \frac1{f^\e} |\nabla_\xi f^\e + (\xi - u^\e)f^\e|^2 dxd\xi dt  + \frac1{4\e} \int_0^T \intr \rho^\e |u^\e - {\rm u}|^2\,dx dt \cr
&\qquad \leq C\lt( \mathscr{H}[f^\e | M_{\rho, {\rm u}}](0) + \frac1\e \mathscr{P}[\rho^\e | \rho](0) + \e\rt).
\end{aligned}
\end{align}

\medskip
\noindent {\bf Case of well-prepared initial data.} In this setting, the initial modulated entropy and interaction energy are sufficiently small, so the estimate \eqref{con_high} implies strong convergence of both $f^\e$ and $\rho^\e$. Moreover, the momentum balance yields
 \[
 \intr |\rho^\e u^\e - \rho {\rm u}|\,dx \leq \lt(\intr \rho^\e|u^\e - {\rm u}|^2\,dx\rt)^\frac12 + \|{\rm u}\|_{L^\infty}\intr |\rho^\e - \rho|\,dx  \leq C \mathscr{H}[f^\e | M_{\rho, {\rm u}}]^\frac12,
 \]
so that
  \[
 \rho^\e u^\e \to -\rho\nabla (-\Delta)^{-\alpha} \rho \quad \mbox{in } L^\infty(0,T; L^1(\R^d)).
 \]
 
 \medskip
\noindent {\bf Case of mildly well-prepared initial data.}  When the initial data are only close in weak topology, the entropy control is not enough by itself, but the bounded Lipschitz estimates from Lemmas \ref{lem_gd} and \ref{lem_gd2} bridge this gap. It gives
\begin{align*}
&\sup_{0 \leq t \leq T} {\rm d}^2_{\rm BL}(\rho^\e(t), \rho(t))  + \int_0^t {\rm d}^2_{\rm BL}(\rho^\e u^\e, \rho {\rm u})\,ds + \int_0^t {\rm d}^2_{\rm BL}(f^\e, M_{\rho, {\rm u}})\,ds \cr
&\quad \leq C  {\rm d}^2_{\rm BL}(\rho^\e_0, \rho_0) + C\int_0^t\iint_{\R^d \times \R^d} \frac1{f^\e}|\nabla_\xi f^\e + (\xi - u^\e) f^\e|^2\,dxd\xi ds + C\int_0^t \intr \rho^\e |u^\e - {\rm u}|^2\,dxds.
\end{align*}
Combining this with the dissipation bounds from \eqref{con_high}, we obtain
\begin{align*}
&\sup_{0 \leq t \leq T} {\rm d}^2_{\rm BL}(\rho^\e(t), \rho(t))  + \int_0^t {\rm d}^2_{\rm BL}(\rho^\e u^\e, \rho {\rm u})\,ds + \int_0^t {\rm d}^2_{\rm BL}(f^\e, M_{\rho, {\rm u}})\,ds \cr
&\quad \leq C  {\rm d}^2_{\rm BL}(\rho^\e_0, \rho_0) + C\lt(\e \mathscr{H}[f^\e | M_{\rho, {\rm u}}](0) +   \mathscr{P}[\rho^\e | \rho](0) + \e^2\rt).
\end{align*}
From the above estimates, we conclude the convergence results stated in Theorem \ref{thm_kin2}.

%
%
%
%
%
%
\subsection{gSQG limit} 
We now consider the generalized Surface Quasi--Geostrophic (gSQG) scaling of \eqref{eq:gkin}. Recall
\bq\label{kin_sqg}
\e \pa_t f^\e + \xi \cdot \nabla_x f^\e  - \nabla (-\Delta)^{-\alpha} \rho^\e  \cdot \nabla_\xi  f^\e + \frac1\e \xi^\perp \cdot \nabla_\xi f^\e = \nabla_\xi \cdot (\nabla_\xi f^\e + \xi f^\e), \quad t>0, \ (x,\xi) \in \R^2 \times \R^2.
\eq
Our goal is to derive the gSQG equation
\[
\pa_t \rho -  \nabla \cdot (\rho \nabla^\perp (-\Delta)^{-\alpha} \rho) =0.
\]
The key new feature of this scaling is the presence of the singular rotation term $\tfrac1\e \xi^\perp \cdot \nabla_\xi f^\e$. At first sight this appears even more singular than in the high-field case, but in fact it induces a fast gyroscopic motion in velocity space that naturally averages out. More precisely, the combination of the $\frac1\e \xi^\perp$ drift and the Fokker--Planck friction/diffusion ensures that the limit dynamics is governed by an effective transport along the orthogonal direction $\nabla^\perp(-\Delta)^{-\alpha}\rho$. In this way, the gSQG regime differs from both the diffusive and high-field cases: while the diffusive limit requires delicate control of the scaled flux $\tfrac1\e \rho^\e u^\e$, and the high-field limit benefits from an $\e$-independent effective velocity, here the leading-order gyroscopic balance enforces that the limit velocity field is purely rotational. 

To see this structure, we rewrite the system in conservative form as
\begin{align}\label{sqg}
\begin{aligned}
&\e\pa_t \rho +  \nabla \cdot (\rho {\rm u_\e}) =0, \cr
&\e \pa_t (\rho {\rm u_\e}) +  \nabla \cdot (\rho {\rm u_\e} \otimes {\rm u_\e}) +\nabla \rho  = - \rho \nabla (-\Delta)^{-\alpha} \rho +  \frac1\e \rho {\rm u_\e}^\perp + \rho {\rm e}_\e,
\end{aligned}
\end{align}
with the remainder term
\[
\rho {\rm e}_\e = \e \rho \pa_t {\rm u_\e} +  \rho {\rm u_\e} \cdot \nabla {\rm u_\e}, \quad {\rm u_\e} = - \e\nabla^\perp (-\Delta)^{-\alpha} \rho - \e\nabla^\perp \log\rho. 
\]
Here, the actual transport vector field ${\rm u_\e}$ scales like $\e$ and the leading-order flux emerges only after combining with the $\frac1\e$-rotation. As a result, we obtain uniform bounds of the type
\[
\|{\rm u_\e}\|_{W^{1,\infty}} \leq C\e \quad \mbox{and} \quad \|{\rm e}_\e\|_{L^\infty} \leq C\e^2
\]
with $C>0$ independent of $\e>0$. These estimates will be used repeatedly in the entropy and stability analysis below.

\begin{lemma}\label{en_sqg}  Let $T>0$ and $f^\e$ be a weak entropy solution to the equation \eqref{kin_sqg} in the sense of Definition \ref{def_weak} on the time interval $[0,T]$. Then we have
\[
\sup_{0 \leq t \leq T}\lt( \mathscr{F}[f^\e(t)] + \mathscr{P}[\rho^\e(t)]\rt)  + \frac1{\e}\int_0^T \iint_{\R^2 \times \R^2} \frac1{f^\e}|\nabla_\xi f^\e + \xi f^\e|^2\,dxd\xi dt \leq \mathscr{F}[f_0^\e] + \mathscr{P}[\rho^\e_0]
\]
and
\[
\int_0^T \mathscr{K}[f^\e(t)]\,dt \leq dT + \e \lt(\mathscr{K}[f^\e_0]+ \mathscr{P}[\rho_0^\e]\rt).
\]
\end{lemma}
\begin{proof} The additional skew term $\xi^\perp$ does not affect the dissipation, since $\xi \cdot \xi^\perp = 0$ and $\nabla_\xi \cdot \xi^\perp = 0$. Hence, the entropy and kinetic energy balances follow directly from the arguments in Lemma \ref{lem_en}.
\end{proof}

The preceding entropy and energy bounds ensure that the kinetic system remains uniformly controlled in time, despite the presence of the singular drift induced by the $\xi^\perp$ term. 
To pass from these a priori controls to a quantitative stability estimate, we combine the relative entropy framework with modulated interaction energies, in complete analogy with the diffusive and high--field regimes. 
The key difference is that the skew-symmetric rotation encoded in $\xi^\perp$ introduces no dissipation, but its contribution can still be handled thanks to the uniform $W^{1,\infty}$ bound on the effective velocity ${\rm u_\e}$ and the smallness of the error term ${\rm e}_\e$. 
This leads to the following estimate.

\begin{proposition}
Let $T>0$ and $f^\e$ be a weak entropy solution to the equation \eqref{kin_sqg} in the sense of Definition \ref{def_weak} on the time interval $[0,T]$,  and let $(\rho, {\rm u_\e})$ be a classical solutions to the equation \eqref{sqg} with $\nabla {\rm u_\e} \in L^\infty((0,T) \times \R^2)$. Then we have
\begin{align*}
& \sup_{0 \leq t \leq T}\lt( \mathscr{H}[f^\e | M_{\rho, {\rm u_\e}}](t) + \mathscr{P}[\rho^\e | \rho](t) \rt)\cr
&\quad + \frac1{8\e} \int_0^T \iint_{\R^2 \times \R^2} \frac1{f^\e} |\nabla_\xi f^\e + (\xi - u^\e)f^\e|^2 dxd\xi dt + \frac1{8\e} \int_0^T \int_{\R^2} \rho^\e |u^\e - {\rm u_\e}|^2 \,dx dt \cr
&\qquad \leq \mathscr{H}[f^\e | M_{\rho, {\rm u_\e}}](0) + \mathscr{P}[\rho^\e | \rho](0) + C\e \lt(1 + \int_0^T\mathscr{K}[f^\e(t)]\,dt\rt)   + C\int_0^T \mathscr{P}[\rho^\e | \rho](t)\,dt,
\end{align*}
where $C>0$ is independent of $\e > 0$.
\end{proposition}
\begin{proof}
As in Propositions \ref{re_diffu} and \ref{re_high}, the time-integrated version of Proposition \ref{prop_key} remains valid for weak entropy solutions by approximation. We therefore apply it here with ${\rm A} = \e$, ${\rm B} = \tau = 1$, ${\rm C} = 0$,  
\[
F[f^\e] = - \nabla (-\Delta)^{-\alpha} \rho^\e +  \frac1\e \xi^\perp\quad \mbox{and} \quad R[\rho, {\rm u_\e}] =  -    \nabla (-\Delta)^{-\alpha} \rho - \frac1{\e} {\rm u_\e}^\perp +  {\rm e}_\e
\]
to obtain
 \begin{align*}
&  \iint_{\R^2 \times \R^2} f^\e \log \lt( \frac{f^\e}{ M_{\rho, {\rm u_\e}}} \rt)  dxd\xi + \frac1{2\e} \int_0^t \iint_{\R^2 \times \R^2} \frac1{f^\e} |\nabla_\xi f^\e + (\xi - u^\e)f^\e|^2 dxd\xi  ds\cr
&\quad \leq  \iint_{\R^2 \times \R^2} f^\e_0 \log \lt( \frac{f^\e_0}{ M_{\rho_0, {\rm u_{\e0}}}} \rt) dxd\xi  + \frac1\e \int_0^t \|\nabla {\rm u_\e}\|_{L^\infty} \int_{\R^2} \rho^\e |u^\e - {\rm u_\e}|^2\,dx ds  \cr
&\qquad + \frac{1}{2\e} \|\nabla {\rm u_\e}\|_{L^\infty}^2\int_0^t \|f^\e(s)\|_{L^1_2}\,ds - \frac1{\e}\int_0^t \int_{\R^2} \rho^\e u^\e \cdot (u^\e - {\rm u_\e})\,dx ds  \cr
&\qquad - \frac1{\e}\int_0^t\iint_{\R^2 \times \R^2}  f^\e ({\rm u_\e}-\xi) \cdot \lt(- \nabla (-\Delta)^{-\alpha} ( \rho^\e - \rho) + \frac1\e (\xi  - {\rm u_\e})^\perp -  {\rm e}_\e\rt) dxd\xi ds.
\end{align*}
The third term on the right-hand side contains the key dissipation contribution.  It can be rewritten as
 \begin{align*}
- \frac1{\e} \int_{\R^2} \rho^\e u^\e \cdot (u^\e - {\rm u_\e})\,dx &= - \frac1\e \int_{\R^2} \rho^\e |u^\e - {\rm u_\e}|^2\,dx - \frac1\e \int_{\R^2} \rho^\e {\rm u_\e} \cdot (u^\e - {\rm u_\e})\,dx\cr
&\leq -\frac1{2\e} \int_{\R^2} \rho^\e |u^\e - {\rm u_\e}|^2 \,dx + \frac1\e \|{\rm u_\e}\|_{L^\infty}^2 \|f^\e\|_{L^1}.
\end{align*}
Next, the last term simplifies to
\[
   - \frac1\e\int_{\R^2}  \rho^\e (u^\e- {\rm u_\e}) \cdot \nabla (-\Delta)^{-\alpha}(\rho^\e  - \rho) \,dx  - \frac1\e\int_{\R^2} \rho^\e(u^\e - {\rm u_\e})\cdot {\rm e}_\e\,dx,
\]
and the contribution involving ${\rm e}_\e$ is controlled as
\[
\frac1\e \lt|\int_{\R^2} \rho^\e(u^\e - {\rm u_\e})\cdot {\rm e}_\e\,dx\rt| \leq C\e \int_{\R^2} \rho^\e |u^\e - {\rm u_\e}|\,dx \leq \frac1{8\e} \int_{\R^2} \rho^\e |u^\e - {\rm u_\e}|^2 \,dx + C\e^2.
\]
This shows that the error terms remain negligible compared to the main coercive contributions.

Thus, for sufficiently small $\e>0$, an application of Proposition \ref{prop_mod} allows us to incorporate the nonlocal interaction term, yielding
 \begin{align*}
 \begin{aligned}
&     \iint_{\R^2 \times \R^2} f^\e \log \lt( \frac{f^\e}{ M_{\rho, {\rm u_\e}}} \rt) dxd\xi + \frac1{2}\int_{\R^2} (\rho^\e - \rho) (-\Delta)^{-\alpha} (\rho^\e - \rho)\,dx  \cr
&\quad + \frac1{2\e}\int_0^t \iint_{\R^2 \times \R^2} \frac1{f^\e} |\nabla_\xi f^\e + (\xi - u^\e)f^\e|^2 dxd\xi ds + \frac1{4\e}\int_0^t \int_{\R^2} \rho^\e |u^\e - {\rm u_\e}|^2 \,dx ds  \cr
&\qquad \leq  \iint_{\R^2 \times \R^2} f^\e_0 \log \lt( \frac{f^\e_0}{ M_{\rho_0, {\rm u_{\e0}}}} \rt) dxd\xi  + \frac1{2}\int_{\R^2} (\rho^\e_0 - \rho_0) (-\Delta)^{-\alpha} (\rho^\e_0 - \rho_0)\,dx \cr
&\qquad \quad + C\e \lt(1 +\int_0^t \|f^\e(s)\|_{L^1_2}\,ds\rt)   + C\int_0^t \int_{\R^2} (\rho^\e - \rho) (-\Delta)^{-\alpha} (\rho^\e - \rho)\,dx ds. 
\end{aligned}
\end{align*}
This completes the proof.
\end{proof}

%
%
%
%
%
%

\subsubsection{Proof of Theorem \ref{thm_kin3}}  Using the kinetic energy bound in Lemma \ref{en_sqg} and applying Gr\"onwall's lemma provide the quantitative control
\begin{align*}
& \sup_{0 \leq t \leq T}\lt( \mathscr{H}[f^\e | M_{\rho, {\rm u_\e}}](t) + \mathscr{P}[\rho^\e | \rho](t) \rt)\cr
&\quad + \frac1{4\e}  \int_0^T \iint_{\R^2 \times \R^2} \frac1{f^\e} |\nabla_\xi f^\e + (\xi - u^\e)f^\e|^2 dxd\xi dt + \frac1{4\e} \int_0^T \int_{\R^2} \rho^\e |u^\e - {\rm u_\e}|^2 \,dx dt \cr
&\quad \leq  \mathscr{H}[f^\e | M_{\rho, {\rm u_\e}}](0) +   \mathscr{P}[\rho^\e | \rho](0) + C\e.
\end{align*}
Furthermore, by using the same argument as in Section \ref{ssec_diff}, we deduce
\begin{align}\label{con_sqg}
\begin{aligned}
&\sup_{0 \leq t \leq T}\lt( \mathscr{H}[f^\e | M_{\rho, 0}](t) + \mathscr{P}[\rho^\e | \rho](t) \rt) \cr
&\quad + \frac1{4\e}  \int_0^T  \iint_{\R^2 \times \R^2} \frac1{f^\e} |\nabla_\xi f^\e + (\xi - u^\e)f^\e|^2 dxd\xi dt + \frac1{4\e} \int_0^T  \int_{\R^2} \rho^\e |u^\e - {\rm u_\e}|^2 \,dx dt \cr
&\quad \leq C\lt( \mathscr{H}[f^\e | M_{\rho, 0}](0) + \mathscr{P}[\rho^\e | \rho](0)\rt) + C\e
\end{aligned}
\end{align}
for some $C>0$ independent of $\e$.

From the uniform bound \eqref{con_sqg}, we immediately obtain the strong convergence of both $f^\e$ and $\rho^\e$. To refine this into a quantitative estimate for the flux, we claim that
\begin{align*}
&{\rm d}_{\rm BL}^T \lt(\frac1\e \rho^\e u^\e, -\rho \nabla^\perp (-\Delta)^{-\alpha}\rho  - \nabla^\perp \rho \rt)\cr
&\quad \leq C \lt( \sqrt{\mathscr{H}[f^\e | M_{\rho, 0}](0)} + \sqrt{\mathscr{P}[\rho^\e | \rho](0)}  + \sqrt \e\rt) + C \lt(\mathscr{H}[f^\e | M_{\rho, 0}](0) +  \mathscr{P}[\rho^\e | \rho](0)  +  \e\rt),
\end{align*}
where the right-hand side involves only the initial perturbations and $\e$.  

Indeed, to prove this claim, we test against any $\phi \in W^{1,\infty}((0,T)\times \R^2)$ and expand
\[
\frac1\e \int_0^T\int_{\R^2} (\rho^\e u^\e - \rho {\rm u_\e})\cdot \phi \,dxdt = \frac1\e \int_0^T\int_{\R^2} \rho^\e (u^\e - {\rm u_\e})\cdot \phi \,dxdt + \int_0^T\int_{\R^2} (\rho^\e -\rho)\frac {\rm u_\e}\e \cdot \phi\,dxdt.
\]
Here, the second term on the right-hand side is more straightforward to control: by boundedness of $u$ and standard $L^1$ estimates on $\rho^\e-\rho$, we obtain
\begin{align*}
\int_0^T\int_{\R^2} (\rho^\e -\rho)\frac {\rm u_\e}\e \cdot \phi\,dxdt &\leq \|\phi\|_{L^\infty}\|\nabla^\perp (-\Delta)^{-\alpha} \rho   + \nabla^\perp \log\rho \|_{L^\infty}\int_0^T \|(\rho^\e - \rho)(t)\|_{L^1}\,dt\cr
&\leq C\|\phi\|_{L^\infty}\lt( \sqrt{\mathscr{H}[f^\e | M_{\rho, 0}](0)} + \sqrt{\mathscr{P}[\rho^\e | \rho](0)}  + \sqrt \e\rt).
\end{align*}

For the first term, which contains the main difficulty, we decompose it carefully into four pieces:
\begin{align*}
&\frac1\e \int_0^T\int_{\R^2} \rho^\e (u^\e - {\rm u_\e})\cdot \phi \,dxdt \\
&\quad = \frac1\e \int_0^T\int_{\R^2} \rho^\e (u^\e + \e \nabla^\perp (-\Delta)^{-\alpha}\rho^\e )\cdot \phi\,dxdt  - \int_0^T\int_{\R^2} \rho^\e \nabla^\perp(-\Delta)^{-\alpha}(\rho^\e-\rho) \cdot \phi\,dxdt\\
&\qquad   + \int_0^T \int_{\R^2} \rho^\e \nabla^\perp \log \rho \cdot\phi\,dxdt  \\
&\quad = \int_0^T \iint_{\R^2 \times \R^2} f^\e \lt( \frac{\xi^\perp}{\e} - \nabla(-\Delta)^{-\alpha}\rho^\e\rt) \cdot \phi^\perp\,dxd\xi dt + \int_0^T\int_{\R^2} (\rho^\e -\rho)\phi^\perp \cdot \nabla(-\Delta)^{-\alpha}(\rho^\e-\rho)\,dxdt \\
&\qquad +\int_0^T \int_{\R^2} \rho\phi^\perp \cdot \nabla(-\Delta)^{-\alpha}(\rho^\e -\rho)\,dxdt  + \int_0^T \int_{\R^2} \rho^\e \nabla^\perp \log \rho \cdot\phi\,dxdt  \\
&\quad =: I +II +III + IV.
\end{align*}

The first term $I$ is the most delicate, and by invoking the kinetic equation \eqref{kin_sqg}, we further split it into three pieces:
\begin{align*}
I &=  -\int_0^T  \iint_{\R^2 \times \R^2} \nabla_\xi \cdot \lt(f^\e \lt( \frac{\xi^\perp}{\e} - \nabla(-\Delta)^{-\alpha}\rho^\e\rt) \rt)\xi\cdot \phi^\perp\,dxd\xi dt\cr
&= \e \int_0^T  \iint_{\R^2 \times \R^2} \pa_t f^\e \xi \cdot \phi^\perp\,dxd\xi dt -\int_0^T  \iint_{\R^2 \times \R^2}  f^\e \xi\otimes \xi :\nabla \phi^\perp \,dxd\xi dt\\
&\quad +\int_0^T  \iint_{\R^2 \times \R^2} (\nabla_\xi f^\e + \xi f^\e)\cdot \phi^\perp\,dxd\xi dt \cr
&=: I_1 + I_2 + I_3,
\end{align*}
where each term can be treated using the entropy bound, moment estimates, and dissipation control as
\begin{align*}
I_1  &= \e \iint_{\R^2 \times \R^2} \xi f^\e(T) \cdot \phi^\perp(T)\,dxd\xi - \e\iint_{\R^2 \times \R^2} \xi f_0^\e \cdot \phi_0^\perp\,dxd\xi - \e \int_0^T  \iint_{\R^2 \times \R^2} \xi f^\e \cdot \pa_t \phi^\perp \,dxd\xi dt\\
&\le  \e \|\phi\|_{L^\infty}\int_{\R^2} \rho^\e |u^\e - {\rm u_\e}|\,dx\bigg|_{t=T} +C \e \|{\rm u_\e}\|_{L^\infty}\|\phi\|_{L^\infty} + \e \|\phi\|_{L^\infty}\int_{\R^2} \rho_0^\e |u_0^\e - {\rm u_{\e0}}|\,dx +C \e \|{\rm u_{\e0}}\|_{L^\infty}\|\phi\|_{L^\infty}\\
&\quad + \e \|\pa_t \phi\|_{L^\infty}\int_0^T \int_{\R^2} \rho^\e |u^\e - {\rm u_\e}|\,dxdt + C \e \|{\rm u_\e}\|_{L^\infty}\|\pa_t \phi\|_{L^\infty}\\
&\le C\|\phi\|_{W^{1,\infty}((0,T)\times\R^2)}( \sqrt{\mathscr{H}[f^\e | M_{\rho, 0}](0)} + \sqrt{\mathscr{P}[\rho^\e | \rho](0)}  + \sqrt \e + \e)\e,
\end{align*}
\begin{align*}
I_2 &= -\int_0^T  \iint_{\R^2 \times \R^2} f^\e (\xi \otimes \xi -\mathbb{I}_d) : \nabla \phi^\perp\,dxd\xi dt +\int_0^T \int_{\R^2} \rho^\e \nabla^\perp \cdot  \phi\,dxdt\\
&= -\int_0^T \iint_{\R^2 \times \R^2} (\xi f^\e +\nabla_\xi f^\e)\otimes \xi : \nabla \phi^\perp\,dxd\xi dt +\int_0^T \int_{\R^2} \rho^\e \nabla^\perp \cdot  \phi\,dxdt\\
&\le \|\phi\|_{W^{1,\infty}((0,T)\times\R^2)} \lt(\int_0^T \iint_{\R^2 \times \R^2} \frac{1}{f^\e}|\nabla_\xi f^\e +\xi f^\e|^2\,dxd\xi dt \rt)^{1/2} \lt(\int_0^T \iint_{\R^2 \times \R^2} |\xi|^2f^\e \,dxd\xi dt \rt)^{1/2} \\
&\quad + \int_0^T \int_{\R^2} (\rho^\e -\rho)\nabla^\perp\cdot \phi\,dxdt + \int_0^T \int_{\R^2} \rho \nabla^\perp \cdot \phi\,dxdt\\
&\le C\lt(\sqrt\e\lt( \sqrt{\mathscr{H}[f^\e | M_{\rho, 0}](0)} + \sqrt{\mathscr{P}[\rho^\e | \rho](0)}  + \sqrt \e\rt)  + \sup_{0\le t \le T}\|\rho^\e - \rho\|_{L^1}\rt)\|\phi\|_{W^{1,\infty}((0,T)\times\R^2)} \cr
&\quad - \int_0^T \int_{\R^2} \rho \nabla^\perp \log\rho \cdot \phi\,dxdt\cr
&\leq C\|\phi\|_{W^{1,\infty}((0,T)\times\R^2)}\lt( \sqrt{\mathscr{H}[f^\e | M_{\rho, 0}](0)} + \sqrt{\mathscr{P}[\rho^\e | \rho](0)}  + \sqrt \e\rt) - \int_0^T \int_{\R^2} \rho \nabla^\perp \log\rho \cdot \phi\,dxdt,
\end{align*}
and
\begin{align*}
I_3 &\le \|\phi\|_{L^\infty} \lt(\int_0^T  \iint_{\R^2 \times \R^2} \frac{1}{f^\e}|\nabla_\xi f^\e +\xi f^\e|^2\,dxd\xi dt\rt)^{1/2}\cr
& \le C\|\phi\|_{W^{1,\infty}((0,T)\times\R^2)}\sqrt\e\lt( \sqrt{\mathscr{H}[f^\e | M_{\rho, 0}](0)} + \sqrt{\mathscr{P}[\rho^\e | \rho](0)}  + \sqrt \e\rt) .
\end{align*}
Here we used
\begin{align*}
&\int_0^T\iint_{\R^2 \times \R^2} \frac{1}{f^\e}|\nabla_\xi f^\e +\xi f^\e|^2\,dxd\xi dt \cr
&\quad = \int_0^T\iint_{\R^2 \times \R^2} \frac{1}{f^\e}|\nabla_\xi f^\e +(\xi - u^\e) f^\e|^2\,dxd\xi dt + \int_0^T \int_{\R^2} \rho^\e |u^\e|^2\,dxdt\cr
&\quad \leq \int_0^T\iint_{\R^2 \times \R^2} \frac{1}{f^\e}|\nabla_\xi f^\e +(\xi - u^\e) f^\e|^2\,dxd\xi dt + 2 \int_0^T \int_{\R^2} \rho^\e |u^\e - {\rm u_\e}|^2\,dxdt +  2\int_0^T \int_{\R^2} \rho^\e|{\rm u_\e}|^2\,dxdt\cr
&\quad \leq C\e\lt( \mathscr{H}[f^\e | M_{\rho, 0}](0) + \mathscr{P}[\rho^\e | \rho](0)\rt) + C\e^2
\end{align*}
for some $C>0$ independent of $\e>0$.

For $II$ and $III$, we estimate
\begin{align*}
II &\le \|\nabla \phi\|_{L^\infty}\int_0^T \|\rho^\e - \rho\|_{\dot{H}^{-\alpha}}^2\,dt \le C\|\phi\|_{W^{1,\infty}((0,T)\times\R^2)}\lt(\mathscr{H}[f^\e | M_{\rho, 0}](0) +  \mathscr{P}[\rho^\e | \rho](0)  +  \e\rt)
\end{align*}
and
\[\begin{aligned}
III &\le \int_0^T\|\Lambda^{1-\alpha}(\rho\phi^\perp)\|_{L^2}\|(\rho^\e - \rho)(t)\|_{\dot{H}^{-\alpha}}\,dt \cr
&\le C\|\phi\|_{W^{1,\infty}((0,T)\times\R^2)}\lt( \sqrt{\mathscr{H}[f^\e | M_{\rho, 0}](0)} + \sqrt{\mathscr{P}[\rho^\e | \rho](0)}  + \sqrt \e\rt),
\end{aligned}\]
where we used the Kato-Ponce type inequality in \cite{L19} to get
\[
\begin{aligned}
\|\Lambda^{1-\alpha}(\rho\phi^\perp)\|_{L^2} &\le \|\Lambda^{1-\alpha}(\rho\phi^\perp) -   (\Lambda^{1-\alpha}\rho) \phi^\perp\|_{L^2} + \|(\Lambda^{1-\alpha}\rho)\phi^\perp\|_{L^2}\\
&\le C\|\rho\|_{L^2}\|\Lambda^{1-\alpha}\phi^\perp\|_{L^\infty} + C\|\rho\|_{\dot{H}^{1-\alpha}}\|\phi\|_{L^\infty}\\
&\le C\|\rho\|_{L^\infty(0,T;H^{1-\alpha})}\|\phi\|_{W^{1,\infty}((0,T)\times\R^2)}.
\end{aligned}
\]
We now combinine the above estimates, together with taking $\|\phi\|_{W^{1,\infty}((0,T)\times\R^2)} \le1$ and $\e \leq 1$  to obtain
\begin{align*}
&\frac1\e \lt|\int_0^T\int_{\R^2} \rho^\e (u^\e - {\rm u_\e})\cdot \phi \,dxdt\rt| \cr
&\quad \leq C \lt( \sqrt{\mathscr{H}[f^\e | M_{\rho, 0}](0)} + \sqrt{\mathscr{P}[\rho^\e | \rho](0)}  + \sqrt \e\rt) + C \lt(\mathscr{H}[f^\e | M_{\rho, 0}](0) +  \mathscr{P}[\rho^\e | \rho](0)  +  \e\rt)\cr
&\quad + \int_0^T \int_{\R^2} (\rho^\e - \rho) \nabla^\perp \log\rho \cdot \phi\,dxdt\cr
&\quad\leq C \lt( \sqrt{\mathscr{H}[f^\e | M_{\rho, 0}](0)} + \sqrt{\mathscr{P}[\rho^\e | \rho](0)}  + \sqrt \e\rt) + C \lt(\mathscr{H}[f^\e | M_{\rho, 0}](0) +  \mathscr{P}[\rho^\e | \rho](0)  +  \e\rt).
\end{align*}
Since $\phi$ was arbitrary, this establishes the quantitative convergence in BL-distance
\[
\frac1\e {\rm d}_{\rm BL}^T(\rho^\e u^\e, \rho {\rm u_\e})\leq C \lt( \sqrt{\mathscr{H}[f^\e | M_{\rho, 0}](0)} + \sqrt{\mathscr{P}[\rho^\e | \rho](0)}  + \sqrt \e\rt) + C \lt(\mathscr{H}[f^\e | M_{\rho, 0}](0) +  \mathscr{P}[\rho^\e | \rho](0)  +  \e\rt).
\]

 To derive a quantitative estimate in a negative Sobolev topology, we test the equation for $\rho^\e-\rho$ along the characteristic flow generated by the limiting velocity field. This requires uniform Sobolev control of the associated flow map, together with composition estimates for test functions transported by the flow. We thus record the following auxiliary lemma.

\begin{lemma}\label{lem_flow_comp_sqg}
Let $\gamma\ge4$, and let
\[
b\in L^\infty (0,T;H^\gamma(\R^2;\R^2)\cap W^{2,\infty}(\R^2;\R^2))
\]
be a divergence-free vector field. Let $X(s,t;x)$ be the characteristic flow associated with $b$, namely,
\[
\frac{d}{ds}X(s,t;x)=b(s,X(s,t;x)),\quad X(t,t;x)=x.
\]
Then, for every $0\le s,t\le T$, the map $X(s,t,\cdot)$ is a $C^1$-diffeomorphism of $\R^2$ satisfying
\[
X(s,t,\cdot)-{\rm id}\in H^\gamma(\R^2), \quad \det \nabla_x X(s,t;x)=1 \quad \text{for all }x\in\R^2,
\]
and
\[
\|X(s,t,\cdot)-{\rm id}\|_{H^\gamma} +\|\nabla_x X(s,t,\cdot)\|_{W^{1,\infty}} \le C\lt(\gamma,T,\|b\|_{L^\infty(0,T;H^\gamma\cap W^{2,\infty})}\rt).
\]
Consequently, for every $\varphi\in H^\gamma(\R^2)$,
\[
\|\varphi\circ X(s,t,\cdot)\|_{H^\gamma} \le C\lt(\gamma,T,\|b\|_{L^\infty(0,T;H^\gamma\cap W^{2,\infty})}\rt)\|\varphi\|_{H^\gamma},
\]
and
\[
\|\nabla^2(\varphi\circ X(s,t,\cdot))\|_{L^\infty} \le C\lt(\gamma,T,\|b\|_{L^\infty(0,T;H^\gamma\cap W^{2,\infty})}\rt)\|\varphi\|_{H^\gamma}.
\]
\end{lemma}

\begin{proof}
Since $b\in L^\infty(0,T;W^{1,\infty}(\R^2;\R^2))$, the classical ODE theory yields a unique $C^1$-flow $X(s,t;\cdot)$ on $\R^2$. Moreover, since $\nabla\cdot b=0$, Liouville's formula implies
\[
\det \nabla_x X(s,t;x)=1.
\]

Next, since $b\in L^\infty(0,T;H^\gamma(\R^2;\R^2))$ with $\gamma>2$, the regularity theory for flows of Sobolev vector fields yields
\[
X(s,t,\cdot)-{\rm id}\in H^\gamma(\R^2)
\]
for all $0\le s,t\le T$, together with the estimate
\[
\|X(s,t,\cdot)-{\rm id}\|_{H^\gamma} \le C\lt(\gamma,T,\|b\|_{L^\infty(0,T;H^\gamma)}\rt).
\]

Differentiating the flow equation in $x$, we obtain
\[
\partial_s \nabla_x X(s,t;x)=\nabla b(s,X(s,t;x))\,\nabla_x X(s,t;x),
\]
and hence
\[
\|\nabla_x X(s,t,\cdot)\|_{L^\infty} \le \exp\lt(\int_s^t \|\nabla b(\tau)\|_{L^\infty}\,d\tau\rt) \le C\lt(T,\|b\|_{L^\infty(0,T;W^{1,\infty})}\rt).
\]

Differentiating once more, for each $1\le i,j,k\le2$ we have
\[\begin{aligned}
\partial_s \partial_{ij}X^k(s,t;x) &= \sum_{\ell,m=1}^2 \partial_{\ell m}b^k(s,X(s,t;x)) \,\partial_i X^\ell(s,t;x)\,\partial_j X^m(s,t;x) \\
&\quad +\sum_{\ell=1}^2 \partial_\ell b^k(s,X(s,t;x))\,\partial_{ij}X^\ell(s,t;x).
\end{aligned}\]
Therefore, by Gr\"onwall's lemma,
\[
\|\nabla_x^2 X(s,t,\cdot)\|_{L^\infty} \le C\lt(T,\|b\|_{L^\infty(0,T;W^{2,\infty})}\rt),
\]
and thus
\[
\|\nabla_x X(s,t,\cdot)\|_{W^{1,\infty}} \le C\lt(\gamma,T,\|b\|_{L^\infty(0,T;H^\gamma\cap W^{2,\infty})}\rt).
\]

Now, since
\[
X(s,t,\cdot)-{\rm id}\in H^\gamma(\R^2), \quad \det\nabla_x X(s,t;x)=1,
\]
the composition estimate of \cite[Lemma 2.7]{IKT13} implies
\[
\|\varphi\circ X(s,t,\cdot)\|_{H^\gamma}
\le C\lt(\gamma,T,\|b\|_{L^\infty(0,T;H^\gamma\cap W^{2,\infty})}\rt)\|\varphi\|_{H^\gamma}.
\]

Finally, by the chain rule, for each $1\le i,j\le2$,
\[
\partial_{ij}(\varphi\circ X) = \sum_{k,\ell=1}^2 (\partial_{k\ell}\varphi)(X)\,\partial_i X^k\,\partial_j X^\ell
+\sum_{k=1}^2 (\partial_k\varphi)(X)\,\partial_{ij}X^k,
\]
and hence
\[\begin{aligned}
\|\nabla^2(\varphi\circ X)\|_{L^\infty} &\le C\lt(\|\nabla^2\varphi\|_{L^\infty}\|\nabla X\|_{L^\infty}^2 +\|\nabla\varphi\|_{L^\infty}\|\nabla^2 X\|_{L^\infty}\rt)\\
&\le C\lt(\gamma,T,\|b\|_{L^\infty(0,T;H^\gamma\cap W^{2,\infty})}\rt)\|\varphi\|_{H^\gamma},
\end{aligned}\]
where we used the embedding $H^\gamma(\R^2)\hookrightarrow W^{2,\infty}(\R^2)$.
This completes the proof.
\end{proof}
 
 In the present setting, we apply Lemma \ref{lem_flow_comp_sqg} with
\[
b(t,x):=-\nabla^\perp(-\Delta)^{-\alpha}\rho(t,x).
\]
Assume, in addition, that
\[
\rho\in L^\infty\lt(0,T;L^1(\R^2)\cap H^{\gamma+1}(\R^2)\rt).
\]
We claim that, for every $\alpha\in[\frac12,1)$,
\[
\|b(t)\|_{H^\gamma}+\|b(t)\|_{W^{2,\infty}} \le C_{\alpha,\gamma}\lt(\|\rho(t)\|_{L^1}+\|\rho(t)\|_{H^{\gamma+1}}\rt)
\]
for a.e. $t\in[0,T]$. Indeed, by Plancherel's theorem,
\[\begin{aligned}
\|b(t)\|_{H^\gamma}^2 &= \int_{\R^2} (1+|\eta|^2)^\gamma |\eta|^{2-4\alpha} |\hat\rho(t,\eta)|^2\,d\eta \\
&= \int_{|\eta|\le1} (1+|\eta|^2)^\gamma |\eta|^{2-4\alpha} |\hat\rho(t,\eta)|^2\,d\eta
 + \int_{|\eta|>1} (1+|\eta|^2)^\gamma |\eta|^{2-4\alpha} |\hat\rho(t,\eta)|^2\,d\eta.
\end{aligned}\]
For the low-frequency part, using $\|\hat\rho(t)\|_{L^\infty}\le \|\rho(t)\|_{L^1}$, we get
\[\begin{aligned}
\int_{|\eta|\le1} (1+|\eta|^2)^\gamma |\eta|^{2-4\alpha} |\hat\rho(t,\eta)|^2\,d\eta
&\le C \|\rho(t)\|_{L^1}^2 \int_{|\eta|\le1} |\eta|^{2-4\alpha}\,d\eta  \le C_\alpha \|\rho(t)\|_{L^1}^2,
\end{aligned}\]
where the last integral is finite precisely since $\alpha<1$ in two dimensions.

For the high-frequency part, since $\alpha\ge\frac12$, we get $2-4\alpha\le0$, and thus $|\eta|^{2-4\alpha}\le1$ for $|\eta|>1$. Hence, we have
\[\begin{aligned}
\int_{|\eta|>1} (1+|\eta|^2)^\gamma |\eta|^{2-4\alpha} |\hat\rho(t,\eta)|^2\,d\eta
&\le \int_{|\eta|>1} (1+|\eta|^2)^\gamma |\hat\rho(t,\eta)|^2\,d\eta  \le \|\rho(t)\|_{H^\gamma}^2
\le \|\rho(t)\|_{H^{\gamma+1}}^2.
\end{aligned}\]
Combining the above bounds, we infer that
\[
\|b(t)\|_{H^\gamma}
\le C_{\alpha,\gamma}\lt(\|\rho(t)\|_{L^1}+\|\rho(t)\|_{H^{\gamma+1}}\rt).
\]

Since $\gamma\ge4$, the Sobolev embedding $H^\gamma(\R^2)\hookrightarrow W^{2,\infty}(\R^2)$ yields
\[
\|b(t)\|_{W^{2,\infty}} \le C_\gamma \|b(t)\|_{H^\gamma} \le C_{\alpha,\gamma}\lt(\|\rho(t)\|_{L^1}+\|\rho(t)\|_{H^{\gamma+1}}\rt).
\]
Hence
\[
b\in L^\infty (0,T;H^\gamma(\R^2;\R^2)\cap W^{2,\infty}(\R^2;\R^2)),
\]
and Lemma \ref{lem_flow_comp_sqg} applies.

We now derive a quantitative estimate for $\rho^\e-\rho$ in a negative Sobolev topology. 
To this end, we assume that $\alpha\in[\frac12,1)$ and fix an integer $\gamma\ge4$.
Since $L^1(\R^2)\subset H^{-s}(\R^2)$ for every $s>1$, we have
\[
\rho^\e(t),\,\rho(t)\in L^1(\R^2)\subset H^{-\gamma}(\R^2),
\quad t\in[0,T].
\]
Assume moreover that
\[
\rho\in L^\infty(0,T;H^{\gamma+1}(\R^2)).
\]
Let  $b(t,x):=-\nabla^\perp(-\Delta)^{-\alpha}\rho(t,x)$, and let $X(s,t;x)$ be the characteristic flow associated with $b$, namely,
\[
\frac{d}{ds}X(s,t;x)=b(s,X(s,t;x)),\quad X(t,t;x)=x.
\]
Since $\nabla\cdot b=0$, the flow is measure-preserving:
\[
\det \nabla_x X(s,t;x)=1.
\]
Moreover, by Lemma \ref{lem_flow_comp_sqg},
\[
\|\varphi(X(s,t;\cdot))\|_{H^\gamma}
+\|\nabla^2(\varphi(X(s,t;\cdot)))\|_{L^\infty}
\le C\|\varphi\|_{H^\gamma}
\]
for all $\varphi\in H^\gamma(\R^2)$, where $C=C(\gamma,T,\|\rho\|_{L^\infty(0,T;H^{\gamma+1})})$.

On the other hand, arguing as in the diffusive case, one checks that $\rho^\e$ satisfies
\[
\partial_t \rho^\e - \nabla\cdot\lt(\rho^\e \nabla^\perp(-\Delta)^{-\alpha}\rho^\e\rt) = -\nabla\cdot F^\e
\]
in the sense of distributions, where
\[
F^\e:=\e\partial_t (m^\e)^\perp -\nabla^\perp\!\cdot\!\int_{\R^2}(\xi\otimes\xi-\mathbb I_2)f^\e\,d\xi -(m^\e)^\perp.
\]
Here, for a matrix field $A=(A_{ij})_{1\le i,j\le2}$, we define
\[
(\nabla^\perp \cdot A)_j := -\pa_2 A_{2j} + \pa_1 A_{1j}, \quad j=1,2.
\]
Testing the equation for $\rho^\e-\rho$ against $\varphi(X(s,t;\cdot))$ and integrating along the flow yields
\begin{align*}
\int_{\R^2}(\rho^\e-\rho)(t,x)\varphi(x)\,dx &= \int_{\R^2}(\rho_0^\e-\rho_0)(x)\varphi(X(0,t;x))\,dx \\
&\quad -\int_0^t\int_{\R^2}\rho^\e \nabla^\perp(-\Delta)^{-\alpha}(\rho^\e-\rho)\cdot \nabla[\varphi(X(s,t;x))]\,dxds \\
&\quad -\int_0^t\int_{\R^2}F^\e(s,x)\cdot \nabla[\varphi(X(s,t;x))]\,dxds \\
&=: J_1 + J_2 + J_3.
\end{align*}

For the initial term, Lemma \ref{lem_flow_comp_sqg} gives
\[
|J_1| \le \|\rho_0^\e-\rho_0\|_{H^{-\gamma}} \|\varphi(X(0,t;\cdot))\|_{H^\gamma} \le C\|\rho_0^\e-\rho_0\|_{H^{-\gamma}}\|\varphi\|_{H^\gamma}.
\]

For $J_2$, we decompose
\begin{align*}
J_2 &= -\int_0^t\int_{\R^2}(\rho^\e-\rho) \nabla^\perp(-\Delta)^{-\alpha}(\rho^\e-\rho)\cdot \nabla[\varphi(X)]\,dxds \\
&\quad -\int_0^t\int_{\R^2}\rho\, \nabla^\perp(-\Delta)^{-\alpha}(\rho^\e-\rho)\cdot \nabla[\varphi(X)]\,dxds \\
&=: J_{21}+J_{22}.
\end{align*}
Using that $\nabla^\perp(-\Delta)^{-\alpha}(\rho^\e-\rho)$ is divergence-free, integrating by parts once, and using
\[
\|\nabla^2(\varphi(X))\|_{L^\infty}\le C\|\varphi\|_{H^\gamma},
\]
by Lemma \ref{lem_flow_comp_sqg}, we obtain
\[
|J_{21}| \le C\|\varphi\|_{H^\gamma} \int_0^t \|\rho^\e-\rho\|_{\dot H^{-\alpha}}^2\,ds.
\]
For $J_{22}$, since $\alpha\in[\frac12,1)$, by using the same argument as in the proof of Lemma \ref{lem_flow_comp_sqg}, we find
\begin{align*}
\|(-\Delta)^{-\alpha}\nabla^\perp \cdot G\|_{H^\gamma}^2&\le \int_{\R^2} (1+|\eta|^2)^\gamma |\eta|^{2-4\alpha} |\hat G (\eta)|^2\,d\eta \\
&\le C\|G\|_{L^1}^2\int_{|\eta|\le 1}|\eta|^{2-4\alpha}\,d\eta +C\int_{|\eta|>1} (1+|\eta|^2)^\gamma|\hat G(\eta)|^2\,d\eta\\
&\le C(\|G\|_{L^1}^2 + \|G\|_{H^\gamma}^2).
\end{align*}
Moreover, since $\gamma>1$ in dimension two, $H^\gamma(\R^2)$ is an algebra, and thus
\[
\|\nabla\rho\,\varphi(X)\|_{L^1} +\|\nabla\rho\,\varphi(X)\|_{H^\gamma} \le C\|\rho\|_{H^{\gamma+1}}\|\varphi(X)\|_{H^\gamma}.
\]
Hence, we have
\begin{align*}
|J_{22}| &= \lt| \int_0^t\int_{\R^2}
(\rho^\e-\rho)
(-\Delta)^{-\alpha} \nabla^\perp \cdot [\nabla\rho \,\varphi(X)]\,dxds \rt|\cr
&\le
C\|\rho\|_{L^\infty(0,T;H^{\gamma+1})}
\|\varphi(X)\|_{L^\infty(0,T;H^\gamma)}
\int_0^t \|\rho^\e-\rho\|_{H^{-\gamma}}\,ds \cr
&\le
C\|\varphi\|_{H^\gamma}
\int_0^t \|\rho^\e-\rho\|_{H^{-\gamma}}\,ds.
\end{align*}

For $J_3$, we decompose according to the three components of $F^\e$:
\begin{align*}
J_3 &= \e \int_0^t \int_{\R^2} \pa_s m^\e \cdot \nabla^\perp [\varphi(X(s,t;x))]\,dxds\\
&\quad -\int_0^t \int_{\R^2} \lt( \int_{\R^2} (\xi\otimes\xi-\mathbb{I}_2)f^\e \,d\xi\rt): \nabla^\perp\otimes \nabla [\varphi(X(s,t;x))]\,dxds\\
&\quad -\int_0^t \int_{\R^2} m^\e \cdot \nabla^\perp[\varphi(X(s,t;x))]\,dxds\\
&=: J_{31} + J_{32} + J_{33}.
\end{align*}
Integrating by parts in time, we obtain
\begin{align*}
J_{31} &= \e \int_{\R^2} m^\e \cdot \nabla^\perp\varphi\,dx - \e \int_{\R^2} m_0^\e \cdot\nabla^\perp[\varphi(X(0,t;x))]\,dx - \e \int_0^t \int_{\R^2} m^\e \pa_s \nabla^\perp[\varphi(X(s,t;x))]\,dxds\\
&\le C\e\|\varphi\|_{H^\gamma}\sup_{0\le t \le T}\int_{\R^2} \lt(\rho^\e |u^\e -{\rm u}_\e|+ \rho^\e |{\rm u}_\e|\rt)\,dx  + C\e\|\varphi\|_{H^\gamma}\int_{\R^2} \rho_0^\e |u_0^\e -{\rm u_0}_\e|\,dx + C\e \int_{\R^2} \rho_0^\e |{\rm u_0}_\e|\,dx\\
&\le C\e \lt( \mathscr{H}[f^\e | M_{\rho, 0}](0) + \mathscr{P}[\rho^\e | \rho](0) +\e\rt)^{\frac12}\|\varphi\|_{H^\gamma}.
\end{align*}
Here we used
\[\begin{aligned}
|\pa_s \nabla^\perp [\varphi(X(s,t;x))]| &=\lt|\pa_s \lt( \nabla X(s,t;x) \cdot (\nabla^\perp \varphi)(X(s,t;x))\rt)\rt|\\
&\le C\|\nabla \nabla^\perp(-\Delta)^{-\alpha}\rho\|_{L^\infty} \|\nabla\varphi\|_{L^\infty} + \|\nabla X(s,t;\cdot)\|_{L^\infty}\|\nabla^2 \varphi\|_{L^\infty}\|\nabla^\perp(-\Delta)^{-\alpha}\rho\|_{L^\infty}\\
&\le C\|\varphi\|_{H^\gamma}.
\end{aligned}\]

For $J_{32}$, rewriting it by the integration by parts in $\xi$, we estimate
\begin{align*}
J_{32} &= \int_0^t \int_{\R^2} \lt(\int_{\R^2} (\xi f^\e +\nabla_\xi f^\e)\otimes \xi\,d\xi \rt):\nabla^\perp \otimes \nabla [\varphi(X(s,t;x))]\,dxds\\
&\le \|\varphi\|_{H^\gamma}\lt(\int_0^t\iint_{\R^2\times\R^2} \frac{1}{f^\e} |\nabla_\xi f^\e + \xi f^\e|^2\,dxd\xi ds\rt)^{\frac12} \lt( \int_0^t \int_{\R^2\times\R^2} |\xi|^2 f^\e\,dxd\xi ds\rt)^{\frac12}\\
&\le C\e^{\frac12} \lt(\mathscr{H}[f^\e | M_{\rho, 0}](0) + \mathscr{P}[\rho^\e | \rho](0)+ \e\rt)^{\frac12}\|\varphi\|_{H^\gamma},
\end{align*}
where we used
\begin{align*}
&\int_0^t\iint_{\R^2 \times \R^2} \frac{1}{f^\e}|\nabla_\xi f^\e +\xi f^\e|^2\,dxd\xi ds \cr
&\quad = \int_0^t\iint_{\R^2 \times \R^2} \frac{1}{f^\e}|\nabla_\xi f^\e +(\xi - u^\e) f^\e|^2\,dxd\xi ds + \int_0^t \int_{\R^2} \rho^\e |u^\e|^2\,dxds\cr
&\quad \leq \int_0^t\iint_{\R^2 \times \R^2} \frac{1}{f^\e}|\nabla_\xi f^\e +(\xi - u^\e) f^\e|^2\,dxd\xi ds + 2 \int_0^t \int_{\R^2} \rho^\e |u^\e - {\rm u_\e}|^2\,dxds +  2\int_0^t \int_{\R^2} \rho^\e|{\rm u_\e}|^2\,dxds\cr
&\quad \leq C\e\lt( \mathscr{H}[f^\e | M_{\rho, 0}](0) + \mathscr{P}[\rho^\e | \rho](0)\rt) + C\e^2,
\end{align*}
due to \eqref{con_sqg}. 

Finally, using \eqref{con_sqg} again, we deduce
\begin{align*}
J_{33}&\le C\|\varphi\|_{H^\gamma} \int_0^t \int_{\R^2} \lt(\rho^\e |u^\e -{\rm u}_\e| + \rho^\e |{\rm u}_\e|\rt)\,dxds\\
&\le C\|\varphi\|_{H^\gamma} \lt(\int_0^t \int_{\R^2} \rho^\e |u^\e -{\rm u}_\e|^2\,dxds\rt)^{\frac12} + C\e \|\varphi\|_{H^\gamma}\\
&\le C\e^{\frac12} \lt(\mathscr{H}[f^\e | M_{\rho, 0}](0) + \mathscr{P}[\rho^\e | \rho](0)+ \e\rt)^{\frac12}\|\varphi\|_{H^\gamma}.
\end{align*}
Thus, we collect all the estimates for $J_i$'s to have
\[
\|(\rho^\e -\rho)(t)\|_{H^{-\gamma}}  \le \|\rho_0^\e - \rho_0\|_{H^{-\gamma}}  + C\lt(\mathscr{H}[f^\e | M_{\rho, 0}](0) + \mathscr{P}[\rho^\e | \rho](0)+ \e\rt) + C\int_0^t \|(\rho^\e -\rho)(s)\|_{H^{-\gamma}}\,ds.
\]
An application of Gr\"onwall's lemma therefore yields
\[
\|(\rho^\e -\rho)(t)\|_{H^{-\gamma}} \le C\lt(\|\rho_0^\e - \rho_0\|_{H^{-\gamma}} + \mathscr{H}[f^\e | M_{\rho, 0}](0) + \mathscr{P}[\rho^\e | \rho](0)+ \e\rt),
\]
which gives the desired quantitative estimate for $\rho^\e$.  This completes the proof.
   
\begin{remark}\label{rmk_opt_gsqg} Recall ${\rm u}_\e:= - \e\lt(\nabla^\perp\log\rho + \nabla^\perp(-\Delta)^{-\alpha}\rho\rt)$. Then Corollary \ref{cor_gd2}, together with \eqref{con_sqg} and the above $H^{-\gamma}$ estimate for $\rho^\e-\rho$, yields
\begin{align*}
\int_0^T \|f-M_{\rho, {\rm u}_\e}\|_{H_{x,\xi}^{-\sigma}}^2\,dt &\le C\int_0^T\iint_{\R^d \times \R^d} \frac1{f^\e}|\nabla_\xi f^\e + (\xi-u^\e) f^\e|^2\,dxd\xi dt +C\int_0^T \intr \rho^\e |u^\e - {\rm u}_\e|^2\,dxdt  \cr
&\quad + C\|{\rm u}_\e\|_{L^\infty((0,T)\times\R^d)}^2 \int_0^T \|\rho^\e - \rho\|_{L^1}^2\,dt  +C\int_0^T \|\rho^\e - \rho\|_{H^{-\sigma}}^2\,dt\\
&\le C\lt( \|\rho_0^\e -\rho_0\|_{H^{-\gamma}} + \mathscr{H}[f^\e | M_{\rho, 0}](0)  + \mathscr{P}[\rho^\e | \rho](0) + \e\rt)^2
\end{align*}
for any $\sigma \ge \gamma$. In particular, the shifted Maxwellian $M_{\rho,{\rm u}_\e}$ provides a rigorously justified refined approximation of the kinetic distribution in a weak topology.
\end{remark}

%
%
%
%
%
%

\section*{Acknowledgments}
The authors would like to thank Jos\'e Carrillo for his valuable comments on this work. The work of Y.-P. Choi was supported by NRF grant no. 2022R1A2C1002820 and RS-2024-00406821. The work of J. Jung was supported by NRF grant no. 2022R1A2C1002820.

%
%
%
%
%
%


%
%
%
%
%
%
\appendix

\section{Formal derivation of the gSQG limit}\label{app:gsqg}

We derive the gSQG equation \eqref{eq_sqg} from the kinetic model \eqref{kin_sqg_sca} in two dimensions, emphasizing the role of fast perpendicular rotations in velocity space. A direct closure of the moment equations does not capture the correct structure under this scaling, because velocity transport and forcing terms appear at the same order as the linear Fokker--Planck operator. The appropriate viewpoint is to separate the fast dynamics generated by the transverse rotation from the slower contributions and to average along the fast orbits. Writing \eqref{kin_sqg_sca} as
\[
\e  \pa_t f^\e + \xi \cdot \nabla f^\e = \nabla_\xi \cdot (\nabla_\xi f^\e + \xi f^\e) + \nabla(-\Delta)^{-\alpha}\rho^\e \cdot \nabla_\xi f^\e - \frac{1}{\e}\,\xi^\perp \cdot \nabla_\xi f^\e,
\]
we single out the fast rotation operator
\[
L_{-1} f := \xi^\perp \cdot \nabla_\xi f,
\]
which generates rotations in the velocity variable $\xi\in\R^2$. In polar coordinates $\xi=(r\cos\theta,r\sin\theta)$, we get $L_{-1}= \pa_\theta$, thus
\[
L_{-1} f=0 \quad \Longleftrightarrow \quad f(\xi)=\varphi(|\xi|) \quad\text{with a radial profile }\varphi.
\]

Equivalently, if $\Pi$ denotes the angular averaging operator
\[
(\Pi g)(x,\xi) := \frac{1}{2\pi}\int_0^{2\pi} g(x,R_\theta \xi)\,d\theta,
\]
where $R_\theta$ denotes the counterclockwise rotation in $\R^2$ by angle $\theta$, then $\Pi$ is the orthogonal projection onto $\ker L_{-1}$. In particular, $\Pi^2=\Pi$, we have $\Pi L_{-1}=0$, and $L_{-1}$ is invertible on the orthogonal complement $(I-\Pi)L^2_\xi$ of mean-zero angular functions.

We consider a Hilbert expansion $f^\e=f^0+\e f^1+O(\e^2)$. The leading constraint $L_{-1}f^0=0$ enforces radial symmetry in $\xi$. Accordingly, we set an ansatz for $f_0$ as 
\[
f^0(t,x,\xi)= \rho^0 M_{1,0}(\xi), \quad \rho^0 := \int_{\R^2} f^0\,d\xi,
\]
which is consistent with mass conservation and serves as the natural local equilibrium. Then we observe $\nabla_\xi f^0=-\xi f^0$ and $\nabla_\xi\cdot(\nabla_\xi f^0+\xi f^0)=0$. At order $O(1)$, we obtain
\[
- L_{-1} f^1 + \xi\cdot\nabla f^0 - \nabla(-\Delta)^{-\alpha}\rho^0 \cdot \nabla_\xi f^0 - \nabla_\xi \cdot (\nabla_\xi f^0 + \xi f^0)=0, 
\]
hence, using the identities above,
\[
L_{-1} f^1 =   (\xi\cdot\nabla \rho^0)M_{1,0} + (\rho^0 \nabla(-\Delta)^{-\alpha}\rho^0\cdot\xi)M_{1,0}.
\]
Equivalently,
\[
L_{-1} f^1 =  \bigl(\xi\cdot(\nabla \rho^0 + \rho^0\nabla(-\Delta)^{-\alpha}\rho^0)\bigr)M_{1,0}.
\]

Since $L_{-1}$ is not invertible on its whole domain, a solvability condition is required: the right-hand side must be orthogonal to $\ker L_{-1}$, which consists of radial functions. This is precisely enforced by applying the angular averaging operator $\Pi$, leading to the macroscopic condition
\[
\Pi\bigl(\xi\cdot\nabla f^0 - \nabla(-\Delta)^{-\alpha}\rho^0\cdot\nabla_\xi f^0\bigr)=0,
\]
which actually holds for $f_0$ we set above.  On the complementary subspace $(I-\Pi)L^2_\xi$, $L_{-1}$ is invertible, so that the non-radial part of the right-hand side uniquely determines $f^1$ up to a radial component. Using
\[
L_{-1}\big((\xi\cdot a)M_{1,0}\big)= (\xi^\perp\cdot a)M_{1,0}
\]
for any constant vector $a$, we may choose the corrector
\[
f^1 = - \bigl((\nabla \rho^0 + \rho^0 \nabla (-\Delta)^{-\alpha}\rho^0)^\perp \cdot \xi\bigr)\,M_{1,0},
\]
which is uniquely defined up to an arbitrary radial function. The macroscopic flux is
\[
m^1(t,x):=\int_{\R^2}\xi f^1\,d\xi = - (\nabla \rho^0 + \rho^0 \nabla (-\Delta)^{-\alpha}\rho^0)^\perp \int_{\R^2}\xi\otimes\xi\,M_{1,0}\,d\xi
=
- \nabla^\perp \rho^0 - \rho^0 \nabla^\perp (-\Delta)^{-\alpha}\rho^0.
\]
Inserting this into the continuity equation
\[
\pa_t\rho^0+\nabla\cdot m^1=0
\]
yields
\[
\pa_t \rho^0 - \nabla\cdot(\nabla^\perp \rho^0) - \nabla \cdot  (\rho^0 \nabla^\perp (-\Delta)^{-\alpha}\rho^0 )=0.
\]
Since $\nabla\cdot\nabla^\perp \equiv 0$ in two dimensions, the intermediate term involving $\nabla^\perp \rho^0$ vanishes, and therefore
\[
\pa_t \rho^0 - \nabla \cdot  (\rho^0 \nabla^\perp (-\Delta)^{-\alpha}\rho^0 )=0,
\]
which coincides with \eqref{eq_sqg}.

We stress that a term involving $\nabla^\perp \rho^0$ genuinely appears in the kinetic corrector and in the intermediate flux. It disappears only after applying the divergence in the continuity equation, thanks to the identity $\nabla\cdot\nabla^\perp=0$ in two dimensions. This cancellation explains why moment equations alone are not sufficient to capture the correct structure, and why the fast rotation averaging coupled with the projection $\Pi$ provides the natural framework for the gSQG limit.

%
%
%
%
%
%

\section{Formal optimal convergence rates}\label{app_optimal}

In this appendix, we discuss the formal convergence rates suggested by the asymptotic expansions in the three singular regimes considered in the paper.

%
%
%
%
%
%
\subsection{Diffusive limit}\label{app_optimal_diff}

We first consider the diffusive scaling
\[
\e \pa_t f^\e + \xi \cdot \nabla f^\e - \nabla (-\Delta)^{-\alpha}\rho^\e \cdot \nabla_\xi f^\e = \frac1\e \nabla_\xi \cdot (\nabla_\xi f^\e + \xi f^\e),
\]
where $\rho^\e=\intr f^\e\,d\xi$. We seek a formal expansion of the form
\[
f^\e = f_0 + \e f_1 + \e^2 f_2 + \cdots.
\]
Substituting this ansatz into the equation and comparing powers of $\e$, the leading-order term gives
\[
\nabla_\xi \cdot (\nabla_\xi f_0 + \xi f_0)=0.
\]
Hence $f_0$ lies in the kernel of the Fokker--Planck operator, and thus must be of the form
\[
f_0=M_{\rho,0}=\frac{\rho(t,x)}{(2\pi)^{\frac d2}}e^{-\frac{|\xi|^2}{2}}.
\]

To determine the first-order correction, we collect the terms of order $O(1)$:
\[
\xi \cdot \nabla f_0 - \nabla (-\Delta)^{-\alpha}\rho \cdot \nabla_\xi f_0 = \nabla_\xi \cdot (\nabla_\xi f_1 + \xi f_1).
\]
Using
\[
\nabla f_0 = M_{1,0}\nabla \rho, \quad \nabla_\xi f_0 = -\xi M_{\rho,0},
\]
we obtain
\[
\xi \cdot \nabla f_0 - \nabla (-\Delta)^{-\alpha}\rho \cdot \nabla_\xi f_0 = \xi \cdot \nabla \rho\, M_{1,0} + \rho \, \xi \cdot \nabla (-\Delta)^{-\alpha}\rho \, M_{1,0}.
\]
Equivalently,
\[
\xi \cdot \nabla f_0 - \nabla (-\Delta)^{-\alpha}\rho \cdot \nabla_\xi f_0 = \xi \cdot \lt(\nabla \rho + \rho \nabla (-\Delta)^{-\alpha}\rho\rt) M_{1,0}.
\]

Since
\[
\nabla_\xi \cdot \lt(\nabla_\xi (\xi_j M_{1,0}) + \xi (\xi_j M_{1,0})\rt) = -\xi_j M_{1,0},
\]
it follows formally that
\[
f_1=-\xi \cdot \lt(\nabla \rho + \rho \nabla (-\Delta)^{-\alpha}\rho\rt) M_{1,0}.
\]
Recalling that
\[
\nabla \rho + \rho \nabla (-\Delta)^{-\alpha}\rho = \rho \lt(\nabla \log \rho + \nabla (-\Delta)^{-\alpha}\rho\rt),
\]
we may also write
\[
f_1=-\rho\,\xi \cdot \lt(\nabla \log \rho + \nabla (-\Delta)^{-\alpha}\rho\rt) M_{1,0}.
\]

A crucial observation is that
\[
\intr f_1\,d\xi =0,
\]
since $f_1$ is odd in $\xi$. Therefore, the first-order correction does not contribute to the density variable. As a consequence, the expansion suggests the improved density approximation
\[
\rho^\e=\rho+O(\e^2).
\]
On the other hand, the momentum satisfies
\[
m^\e:=\intr\xi f^\e\,d\xi = \e \intr\xi f_1\,d\xi + O(\e^2),
\]
and hence
\[
\frac1\e m^\e = -\rho \lt(\nabla \log \rho + \nabla (-\Delta)^{-\alpha}\rho\rt)+O(\e).
\]

We are thus led to the formal convergence rates
\[
f^\e-M_{\rho,0}=O(\e), \quad
\rho^\e-\rho=O(\e^2),
\]
and
\[
\frac1\e m^\e
+\rho \Bigl(\nabla (-\Delta)^{-\alpha}\rho + \nabla \log \rho\Bigr)
=O(\e).
\]
In particular, the density variable gains one order compared to the kinetic distribution itself. This additional order is due to the oddness of the first-order corrector in the velocity variable and is consistent with the parabolic structure of the diffusive limit. The estimates obtained in Theorem \ref{thm_kin1}, especially the quantitative control in a negative Sobolev topology, are therefore compatible with the formally optimal diffusive scaling.

Concerning the optimal rate toward the local Maxwellian state $M_{\rho,{\rm u}_\e}$, recall that
\[
{\rm u}_\e=-\e \lt(\nabla (-\Delta)^{-\alpha}\rho+\nabla \log \rho\rt).
\]
By the expression of $f_1$, we have
\[
\e f_1=\xi\cdot {\rm u}_\e\, M_{\rho,0}.
\]
Hence,
\[
f^\e=M_{\rho,0}+\xi\cdot {\rm u}_\e\, M_{\rho,0}+O(\e^2).
\]

On the other hand, since
\[
M_{\rho,{\rm u}}(\xi)=\rho M_{1,{\rm u}}(\xi),
\]
we apply the Taylor expansion with integral remainder with respect to ${\rm u}$ around ${\rm u}=0$. We first write
\[
M_{\rho,{\rm u}_\e}-M_{\rho,0} = \int_0^1 \frac{d}{d\tau} M_{\rho,\tau {\rm u}_\e}\,d\tau = \rho\int_0^1 M_{1,\tau{\rm u}_\e}(\xi-\tau{\rm u}_\e)\cdot {\rm u}_\e\,d\tau.
\]
Subtracting the first-order term $\xi\cdot {\rm u}_\e\,M_{\rho,0}$, we obtain
\[
 M_{\rho,{\rm u}_\e}-M_{\rho,0}-\xi\cdot {\rm u}_\e\,M_{\rho,0}  = \rho(\xi\cdot {\rm u}_\e)\int_0^1 \bigl(M_{1,\tau{\rm u}_\e}-M_{1,0}\bigr)\,d\tau -\rho |{\rm u}_\e|^2 \int_0^1 \tau M_{1,\tau{\rm u}_\e}\,d\tau.
\]
Moreover, applying the fundamental theorem of calculus to the map
$\nu \mapsto M_{1,\nu\tau{\rm u}_\e}$, we get
\[
M_{1,\tau{\rm u}_\e}-M_{1,0} = \int_0^1 M_{1,\nu\tau{\rm u}_\e}(\xi-\nu\tau{\rm u}_\e)\cdot \tau{\rm u}_\e\,d\nu.
\]
This gives
\[\begin{aligned}
|M_{\rho,{\rm u}_\e}-M_{\rho,0}-\xi\cdot {\rm u}_\e\,M_{\rho,0}| &\le C\rho |\xi|\,|{\rm u}_\e| \int_0^1\int_0^1 M_{1,\nu\tau{\rm u}_\e}\bigl(|\xi|+\nu\tau|{\rm u}_\e|\bigr)\tau|{\rm u}_\e|\,d\nu d\tau   + C\rho |{\rm u}_\e|^2 \int_0^1 M_{1,\tau{\rm u}_\e}\,d\tau \\
&\le C |{\rm u}_\e|^2 (1+|\xi|^2)M_{\rho,0},
\end{aligned}\]
at least formally for $|{\rm u}_\e|$ small, where we used that the shifted Maxwellians
$M_{1,\nu\tau{\rm u}_\e}$ remain uniformly comparable to $M_{1,0}$.  Hence,
\[
M_{\rho,{\rm u}_\e} = M_{\rho,0} +\xi\cdot {\rm u}_\e\, M_{\rho,0} +O\bigl(|{\rm u}_\e|^2(1+|\xi|^2)M_{\rho,0}\bigr).
\]
Since ${\rm u}_\e=O(\e)$, it follows that
\[
M_{\rho,{\rm u}_\e} = M_{\rho,0} +\xi\cdot {\rm u}_\e\, M_{\rho,0} +O\bigl(\e^2(1+|\xi|^2)M_{\rho,0}\bigr),
\]
and subsequently,
\[
f^\e-M_{\rho,{\rm u}_\e} = O\bigl(\e^2(1+|\xi|^2)M_{\rho,0}\bigr).
\]

%
%
%
%
%
%
\subsection{High-field limit}\label{app_optimal_high}

We next consider the high-field scaling
\[
\e \pa_t f^\e + \e \xi \cdot \nabla f^\e - \nabla (-\Delta)^{-\alpha}\rho^\e \cdot \nabla_\xi f^\e = \nabla_\xi \cdot (\nabla_\xi f^\e + \xi f^\e).
\]
We again assume an expansion
\[
f^\e=f_0+\e f_1+\e^2 f_2+\cdots,
\quad
\rho^\e=\rho+\e \rho_1+\e^2 \rho_2+\cdots.
\]
At leading order, the equation becomes
\[
-\nabla (-\Delta)^{-\alpha}\rho \cdot \nabla_\xi f_0 = \nabla_\xi \cdot (\nabla_\xi f_0+\xi f_0).
\]
This relation characterizes the local equilibrium associated with the drift velocity
\[
u=-\nabla (-\Delta)^{-\alpha}\rho,
\]
namely
\[
f_0=M_{\rho,u}=M_{\rho,-\nabla (-\Delta)^{-\alpha}\rho}.
\]

Indeed, one may rewrite the leading-order equation as
\[
\nabla_\xi \cdot \lt(\nabla_\xi f_0 + (\xi-u)f_0\rt)=0,
\]
whose equilibrium is exactly the Maxwellian centered at $u$. Thus, the dominant balance is no longer around the zero-velocity Maxwellian, but around the shifted local Maxwellian determined by the interaction field.

At the next order, we obtain
\[
\pa_t f_0 + \xi \cdot \nabla f_0 -\nabla (-\Delta)^{-\alpha}\rho \cdot \nabla_\xi f_1 -\nabla (-\Delta)^{-\alpha}\rho_1 \cdot \nabla_\xi f_0
=
\nabla_\xi \cdot (\nabla_\xi f_1+\xi f_1).
\]
Equivalently, introducing the linearized Fokker--Planck operator around the shifted equilibrium
\[
\mathcal{L}_u g:=\nabla_\xi \cdot \lt(\nabla_\xi g+(\xi-u)g\rt),
\]
the first-order corrector $f_1$ is formally determined by
\[
\mathcal{L}_u f_1
=
\pa_t f_0 + \xi \cdot \nabla f_0 -\nabla (-\Delta)^{-\alpha}\rho_1 \cdot \nabla_\xi f_0.
\]
Hence, unlike in the diffusive regime, the first-order correction can still be identified at the formal level, but its structure is more involved because the reference equilibrium already carries the nontrivial drift $u=-\nabla (-\Delta)^{-\alpha}\rho$.

In particular, there is no reason to expect that
\[
\intr f_1\,d\xi=\rho_1
\]
vanishes in general. Therefore, unlike in the diffusive regime, the density variable does not benefit from an additional cancellation at first order. For this reason, the formal expansion suggests only
\[
f^\e=M_{\rho,-\nabla (-\Delta)^{-\alpha}\rho}+O(\e),
\]
and, at the macroscopic level,
\[
\rho^\e-\rho=O(\e),
\quad
m^\e + \rho \nabla (-\Delta)^{-\alpha}\rho = O(\e).
\]
Accordingly, one expects first-order convergence for both the kinetic distribution and its macroscopic moments. This is consistent with the quantitative estimates established in Theorem \ref{thm_kin2}.

%
%
%
%
%
%
\subsection{gSQG limit}

We finally consider the strong magnetic field regime leading to the gSQG equation. Since the formal derivation of the limiting equation has already been presented in Appendix \ref{app:gsqg}, we only recall here the main consequence relevant to the present appendix, namely the expected convergence rate suggested by the corresponding Hilbert expansion.

More precisely, as shown in Appendix \ref{app:gsqg}, the kinetic equation
\[
\e \pa_t f^\e + \xi \cdot \nabla f^\e - \nabla (-\Delta)^{-\alpha}\rho^\e \cdot \nabla_\xi f^\e + \frac1\e \xi^\perp \cdot \nabla_\xi f^\e = \nabla_\xi \cdot (\nabla_\xi f^\e+\xi f^\e)
\]
in two dimensions admits a formal expansion of the form
\[
f^\e=f_0+\e f_1+O(\e^2),
\]
where the leading-order term is given by
\[
f_0=M_{\rho,0},
\]
and the first-order corrector $f_1$  is given by
\[
f_1=-\lt((\nabla \rho + \rho \nabla (-\Delta)^{-\alpha}\rho)^\perp \cdot \xi\rt)M_{1,0},
\]
which  encodes the effect of transport and nonlocal interaction after averaging along the fast rotational flow in velocity space.

At the macroscopic level, this first-order correction generates the effective  flux
\[
m^1=- \nabla^\perp \rho-\rho \nabla^\perp (-\Delta)^{-\alpha}\rho,
\]

and the corresponding limit equation is
\[
\partial_t \rho - \nabla \cdot \bigl(\rho \nabla^\perp (-\Delta)^{-\alpha}\rho \bigr)=0.
\]
Equivalently, using $\nabla \cdot \nabla^\perp \phi =0$, one may write
\[
\partial_t \rho + \nabla^\perp (-\Delta)^{-\alpha}\rho \cdot \nabla \rho =0,
\]
which is the generalized surface quasi-geostrophic equation in active scalar form.

Therefore, at the level of the formal expansion, one expects
\[
f^\e-M_{\rho,0}=O(\e).
\]
 Moreover, since the first-order corrector $f_1$ is odd in $\xi$, we have
\[
\int_{\R^2} f_1\,d\xi=0.
\]
Hence the first-order correction does not contribute to the density variable, and formally
\[
\rho^\e-\rho=O(\e^2).
\]
On the other hand, the momentum satisfies
\[
m^\e:=\int_{\R^2} \xi f^\e\,d\xi = \e \int_{\R^2}\xi f_1\,d\xi + O(\e^2),
\]
so that
\[
\frac1\e m^\e
=
-\nabla^\perp \rho-\rho \nabla^\perp (-\Delta)^{-\alpha}\rho + O(\e).
\]
Thus, as in the diffusive regime, the density variable gains one additional order compared to the kinetic distribution itself. This improvement is again due to the oddness of the first-order corrector in the velocity variable. We emphasize, however, that this is only a formal prediction of the asymptotic expansion; our rigorous estimates do not address this sharper $O(\e^2)$ density rate.

Nevertheless, the first-order corrector can still be absorbed into a suitably shifted Maxwellian, which yields a more accurate kinetic approximation. Indeed, since
\[
f_1= -\lt((\nabla \rho+\rho \nabla (-\Delta)^{-\alpha}\rho)^\perp \cdot \xi\rt)M_{1,0} = - \xi\cdot\lt(\nabla^\perp \log\rho + \nabla^\perp (-\Delta)^{-\alpha}\rho\rt)M_{\rho,0},
\]
it is natural to introduce
\[
{\rm u}_\e := - \e\lt(\nabla^\perp \log\rho + \nabla^\perp (-\Delta)^{-\alpha}\rho\rt),
\]
so that
\[
f^\e=M_{\rho,0}+\xi\cdot {\rm u}_\e\,M_{\rho,0}+O(\e^2).
\]
Using the same first-order Taylor expansion of $M_{\rho,{\rm u}_\e}$ around ${\rm u}=0$ as in the diffusive case, we formally obtain
\[
M_{\rho,{\rm u}_\e} = M_{\rho,0} +\xi\cdot {\rm u}_\e\,M_{\rho,0} +O\bigl(\e^2(1+|\xi|^2)M_{\rho,0}\bigr),
\]
and hence
\[
f^\e-M_{\rho,{\rm u}_\e} = O\bigl(\e^2(1+|\xi|^2)M_{\rho,0}\bigr).
\]

Thus, the formal expansion suggests that the density variable satisfies
\[
\rho^\e-\rho=O(\e^2),
\]
while, at the kinetic level, the shifted local Maxwellian $M_{\rho,{\rm u}_\e}$ provides a sharper second-order approximation of $f^\e$.

%
%
%
%
%
%
\section{Propagation of regularity for the limit equations}\label{app_reg}

In this appendix, we verify that the regularity assumptions imposed on the velocity field in the limiting equations \eqref{eq_agdi}, \eqref{eq_ag}, and \eqref{eq_sqg} are consistent, in the sense that they propagate (at least locally in time) under appropriate conditions on the initial data.

 We assume that the initial data satisfy
\[
\rho_0 \in H^m(\R^d), \quad m \gg 1, \quad 
\phi_0 := \nabla \log \rho_0 \in L^p \cap \dot H^s(\R^d),
\]
with $p>d$, $s \gg 1$, and $s \le m-2$. 
As an example, $\rho_0(x) := (1+|x|^2)^{-k}$ with $k \gg 1$ fulfills the above assumptions. 
Local existence of strong solutions $\rho \in L^\infty(0,T;H^m(\R^d))$ is guaranteed by \cite{CJe21-1}. 
We recall that the limiting equations can be written as
\[
\partial_t \rho + \nabla\cdot(\rho u) = \sigma \Delta \rho,
\]
where the precise form of $u$ depends on the equation under consideration, but in all cases it is comparable to $\nabla (-\Delta)^{-\alpha} \rho$ and in particular $u \in H^{m-1}(\R^d)$.

Introducing $\phi = \nabla \log \rho$, we write
\[
\pa_t \phi + \nabla (\phi \cdot u) + \nabla(\nabla\cdot u) = \sigma(\nabla (\nabla\cdot \phi) + \nabla(|\phi|^2)),
\]
where we used
\[
\nabla\cdot\phi = \nabla \cdot \lt(\frac{\nabla\rho}{\rho} \rt) = \frac{\Delta\rho}{\rho} - \frac{|\nabla\rho|^2}{\rho^2} =  \frac{\Delta\rho}{\rho} - |\phi|^2.
\]

\noindent
{\bf $L^p$-estimate.} A direct computation yields
\[\begin{aligned}
\frac1p\frac{d}{dt}\|\phi\|_{L^p}^p &= \intr |\phi|^{p-2}\phi \cdot (-\nabla u \cdot \phi - \nabla\phi \cdot u - \nabla(\nabla\cdot u) + \sigma \nabla(\nabla\cdot \phi) + \sigma \nabla(|\phi|^2))\,dx\\
&\le C\lt( \|\nabla u\|_{L^\infty}\|\phi\|_{L^p}  + \|\nabla \phi\|_{L^\infty}\|u\|_{L^p} + \|\nabla^2 u\|_{L^p}+\|\nabla\phi\|_{L^\infty}\|\phi\|_{L^p}\rt)\|\phi\|_{L^p}^{p-1}\\
&\quad + \sigma \intr |\phi|^{p-2}\phi \cdot \nabla(\nabla\cdot\phi)\,dx\\
&= C\lt( \|\nabla u\|_{L^\infty}\|\phi\|_{L^p}  + \|\nabla \phi\|_{L^\infty}\|u\|_{L^p} + \|\nabla^2 u\|_{L^p}+\|\nabla\phi\|_{L^\infty}\|\phi\|_{L^p}\rt)\|\phi\|_{L^p}^{p-1}\\
&\quad -\intr \nabla\cdot (|\phi|^{p-2}\phi) \nabla\cdot \phi\,dx\\
&\le  C(1+ \|\nabla u\|_{H^s})(1+\|\phi\|_{L^p \cap \dot{H}^s})^2\|\phi\|_{L^p \cap \dot{H}^s}^{p-2}.  
\end{aligned}\]

\noindent
{\bf $\dot H^s$-estimate.} Similarly, applying $\Lambda^s$ and testing against $\phi$ gives
\[\begin{aligned}
\frac12\frac{d}{dt}\|\Lambda^s \phi\|_{L^2}^2 + \sigma \|\Lambda^s(\nabla \cdot \phi)\|_{L^2}^2&= - \intr \Lambda^s \nabla (\phi\cdot u) \cdot \Lambda^s \phi\,dx - \intr \Lambda^s (\nabla (\nabla\cdot u)) \cdot \Lambda^s \phi \,dx \\
&\quad  +\sigma \intr \Lambda^s (\nabla |\phi|^2) \cdot \Lambda^s \phi\,dx\\
&=: I + II + III.
\end{aligned}\]
For $I$, we use the symmetry $\pa_j \phi_i = \pa_i \phi_j$ (since $\phi$ itself is a gradient of a certain function) to estimate
\[\begin{aligned}
I &= -\sum_{i,j=1}^d\intr \Lambda^s \pa_j \phi_i u_i \Lambda^s \phi_j\,dx  - \intr [\Lambda^s \nabla (\phi \cdot u) - (\Lambda^s \nabla \phi) \cdot u]\cdot \Lambda^s \phi\,dx\\
&=-\sum_{i,j=1}^d\intr \Lambda^s \pa_i \phi_j u_i \Lambda^s \phi_j\,dx  - \intr [\Lambda^s \nabla (\phi \cdot u) - (\Lambda^s \nabla \phi) \cdot u]\cdot \Lambda^s \phi\,dx\\
&\le \frac{\|\nabla \cdot u\|_{L^\infty}}{2}\|\phi\|_{\dot{H}^s}^2 + C(\|\nabla u\|_{L^\infty}\|\phi\|_{\dot{H}^s} + \|\nabla \phi\|_{L^\infty}\|u\|_{\dot{H}^s})\|\phi\|_{\dot{H}^s}.
\end{aligned}\]
The estimate of $II$ is obvious, and for $III$, we get
\[\begin{aligned}
III &= \sigma \intr ((\Lambda^s \nabla \phi)\cdot \phi)\cdot \Lambda^s \phi\,dx + \sigma \intr [\Lambda^s (\nabla |\phi|^2) - (\Lambda^s \nabla \phi)\cdot \phi]\cdot \Lambda^s\phi\,dx\\
&\le \sigma \|\phi\|_{L^\infty}\|\nabla \phi\|_{\dot{H}^s}\|\phi\|_{\dot{H}^s}+C\|\nabla\phi\|_{L^\infty}\|\phi\|_{\dot{H}^s}^2\\
&\le C\|\phi\|_{L^\infty}^2\|\phi\|_{\dot{H}^s}^2 + C\|\nabla\phi\|_{L^\infty}\|\phi\|_{\dot{H}^s}^2 + \frac\sigma2 \|\nabla \phi\|_{\dot{H}^s}^2.
\end{aligned}\]
Moreover, thanks to the symmetry of $\nabla \phi$, we have $\|\nabla \cdot \phi\|_{\dot{H}^s} = \|\nabla \phi\|_{\dot{H}^s}$. Thus, we gather all the estimates to obtain
\[
\frac{d}{dt}\|\phi\|_{\dot{H}^s}^2 \le C(1+\|\nabla u\|_{H^s})(1+\|\phi\|_{L^p \cap \dot{H}^s})^2 \|\phi\|_{\dot{H}^s}^2.
\]

  Collecting the above bounds, we arrive at the closed differential inequality
\[
\frac{d}{dt}\|\phi\|_{L^p \cap \dot H^s}^2 
\le C(1+\|\nabla u\|_{H^s})(1+\|\phi\|_{L^p \cap \dot H^s})^4.
\]
This shows that $\phi=\nabla\log\rho$ remains bounded in $L^p \cap \dot H^s(\R^d)$ for a short time, and hence the assumed regularity of the velocity field in our main results indeed propagates locally in time. 
Together with the local existence result of \cite{CJe21-1}, this confirms the consistency of the a priori regularity framework.

%
%
%
%

\bibliographystyle{abbrv}
\bibliography{CJ_rel_VFP}

@book {AGS08,
    AUTHOR = {Ambrosio, Luigi and Gigli, Nicola and Savar\'e, Giuseppe},
     TITLE = {Gradient flows in metric spaces and in the space of
              probability measures},
    SERIES = {Lectures in Mathematics ETH Z\"urich},
   EDITION = {Second},
 PUBLISHER = {Birkh\"auser Verlag, Basel},
      YEAR = {2008},
     PAGES = {x+334},
      ISBN = {978-3-7643-8721-1},
   MRCLASS = {49-02 (28A33 35K55 35K90 49Q20 60B05)},
  MRNUMBER = {2401600},
MRREVIEWER = {Pietro\ Celada},
}

@incollection {ACGS01,
    AUTHOR = {Arnold, A. and Carrillo, J. A. and Gamba, I. and Shu, C.-W.},
     TITLE = {Low and high field scaling limits for the {V}lasov- and
              {W}igner-{P}oisson-{F}okker-{P}lanck systems},
      NOTE = {The Sixteenth International Conference on Transport Theory,
              Part I (Atlanta, GA, 1999)},
   JOURNAL = {Transport Theory Statist. Phys.},
  FJOURNAL = {Transport Theory and Statistical Physics},
    VOLUME = {30},
      YEAR = {2001},
    NUMBER = {2-3},
     PAGES = {121--153},
      ISSN = {0041-1450,1532-2424},
   MRCLASS = {82D37 (82C40)},
  MRNUMBER = {1848592},
MRREVIEWER = {Vittorio\ Romano},
       DOI = {10.1081/TT-100105365},
       URL = {https://doi.org/10.1081/TT-100105365},
}

@article {AMTU00,
    AUTHOR = {Unterreiter, Andreas and Arnold, Anton and Markowich, Peter
              and Toscani, Giuseppe},
     TITLE = {On generalized {C}sisz\'ar-{K}ullback inequalities},
   JOURNAL = {Monatsh. Math.},
  FJOURNAL = {Monatshefte f\"ur Mathematik},
    VOLUME = {131},
      YEAR = {2000},
    NUMBER = {3},
     PAGES = {235--253},
      ISSN = {0026-9255,1436-5081},
   MRCLASS = {94A17 (28A33 28D20 46N30 52A40)},
  MRNUMBER = {1801751},
MRREVIEWER = {N.\ N.\ Ganikhodjaev},
       DOI = {10.1007/s006050070013},
       URL = {https://doi.org/10.1007/s006050070013},
}

@article {BGL91,
    AUTHOR = {Bardos, Claude and Golse, Fran\c cois and Levermore, David},
     TITLE = {Fluid dynamic limits of kinetic equations. {I}. {F}ormal
              derivations},
   JOURNAL = {J. Statist. Phys.},
  FJOURNAL = {Journal of Statistical Physics},
    VOLUME = {63},
      YEAR = {1991},
    NUMBER = {1-2},
     PAGES = {323--344},
      ISSN = {0022-4715,1572-9613},
   MRCLASS = {82C40 (82C31)},
  MRNUMBER = {1115587},
MRREVIEWER = {Andrzej\ Fuli\'nski},
       DOI = {10.1007/BF01026608},
       URL = {https://doi.org/10.1007/BF01026608},
}

@article {BGL93,
    AUTHOR = {Bardos, Claude and Golse, Fran\c cois and Levermore, C. David},
     TITLE = {Fluid dynamic limits of kinetic equations. {II}. {C}onvergence
              proofs for the {B}oltzmann equation},
   JOURNAL = {Comm. Pure Appl. Math.},
  FJOURNAL = {Communications on Pure and Applied Mathematics},
    VOLUME = {46},
      YEAR = {1993},
    NUMBER = {5},
     PAGES = {667--753},
      ISSN = {0010-3640,1097-0312},
   MRCLASS = {82C40 (76A02 76D05 76P05)},
  MRNUMBER = {1213991},
MRREVIEWER = {Andrzej\ Fuli\'nski},
       DOI = {10.1002/cpa.3160460503},
       URL = {https://doi.org/10.1002/cpa.3160460503},
}

@article {BV05,
    AUTHOR = {Berthelin, F. and Vasseur, A.},
     TITLE = {From kinetic equations to multidimensional isentropic gas
              dynamics before shocks},
   JOURNAL = {SIAM J. Math. Anal.},
  FJOURNAL = {SIAM Journal on Mathematical Analysis},
    VOLUME = {36},
      YEAR = {2005},
    NUMBER = {6},
     PAGES = {1807--1835},
      ISSN = {0036-1410,1095-7154},
   MRCLASS = {35L65 (35F20 76A02 76N15 82C40)},
  MRNUMBER = {2178222},
MRREVIEWER = {Kenneth\ H.\ Karlsen},
       DOI = {10.1137/S0036141003431554},
       URL = {https://doi.org/10.1137/S0036141003431554},
}

@article {BIR25,
    AUTHOR = {Ben-Porat, Immanuel and Iacobelli, Mikaela and Rege,
              Alexandre},
     TITLE = {Derivation of {Y}udovich solutions of incompressible {E}uler
              from the {V}lasov-{P}oisson system},
   JOURNAL = {SIAM J. Math. Anal.},
  FJOURNAL = {SIAM Journal on Mathematical Analysis},
    VOLUME = {57},
      YEAR = {2025},
    NUMBER = {1},
     PAGES = {886--906},
      ISSN = {0036-1410,1095-7154},
   MRCLASS = {35Q35 (35Q31 35Q83)},
  MRNUMBER = {4858712},
MRREVIEWER = {Juan\ Carlos\ Mu\~noz Grajales},
       DOI = {10.1137/24M1648417},
       URL = {https://doi.org/10.1137/24M1648417},
}

@article {Bos07,
    AUTHOR = {Bostan, Mihai},
     TITLE = {The {V}lasov-{M}axwell system with strong initial magnetic
              field: guiding-center approximation},
   JOURNAL = {Multiscale Model. Simul.},
  FJOURNAL = {Multiscale Modeling \& Simulation. A SIAM Interdisciplinary
              Journal},
    VOLUME = {6},
      YEAR = {2007},
    NUMBER = {3},
     PAGES = {1026--1058},
      ISSN = {1540-3459,1540-3467},
   MRCLASS = {35Q60 (35B40 76X05 78A35 82D10)},
  MRNUMBER = {2368978},
       DOI = {10.1137/070689383},
       URL = {https://doi.org/10.1137/070689383},
}

@article {BG08,
    AUTHOR = {Bostan, Mihai and Goudon, Thierry},
     TITLE = {Low field regime for the relativistic
              {V}lasov-{M}axwell-{F}okker-{P}lanck system; the one and one
              half dimensional case},
   JOURNAL = {Kinet. Relat. Models},
  FJOURNAL = {Kinetic and Related Models},
    VOLUME = {1},
      YEAR = {2008},
    NUMBER = {1},
     PAGES = {139--170},
      ISSN = {1937-5093,1937-5077},
   MRCLASS = {35Q75 (35B40 35K55 82D10)},
  MRNUMBER = {2383720},
       DOI = {10.3934/krm.2008.1.139},
       URL = {https://doi.org/10.3934/krm.2008.1.139},
}

@article {BV25,
    AUTHOR = {Bostan, Miha\"i{} and Vu, Anh-Tuan},
     TITLE = {Asymptotic behavior of the two-dimensional
              {V}lasov-{P}oisson-{F}okker-{P}lanck equation with a strong
              external magnetic field},
   JOURNAL = {Kinet. Relat. Models},
  FJOURNAL = {Kinetic and Related Models},
    VOLUME = {18},
      YEAR = {2025},
    NUMBER = {1},
     PAGES = {101--147},
      ISSN = {1937-5093,1937-5077},
   MRCLASS = {35Q60 (35Q31 35Q83 35Q84 78A35 82D10)},
  MRNUMBER = {4830254},
       DOI = {10.3934/krm.2024013},
       URL = {https://doi.org/10.3934/krm.2024013},
}

@article {Bou93,
    AUTHOR = {Bouchut, Fran\c cois},
     TITLE = {Existence and uniqueness of a global smooth solution for the
              {V}lasov-{P}oisson-{F}okker-{P}lanck system in three
              dimensions},
   JOURNAL = {J. Funct. Anal.},
  FJOURNAL = {Journal of Functional Analysis},
    VOLUME = {111},
      YEAR = {1993},
    NUMBER = {1},
     PAGES = {239--258},
      ISSN = {0022-1236,1096-0783},
   MRCLASS = {82D10 (35Q99 45K05 82C31)},
  MRNUMBER = {1200643},
MRREVIEWER = {Klaus\ Dressler},
       DOI = {10.1006/jfan.1993.1011},
       URL = {https://doi.org/10.1006/jfan.1993.1011},
}

@article {Bre00,
    AUTHOR = {Brenier, Y.},
     TITLE = {Convergence of the {V}lasov-{P}oisson system to the
              incompressible {E}uler equations},
   JOURNAL = {Comm. Partial Differential Equations},
  FJOURNAL = {Communications in Partial Differential Equations},
    VOLUME = {25},
      YEAR = {2000},
    NUMBER = {3-4},
     PAGES = {737--754},
      ISSN = {0360-5302,1532-4133},
   MRCLASS = {76X05 (35Q05 76N99 82C21 82D10)},
  MRNUMBER = {1748352},
MRREVIEWER = {R.\ Glassey},
       DOI = {10.1080/03605300008821529},
       URL = {https://doi.org/10.1080/03605300008821529},
}

@article {CS07,
    AUTHOR = {Caffarelli, Luis and Silvestre, Luis},
     TITLE = {An extension problem related to the fractional {L}aplacian},
   JOURNAL = {Comm. Partial Differential Equations},
  FJOURNAL = {Communications in Partial Differential Equations},
    VOLUME = {32},
      YEAR = {2007},
    NUMBER = {7-9},
     PAGES = {1245--1260},
      ISSN = {0360-5302,1532-4133},
   MRCLASS = {35J70},
  MRNUMBER = {2354493},
MRREVIEWER = {Francesco\ Petitta},
       DOI = {10.1080/03605300600987306},
       URL = {https://doi.org/10.1080/03605300600987306},
}

@article {Caf80,
    AUTHOR = {Caflisch, Russel E.},
     TITLE = {The fluid dynamic limit of the nonlinear {B}oltzmann equation},
   JOURNAL = {Comm. Pure Appl. Math.},
  FJOURNAL = {Communications on Pure and Applied Mathematics},
    VOLUME = {33},
      YEAR = {1980},
    NUMBER = {5},
     PAGES = {651--666},
      ISSN = {0010-3640,1097-0312},
   MRCLASS = {76P05 (35L65 35L67 82A40)},
  MRNUMBER = {586416},
MRREVIEWER = {Tai\ Ping\ Liu},
       DOI = {10.1002/cpa.3160330506},
       URL = {https://doi.org/10.1002/cpa.3160330506},
}

@article {CC20,
    AUTHOR = {Carrillo, Jos\'e{} A. and Choi, Young-Pil},
     TITLE = {Quantitative error estimates for the large friction limit of
              {V}lasov equation with nonlocal forces},
   JOURNAL = {Ann. Inst. H. Poincar\'e{} C Anal. Non Lin\'eaire},
  FJOURNAL = {Annales de l'Institut Henri Poincar\'e{} C. Analyse Non
              Lin\'eaire},
    VOLUME = {37},
      YEAR = {2020},
    NUMBER = {4},
     PAGES = {925--954},
      ISSN = {0294-1449,1873-1430},
   MRCLASS = {35Q83 (35B25 35Q35 35Q92)},
  MRNUMBER = {4104830},
       DOI = {10.1016/j.anihpc.2020.02.001},
       URL = {https://doi.org/10.1016/j.anihpc.2020.02.001},
}

@article {CC21,
    AUTHOR = {Carrillo, Jos\'e{} A. and Choi, Young-Pil},
     TITLE = {Mean-field limits: from particle descriptions to macroscopic
              equations},
   JOURNAL = {Arch. Ration. Mech. Anal.},
  FJOURNAL = {Archive for Rational Mechanics and Analysis},
    VOLUME = {241},
      YEAR = {2021},
    NUMBER = {3},
     PAGES = {1529--1573},
      ISSN = {0003-9527,1432-0673},
   MRCLASS = {82C22 (35Q70)},
  MRNUMBER = {4284530},
MRREVIEWER = {Anamaria\ Savu},
       DOI = {10.1007/s00205-021-01676-x},
       URL = {https://doi.org/10.1007/s00205-021-01676-x},
}

@article {CCJ21,
    AUTHOR = {Carrillo, Jos\'e{} A. and Choi, Young-Pil and Jung, Jinwook},
     TITLE = {Quantifying the hydrodynamic limit of {V}lasov-type equations
              with alignment and nonlocal forces},
   JOURNAL = {Math. Models Methods Appl. Sci.},
  FJOURNAL = {Mathematical Models and Methods in Applied Sciences},
    VOLUME = {31},
      YEAR = {2021},
    NUMBER = {2},
     PAGES = {327--408},
      ISSN = {0218-2025,1793-6314},
   MRCLASS = {35B40 (82C40)},
  MRNUMBER = {4227888},
MRREVIEWER = {Crist\'obal\ Qui\~ninao},
       DOI = {10.1142/S0218202521500081},
       URL = {https://doi.org/10.1142/S0218202521500081},
}

@article {CCP22,
    AUTHOR = {Carrillo, Jos\'e{} A. and Choi, Young-Pil and Peng, Yingping},
     TITLE = {Large friction-high force fields limit for the nonlinear
              {V}lasov-{P}oisson-{F}okker-{P}lanck system},
   JOURNAL = {Kinet. Relat. Models},
  FJOURNAL = {Kinetic and Related Models},
    VOLUME = {15},
      YEAR = {2022},
    NUMBER = {3},
     PAGES = {355--384},
      ISSN = {1937-5093,1937-5077},
   MRCLASS = {35Q70 (35Q83)},
  MRNUMBER = {4414611},
       DOI = {10.3934/krm.2021052},
       URL = {https://doi.org/10.3934/krm.2021052},
}

@article {CCT19,
    AUTHOR = {Carrillo, Jos\'e{} A. and Choi, Young-Pil and Tse, Oliver},
     TITLE = {Convergence to equilibrium in {W}asserstein distance for
              damped {E}uler equations with interaction forces},
   JOURNAL = {Comm. Math. Phys.},
  FJOURNAL = {Communications in Mathematical Physics},
    VOLUME = {365},
      YEAR = {2019},
    NUMBER = {1},
     PAGES = {329--361},
      ISSN = {0010-3616,1432-0916},
   MRCLASS = {35Q31},
  MRNUMBER = {3900833},
MRREVIEWER = {Xinyu\ He},
       DOI = {10.1007/s00220-018-3276-8},
       URL = {https://doi.org/10.1007/s00220-018-3276-8},
}

@article {CS95,
    AUTHOR = {Carrillo, Jos\'e{} A. and Soler, Juan},
     TITLE = {On the initial value problem for the
              {V}lasov-{P}oisson-{F}okker-{P}lanck system with initial data
              in {$L^p$} spaces},
   JOURNAL = {Math. Methods Appl. Sci.},
  FJOURNAL = {Mathematical Methods in the Applied Sciences},
    VOLUME = {18},
      YEAR = {1995},
    NUMBER = {10},
     PAGES = {825--839},
      ISSN = {0170-4214,1099-1476},
   MRCLASS = {76X05 (35Q99 82D10 85A15)},
  MRNUMBER = {1343393},
MRREVIEWER = {Gerhard\ Rein},
       DOI = {10.1002/mma.1670181006},
       URL = {https://doi.org/10.1002/mma.1670181006},
}

@book {CIP94,
    AUTHOR = {Cercignani, Carlo and Illner, Reinhard and Pulvirenti, Mario},
     TITLE = {The mathematical theory of dilute gases},
    SERIES = {Applied Mathematical Sciences},
    VOLUME = {106},
 PUBLISHER = {Springer-Verlag, New York},
      YEAR = {1994},
     PAGES = {viii+347},
      ISBN = {0-387-94294-7},
   MRCLASS = {82C40 (76-02 76P05 82-02 82B40)},
  MRNUMBER = {1307620},
MRREVIEWER = {Giuseppe\ Toscani},
       DOI = {10.1007/978-1-4419-8524-8},
       URL = {https://doi.org/10.1007/978-1-4419-8524-8},
}

@article {C21,
    AUTHOR = {Choi, Young-Pil},
     TITLE = {Large friction limit of pressureless {E}uler equations with
              nonlocal forces},
   JOURNAL = {J. Differential Equations},
  FJOURNAL = {Journal of Differential Equations},
    VOLUME = {299},
      YEAR = {2021},
     PAGES = {196--228},
      ISSN = {0022-0396,1090-2732},
   MRCLASS = {35Q31},
  MRNUMBER = {4293722},
MRREVIEWER = {Yongkai\ Liao},
       DOI = {10.1016/j.jde.2021.07.024},
       URL = {https://doi.org/10.1016/j.jde.2021.07.024},
}

@article {CFI25,
    AUTHOR = {Choi, Young-Pil and Fagioli, Simone and Iorio, Valeria},
     TITLE = {Small inertia limit for coupled kinetic swarming models},
   JOURNAL = {J. Nonlinear Sci.},
  FJOURNAL = {Journal of Nonlinear Science},
    VOLUME = {35},
      YEAR = {2025},
    NUMBER = {2},
     PAGES = {Paper No. 39, 52},
      ISSN = {0938-8974,1432-1467},
   MRCLASS = {35A01 (35A21 35D30 35Q70 35R09 82C40)},
  MRNUMBER = {4861855},
MRREVIEWER = {Jin\ Woo\ Jang},
       DOI = {10.1007/s00332-025-10134-x},
       URL = {https://doi.org/10.1007/s00332-025-10134-x},
}

@article {CJe21-1,
    AUTHOR = {Choi, Young-Pil and Jeong, In-Jee},
     TITLE = {Classical solutions for fractional porous medium flow},
   JOURNAL = {Nonlinear Anal.},
  FJOURNAL = {Nonlinear Analysis. Theory, Methods \& Applications. An
              International Multidisciplinary Journal},
    VOLUME = {210},
      YEAR = {2021},
     PAGES = {Paper No. 112393, 13},
      ISSN = {0362-546X,1873-5215},
   MRCLASS = {35Q35},
  MRNUMBER = {4253949},
       DOI = {10.1016/j.na.2021.112393},
       URL = {https://doi.org/10.1016/j.na.2021.112393},
}

@article {CJe21-2,
    AUTHOR = {Choi, Young-Pil and Jeong, In-Jee},
     TITLE = {Relaxation to fractional porous medium equation from
              {E}uler-{R}iesz system},
   JOURNAL = {J. Nonlinear Sci.},
  FJOURNAL = {Journal of Nonlinear Science},
    VOLUME = {31},
      YEAR = {2021},
    NUMBER = {6},
     PAGES = {Paper No. 95, 28},
      ISSN = {0938-8974,1432-1467},
   MRCLASS = {76N10 (35Q31)},
  MRNUMBER = {4319923},
       DOI = {10.1007/s00332-021-09754-w},
       URL = {https://doi.org/10.1007/s00332-021-09754-w},
}

@article {CJe23,
    AUTHOR = {Choi, Young-Pil and Jeong, In-Jee},
     TITLE = {Global-in-time existence of weak solutions for
              {V}lasov-{M}anev-{F}okker-{P}lanck system},
   JOURNAL = {Kinet. Relat. Models},
  FJOURNAL = {Kinetic and Related Models},
    VOLUME = {16},
      YEAR = {2023},
    NUMBER = {1},
     PAGES = {41--53},
      ISSN = {1937-5093,1937-5077},
   MRCLASS = {82C40 (35Q70)},
  MRNUMBER = {4509375},
       DOI = {10.3934/krm.2022021},
       URL = {https://doi.org/10.3934/krm.2022021},
}

@article {CJK24,
    AUTHOR = {Choi, Young-Pil and Jeong, In-Jee and Kang, Kyungkeun},
     TITLE = {Global {C}auchy problem for the
              {V}lasov-{R}iesz-{F}okker-{P}lanck system near the global
              {M}axwellian},
   JOURNAL = {J. Evol. Equ.},
  FJOURNAL = {Journal of Evolution Equations},
    VOLUME = {24},
      YEAR = {2024},
    NUMBER = {3},
     PAGES = {Paper No. 67, 25},
      ISSN = {1424-3199,1424-3202},
   MRCLASS = {76B47 (35Q35 35Q83 35Q84)},
  MRNUMBER = {4780024},
       DOI = {10.1007/s00028-024-00995-2},
       URL = {https://doi.org/10.1007/s00028-024-00995-2},
}

@article {CJ24,
    AUTHOR = {Choi, Young-Pil and Jung, Jinwook},
     TITLE = {Modulated energy estimates for singular kernels and their
              applications to asymptotic analyses for kinetic equations},
   JOURNAL = {SIAM J. Math. Anal.},
  FJOURNAL = {SIAM Journal on Mathematical Analysis},
    VOLUME = {56},
      YEAR = {2024},
    NUMBER = {2},
     PAGES = {1525--1559},
      ISSN = {0036-1410,1095-7154},
   MRCLASS = {35B35 (35Q35)},
  MRNUMBER = {4711951},
       DOI = {10.1137/22M1537643},
       URL = {https://doi.org/10.1137/22M1537643},
}

@article {CT22,
    AUTHOR = {Choi, Young-Pil and Tse, Oliver},
     TITLE = {Quantified overdamped limit for kinetic
              {V}lasov-{F}okker-{P}lanck equations with singular interaction
              forces},
   JOURNAL = {J. Differential Equations},
  FJOURNAL = {Journal of Differential Equations},
    VOLUME = {330},
      YEAR = {2022},
     PAGES = {150--207},
      ISSN = {0022-0396,1090-2732},
   MRCLASS = {35K57 (60H30 82C31)},
  MRNUMBER = {4428799},
       DOI = {10.1016/j.jde.2022.05.008},
       URL = {https://doi.org/10.1016/j.jde.2022.05.008},
}

@article {Deg86,
    AUTHOR = {Degond, Pierre},
     TITLE = {Global existence of smooth solutions for the
              {V}lasov-{F}okker-{P}lanck equation in {$1$} and {$2$} space
              dimensions},
   JOURNAL = {Ann. Sci. \'Ecole Norm. Sup. (4)},
  FJOURNAL = {Annales Scientifiques de l'\'Ecole Normale Sup\'erieure.
              Quatri\`eme S\'erie},
    VOLUME = {19},
      YEAR = {1986},
    NUMBER = {4},
     PAGES = {519--542},
      ISSN = {0012-9593},
   MRCLASS = {35Q20 (76W05)},
  MRNUMBER = {875086},
MRREVIEWER = {P.\ L.\ Sulem},
       URL = {http://www.numdam.org/item?id=ASENS_1986_4_19_4_519_0},
}

@book {DZ98,
    AUTHOR = {Dembo, Amir and Zeitouni, Ofer},
     TITLE = {Large deviations techniques and applications},
    SERIES = {Applications of Mathematics (New York)},
    VOLUME = {38},
   EDITION = {Second},
 PUBLISHER = {Springer-Verlag, New York},
      YEAR = {1998},
     PAGES = {xvi+396},
      ISBN = {0-387-98406-2},
   MRCLASS = {60F10},
  MRNUMBER = {1619036},
       DOI = {10.1007/978-1-4612-5320-4},
       URL = {https://doi.org/10.1007/978-1-4612-5320-4},
}

@article {DV75,
    AUTHOR = {Donsker, M. D. and Varadhan, S. R. S.},
     TITLE = {Asymptotic evaluation of certain {M}arkov process expectations
              for large time. {I}. {II}},
   JOURNAL = {Comm. Pure Appl. Math.},
  FJOURNAL = {Communications on Pure and Applied Mathematics},
    VOLUME = {28},
      YEAR = {1975},
     PAGES = {1--47; ibid. 28, 279--301},
      ISSN = {0010-3640,1097-0312},
   MRCLASS = {60J25 (60J65)},
  MRNUMBER = {386024},
MRREVIEWER = {D.\ W.\ Stroock},
       DOI = {10.1002/cpa.3160280102},
       URL = {https://doi.org/10.1002/cpa.3160280102},
}

@article {DLPSS18,
    AUTHOR = {Duong, M. H. and Lamacz, A. and Peletier, M. A. and
              Schlichting, A. and Sharma, U.},
     TITLE = {Quantification of coarse-graining error in {L}angevin and
              overdamped {L}angevin dynamics},
   JOURNAL = {Nonlinearity},
  FJOURNAL = {Nonlinearity},
    VOLUME = {31},
      YEAR = {2018},
    NUMBER = {10},
     PAGES = {4517--4566},
      ISSN = {0951-7715,1361-6544},
   MRCLASS = {60J60 (35Q20 35Q83 35Q84 60F10 60H10 82C31)},
  MRNUMBER = {3846438},
       DOI = {10.1088/1361-6544/aaced5},
       URL = {https://doi.org/10.1088/1361-6544/aaced5},
}

@article {DLPS17,
    AUTHOR = {Duong, Manh Hong and Lamacz, Agnes and Peletier, Mark A. and
              Sharma, Upanshu},
     TITLE = {Variational approach to coarse-graining of generalized
              gradient flows},
   JOURNAL = {Calc. Var. Partial Differential Equations},
  FJOURNAL = {Calculus of Variations and Partial Differential Equations},
    VOLUME = {56},
      YEAR = {2017},
    NUMBER = {4},
     PAGES = {Paper No. 100, 65},
      ISSN = {0944-2669,1432-0835},
   MRCLASS = {35K67 (35B25 35K10 35K20 49J45 49S99 60F10 70F40)},
  MRNUMBER = {3666792},
MRREVIEWER = {Joseph\ L.\ Shomberg},
       DOI = {10.1007/s00526-017-1186-9},
       URL = {https://doi.org/10.1007/s00526-017-1186-9},
}

@book {DE97,
    AUTHOR = {Dupuis, Paul and Ellis, Richard S.},
     TITLE = {A weak convergence approach to the theory of large deviations},
    SERIES = {Wiley Series in Probability and Statistics: Probability and
              Statistics},
      NOTE = {A Wiley-Interscience Publication},
 PUBLISHER = {John Wiley \& Sons, Inc., New York},
      YEAR = {1997},
     PAGES = {xviii+479},
      ISBN = {0-471-07672-4},
   MRCLASS = {60F10 (60B10 60F17)},
  MRNUMBER = {1431744},
MRREVIEWER = {Anatolii\ A.\ Pukhal\cprime ski\u i},
       DOI = {10.1002/9781118165904},
       URL = {https://doi.org/10.1002/9781118165904},
}

@article {FS15,
    AUTHOR = {Fetecau, R. C. and Sun, W.},
     TITLE = {First-order aggregation models and zero inertia limits},
   JOURNAL = {J. Differential Equations},
  FJOURNAL = {Journal of Differential Equations},
    VOLUME = {259},
      YEAR = {2015},
    NUMBER = {11},
     PAGES = {6774--6802},
      ISSN = {0022-0396,1090-2732},
   MRCLASS = {35F25 (35B25 35B30 35R06 45K05)},
  MRNUMBER = {3397338},
MRREVIEWER = {Jana\ Kopfova},
       DOI = {10.1016/j.jde.2015.08.018},
       URL = {https://doi.org/10.1016/j.jde.2015.08.018},
}

@article {FK19,
    AUTHOR = {Figalli, Alessio and Kang, Moon-Jin},
     TITLE = {A rigorous derivation from the kinetic {C}ucker-{S}male model
              to the pressureless {E}uler system with nonlocal alignment},
   JOURNAL = {Anal. PDE},
  FJOURNAL = {Analysis \& PDE},
    VOLUME = {12},
      YEAR = {2019},
    NUMBER = {3},
     PAGES = {843--866},
      ISSN = {2157-5045,1948-206X},
   MRCLASS = {35Q31 (35B25 35Q83 35Q84)},
  MRNUMBER = {3864212},
MRREVIEWER = {G.-M.\ Gie},
       DOI = {10.2140/apde.2019.12.843},
       URL = {https://doi.org/10.2140/apde.2019.12.843},
}

@article {Fre04,
    AUTHOR = {Freidlin, Mark},
     TITLE = {Some remarks on the {S}moluchowski-{K}ramers approximation},
   JOURNAL = {J. Statist. Phys.},
  FJOURNAL = {Journal of Statistical Physics},
    VOLUME = {117},
      YEAR = {2004},
    NUMBER = {3-4},
     PAGES = {617--634},
      ISSN = {0022-4715,1572-9613},
   MRCLASS = {82C31 (60H10)},
  MRNUMBER = {2099730},
MRREVIEWER = {Alexander\ N.\ Pechen\cprime},
       DOI = {10.1007/s10955-004-2273-9},
       URL = {https://doi.org/10.1007/s10955-004-2273-9},
}

@article {GM10,
    AUTHOR = {El Ghani, Najoua and Masmoudi, Nader},
     TITLE = {Diffusion limit of the {V}lasov-{P}oisson-{F}okker-{P}lanck
              system},
   JOURNAL = {Commun. Math. Sci.},
  FJOURNAL = {Communications in Mathematical Sciences},
    VOLUME = {8},
      YEAR = {2010},
    NUMBER = {2},
     PAGES = {463--479},
      ISSN = {1539-6746,1945-0796},
   MRCLASS = {82D10 (35Q82 45K05)},
  MRNUMBER = {2664460},
       DOI = {10.4310/cms.2010.v8.n2.a9},
       URL = {https://doi.org/10.4310/cms.2010.v8.n2.a9},
}

@article {GSR99,
    AUTHOR = {Golse, Fran\c cois and Saint-Raymond, Laure},
     TITLE = {The {V}lasov-{P}oisson system with strong magnetic field},
   JOURNAL = {J. Math. Pures Appl. (9)},
  FJOURNAL = {Journal de Math\'ematiques Pures et Appliqu\'ees. Neuvi\`eme
              S\'erie},
    VOLUME = {78},
      YEAR = {1999},
    NUMBER = {8},
     PAGES = {791--817},
      ISSN = {0021-7824},
   MRCLASS = {35Q99 (76X05 82C21)},
  MRNUMBER = {1715342},
MRREVIEWER = {R.\ Glassey},
       DOI = {10.1016/S0021-7824(99)00021-5},
       URL = {https://doi.org/10.1016/S0021-7824(99)00021-5},
}

@article {Gou05,
    AUTHOR = {Goudon, Thierry},
     TITLE = {Hydrodynamic limit for the
              {V}lasov-{P}oisson-{F}okker-{P}lanck system: analysis of the
              two-dimensional case},
   JOURNAL = {Math. Models Methods Appl. Sci.},
  FJOURNAL = {Mathematical Models and Methods in Applied Sciences},
    VOLUME = {15},
      YEAR = {2005},
    NUMBER = {5},
     PAGES = {737--752},
      ISSN = {0218-2025,1793-6314},
   MRCLASS = {82C31 (45K05 82D10)},
  MRNUMBER = {2139941},
MRREVIEWER = {Guillaume\ Bal},
       DOI = {10.1142/S021820250500056X},
       URL = {https://doi.org/10.1142/S021820250500056X},
}

@article {GNPS05,
    AUTHOR = {Goudon, T. and Nieto, J. and Poupaud, F. and Soler, J.},
     TITLE = {Multidimensional high-field limit of the electrostatic
              {V}lasov-{P}oisson-{F}okker-{P}lanck system},
   JOURNAL = {J. Differential Equations},
  FJOURNAL = {Journal of Differential Equations},
    VOLUME = {213},
      YEAR = {2005},
    NUMBER = {2},
     PAGES = {418--442},
      ISSN = {0022-0396,1090-2732},
   MRCLASS = {35Q99 (35B25 35F25 82D10)},
  MRNUMBER = {2142374},
MRREVIEWER = {Hidetoshi\ Tahara},
       DOI = {10.1016/j.jde.2004.09.008},
       URL = {https://doi.org/10.1016/j.jde.2004.09.008},
}

@article {G75,
    AUTHOR = {Gross, Leonard},
     TITLE = {Logarithmic {S}obolev inequalities},
   JOURNAL = {Amer. J. Math.},
  FJOURNAL = {American Journal of Mathematics},
    VOLUME = {97},
      YEAR = {1975},
    NUMBER = {4},
     PAGES = {1061--1083},
      ISSN = {0002-9327,1080-6377},
   MRCLASS = {46E35 (81.47)},
  MRNUMBER = {420249},
MRREVIEWER = {R.\ H\o egh-Krohn},
       DOI = {10.2307/2373688},
       URL = {https://doi.org/10.2307/2373688},
}

@article {Han11,
    AUTHOR = {Han-Kwan, Daniel},
     TITLE = {Quasineutral limit of the {V}lasov-{P}oisson system with
              massless electrons},
   JOURNAL = {Comm. Partial Differential Equations},
  FJOURNAL = {Communications in Partial Differential Equations},
    VOLUME = {36},
      YEAR = {2011},
    NUMBER = {8},
     PAGES = {1385--1425},
      ISSN = {0360-5302,1532-4133},
   MRCLASS = {35Q83 (76X05 82D10)},
  MRNUMBER = {2825596},
       DOI = {10.1080/03605302.2011.555804},
       URL = {https://doi.org/10.1080/03605302.2011.555804},
}

@article {Jab00,
    AUTHOR = {Jabin, Pierre-Emmanuel},
     TITLE = {Macroscopic limit of {V}lasov type equations with friction},
   JOURNAL = {Ann. Inst. H. Poincar\'e{} C Anal. Non Lin\'eaire},
  FJOURNAL = {Annales de l'Institut Henri Poincar\'e{} C. Analyse Non
              Lin\'eaire},
    VOLUME = {17},
      YEAR = {2000},
    NUMBER = {5},
     PAGES = {651--672},
      ISSN = {0294-1449,1873-1430},
   MRCLASS = {82C40 (35B25 45K05)},
  MRNUMBER = {1791881},
       DOI = {10.1016/S0294-1449(00)00118-9},
       URL = {https://doi.org/10.1016/S0294-1449(00)00118-9},
}

@article {Kra40,
    AUTHOR = {Kramers, H. A.},
     TITLE = {Brownian motion in a field of force and the diffusion model of
              chemical reactions},
   JOURNAL = {Physica},
  FJOURNAL = {Physica},
    VOLUME = {7},
      YEAR = {1940},
     PAGES = {284--304},
      ISSN = {0031-8914},
   MRCLASS = {79.0X},
  MRNUMBER = {2962},
MRREVIEWER = {N.\ Wiener},
}

@article {LJH21,
    AUTHOR = {Lee, Jae Yong and Jang, Jin Woo and Hwang, Hyung Ju},
     TITLE = {The model reduction of the
              {V}lasov-{P}oisson-{F}okker-{P}lanck system to the
              {P}oisson-{N}ernst-{P}lanck system {\it via} the deep neural
              network approach},
   JOURNAL = {ESAIM Math. Model. Numer. Anal.},
  FJOURNAL = {ESAIM. Mathematical Modelling and Numerical Analysis},
    VOLUME = {55},
      YEAR = {2021},
    NUMBER = {5},
     PAGES = {1803--1846},
      ISSN = {2822-7840,2804-7214},
   MRCLASS = {76P05 (35B40 35Q84 68T20 82C40)},
  MRNUMBER = {4313375},
       DOI = {10.1051/m2an/2021038},
       URL = {https://doi.org/10.1051/m2an/2021038},
}

@article {L19,
    AUTHOR = {Li, Dong},
     TITLE = {On {K}ato-{P}once and fractional {L}eibniz},
   JOURNAL = {Rev. Mat. Iberoam.},
  FJOURNAL = {Revista Matem\'atica Iberoamericana},
    VOLUME = {35},
      YEAR = {2019},
    NUMBER = {1},
     PAGES = {23--100},
      ISSN = {0213-2230,2235-0616},
   MRCLASS = {35R11 (35Q35 35Q86)},
  MRNUMBER = {3914540},
MRREVIEWER = {Vincenzo\ Ambrosio},
       DOI = {10.4171/rmi/1049},
       URL = {https://doi.org/10.4171/rmi/1049},
}

@article {Mio19,
    AUTHOR = {Miot, Evelyne},
     TITLE = {The gyrokinetic limit for the {V}lasov-{P}oisson system with a
              point charge},
   JOURNAL = {Nonlinearity},
  FJOURNAL = {Nonlinearity},
    VOLUME = {32},
      YEAR = {2019},
    NUMBER = {2},
     PAGES = {654--677},
      ISSN = {0951-7715,1361-6544},
   MRCLASS = {35Q83 (35A02)},
  MRNUMBER = {3903265},
       DOI = {10.1088/1361-6544/aaece7},
       URL = {https://doi.org/10.1088/1361-6544/aaece7},
}

@book {Nel67,
    AUTHOR = {Nelson, Edward},
     TITLE = {Dynamical theories of {B}rownian motion},
 PUBLISHER = {Princeton University Press, Princeton, NJ},
      YEAR = {1967},
     PAGES = {iii+142},
   MRCLASS = {60.62 (82.00)},
  MRNUMBER = {214150},
MRREVIEWER = {H.\ P.\ McKean, Jr.},
}

@article {NPS01,
    AUTHOR = {Nieto, Juan and Poupaud, Fr\'ed\'eric and Soler, Juan},
     TITLE = {High-field limit for the {V}lasov-{P}oisson-{F}okker-{P}lanck
              system},
   JOURNAL = {Arch. Ration. Mech. Anal.},
  FJOURNAL = {Archive for Rational Mechanics and Analysis},
    VOLUME = {158},
      YEAR = {2001},
    NUMBER = {1},
     PAGES = {29--59},
      ISSN = {0003-9527,1432-0673},
   MRCLASS = {82D10 (35B40 82C31)},
  MRNUMBER = {1834113},
MRREVIEWER = {Yan\ Guo},
       DOI = {10.1007/s002050100139},
       URL = {https://doi.org/10.1007/s002050100139},
}

@article {NRS22,
    AUTHOR = {Nguyen, Quoc-Hung and Rosenzweig, Matthew and Serfaty, Sylvia},
     TITLE = {Mean-field limits of {R}iesz-type singular flows},
   JOURNAL = {Ars Inven. Anal.},
  FJOURNAL = {Ars Inveniendi Analytica},
      YEAR = {2022},
     PAGES = {Paper No. 4, 45},
      ISSN = {2769-8505},
   MRCLASS = {70F10 (35Q70 82C22)},
  MRNUMBER = {4462479},
MRREVIEWER = {Rutwig\ Campoamor-Stursberg},
}

@article {PS17,
    AUTHOR = {Petrache, Mircea and Serfaty, Sylvia},
     TITLE = {Next order asymptotics and renormalized energy for {R}iesz
              interactions},
   JOURNAL = {J. Inst. Math. Jussieu},
  FJOURNAL = {Journal of the Institute of Mathematics of Jussieu. JIMJ.
              Journal de l'Institut de Math\'ematiques de Jussieu},
    VOLUME = {16},
      YEAR = {2017},
    NUMBER = {3},
     PAGES = {501--569},
      ISSN = {1474-7480,1475-3030},
   MRCLASS = {82B05 (15B52 82B21 82B26)},
  MRNUMBER = {3646281},
       DOI = {10.1017/S1474748015000201},
       URL = {https://doi.org/10.1017/S1474748015000201},
}

@article {Pou92,
    AUTHOR = {Poupaud, F.},
     TITLE = {Runaway phenomena and fluid approximation under high fields in
              semiconductor kinetic theory},
   JOURNAL = {Z. Angew. Math. Mech.},
  FJOURNAL = {Zeitschrift f\"ur Angewandte Mathematik und Mechanik},
    VOLUME = {72},
      YEAR = {1992},
    NUMBER = {8},
     PAGES = {359--372},
      ISSN = {0044-2267,1521-4001},
   MRCLASS = {82D20 (76P05 78A35 82C70)},
  MRNUMBER = {1178932},
MRREVIEWER = {Alexander\ Orlov},
       DOI = {10.1002/zamm.19920720813},
       URL = {https://doi.org/10.1002/zamm.19920720813},
}

@article {PS00,
    AUTHOR = {Poupaud, F. and Soler, J.},
     TITLE = {Parabolic limit and stability of the
              {V}lasov-{F}okker-{P}lanck system},
   JOURNAL = {Math. Models Methods Appl. Sci.},
  FJOURNAL = {Mathematical Models and Methods in Applied Sciences},
    VOLUME = {10},
      YEAR = {2000},
    NUMBER = {7},
     PAGES = {1027--1045},
      ISSN = {0218-2025,1793-6314},
   MRCLASS = {76X05 (45K05 82C21 82C22 82D10)},
  MRNUMBER = {1780148},
MRREVIEWER = {R.\ Glassey},
       DOI = {10.1142/S0218202500000525},
       URL = {https://doi.org/10.1142/S0218202500000525},
}

@article {SR00,
    AUTHOR = {Saint-Raymond, L.},
     TITLE = {The gyrokinetic approximation for the {V}lasov-{P}oisson
              system},
   JOURNAL = {Math. Models Methods Appl. Sci.},
  FJOURNAL = {Mathematical Models and Methods in Applied Sciences},
    VOLUME = {10},
      YEAR = {2000},
    NUMBER = {9},
     PAGES = {1305--1332},
      ISSN = {0218-2025,1793-6314},
   MRCLASS = {82D10 (35Q60)},
  MRNUMBER = {1796566},
MRREVIEWER = {Kamel\ Hamdache},
       DOI = {10.1142/S0218202500000641},
       URL = {https://doi.org/10.1142/S0218202500000641},
}

@article {SR04,
    AUTHOR = {Saint-Raymond, Laure},
     TITLE = {Convergence of solutions to the {B}oltzmann equation in the
              incompressible {E}uler limit},
   JOURNAL = {Arch. Ration. Mech. Anal.},
  FJOURNAL = {Archive for Rational Mechanics and Analysis},
    VOLUME = {166},
      YEAR = {2003},
    NUMBER = {1},
     PAGES = {47--80},
      ISSN = {0003-9527,1432-0673},
   MRCLASS = {35F20 (35Q30 45K05 76A02 76B99 82C40)},
  MRNUMBER = {1952079},
MRREVIEWER = {C\'edric\ Villani},
       DOI = {10.1007/s00205-002-0228-3},
       URL = {https://doi.org/10.1007/s00205-002-0228-3},
}

@article {Ser20,
    AUTHOR = {Serfaty, Sylvia},
     TITLE = {Mean field limit for {C}oulomb-type flows},
      NOTE = {With an appendix by Mitia Duerinckx and Serfaty},
   JOURNAL = {Duke Math. J.},
  FJOURNAL = {Duke Mathematical Journal},
    VOLUME = {169},
      YEAR = {2020},
    NUMBER = {15},
     PAGES = {2887--2935},
      ISSN = {0012-7094,1547-7398},
   MRCLASS = {35Q82 (82C22)},
  MRNUMBER = {4158670},
       DOI = {10.1215/00127094-2020-0019},
       URL = {https://doi.org/10.1215/00127094-2020-0019},
}

@incollection {Vas08,
    AUTHOR = {Vasseur, Alexis F.},
     TITLE = {Recent results on hydrodynamic limits},
 BOOKTITLE = {Handbook of differential equations: evolutionary equations.
              {V}ol. {IV}},
    SERIES = {Handb. Differ. Equ.},
     PAGES = {323--376},
 PUBLISHER = {Elsevier/North-Holland, Amsterdam},
      YEAR = {2008},
      ISBN = {978-0-444-53034-9},
   MRCLASS = {76D05 (35Q35 76A02 82C40)},
  MRNUMBER = {2508169},
MRREVIEWER = {J\"urgen\ Socolowsky},
       DOI = {10.1016/S1874-5717(08)00007-8},
       URL = {https://doi.org/10.1016/S1874-5717(08)00007-8},
}

@article {WLL15,
    AUTHOR = {Wu, Hao and Lin, Tai-Chia and Liu, Chun},
     TITLE = {Diffusion limit of kinetic equations for multiple species
              charged particles},
   JOURNAL = {Arch. Ration. Mech. Anal.},
  FJOURNAL = {Archive for Rational Mechanics and Analysis},
    VOLUME = {215},
      YEAR = {2015},
    NUMBER = {2},
     PAGES = {419--441},
      ISSN = {0003-9527,1432-0673},
   MRCLASS = {82D10 (35Q83 82C31 92C37)},
  MRNUMBER = {3294407},
       DOI = {10.1007/s00205-014-0784-3},
       URL = {https://doi.org/10.1007/s00205-014-0784-3},
}

@article {Zho22,
    AUTHOR = {Zhong, Mingying},
     TITLE = {Diffusion limit and the optimal convergence rate of the
              {V}lasov-{P}oisson-{F}okker-{P}lanck system},
   JOURNAL = {Kinet. Relat. Models},
  FJOURNAL = {Kinetic and Related Models},
    VOLUME = {15},
      YEAR = {2022},
    NUMBER = {1},
     PAGES = {1--26},
      ISSN = {1937-5093,1937-5077},
   MRCLASS = {76P05 (35Q60 82C40 82D10)},
  MRNUMBER = {4389615},
MRREVIEWER = {Marzia\ Bisi},
       DOI = {10.3934/krm.2021041},
       URL = {https://doi.org/10.3934/krm.2021041},
}

@article {CJe24,
    AUTHOR = {Choi, Young-Pil and Jeong, In-Jee},
     TITLE = {Well-posedness and singularity formation for {V}lasov-{R}iesz
              system},
   JOURNAL = {Kinet. Relat. Models},
  FJOURNAL = {Kinetic and Related Models},
    VOLUME = {17},
      YEAR = {2024},
    NUMBER = {3},
     PAGES = {489--513},
      ISSN = {1937-5093,1937-5077},
   MRCLASS = {35Q83 (35Q35 82C40)},
  MRNUMBER = {4739728},
MRREVIEWER = {Arnaud\ Triay},
       DOI = {10.3934/krm.2023030},
       URL = {https://doi.org/10.3934/krm.2023030},
}

@article {Cha23,
    AUTHOR = {Chaub, Thomas},
     TITLE = {Local well-posedness for a class of singular {V}lasov
              equations},
   JOURNAL = {Kinet. Relat. Models},
  FJOURNAL = {Kinetic and Related Models},
    VOLUME = {16},
      YEAR = {2023},
    NUMBER = {2},
     PAGES = {187--206},
      ISSN = {1937-5093,1937-5077},
   MRCLASS = {35Q83 (35A01 35A02)},
  MRNUMBER = {4520381},
       DOI = {10.3934/krm.2022027},
       URL = {https://doi.org/10.3934/krm.2022027},
}

@article {IKT13,
    AUTHOR = {Inci, H. and Kappeler, T. and Topalov, P.},
     TITLE = {On the regularity of the composition of diffeomorphisms},
   JOURNAL = {Mem. Amer. Math. Soc.},
  FJOURNAL = {Memoirs of the American Mathematical Society},
    VOLUME = {226},
      YEAR = {2013},
    NUMBER = {1062},
     PAGES = {vi+60},
      ISSN = {0065-9266,1947-6221},
      ISBN = {978-0-8218-8741-7},
   MRCLASS = {58-02 (35B65 35R01 58D05)},
  MRNUMBER = {3135704},
MRREVIEWER = {Nikos\ Labropoulos},
       DOI = {10.1090/S0065-9266-2013-00676-4},
       URL = {https://doi.org/10.1090/S0065-9266-2013-00676-4},
}

\end{document}